\documentclass[english,a4paper,11pt]{amsart}
\RequirePackage[T1]{fontenc}
\RequirePackage{amsthm,amsmath}
\RequirePackage[numbers,sort&compress]{natbib}
\usepackage{amsthm,amsmath}
\usepackage{hyperref}
\usepackage[a4paper]{geometry}

\allowdisplaybreaks[4]

\theoremstyle{plain}
\newtheorem{theorem}{Theorem}[section]
\newtheorem{lemma}[theorem]{Lemma}
\newtheorem{corollary}[theorem]{Corollary}
\newtheorem{proposition}[theorem]{Proposition}
\newtheorem{tlemma}[theorem]{Technical Lemma}
\newtheorem{tproposition}[theorem]{Technical Proposition}
\newtheorem{claim}[theorem]{Claim}

\theoremstyle{definition}
\newtheorem{definition}[theorem]{Definition}

\newtheorem{hypo}[theorem]{Hypothesis}

\theoremstyle{remark}
\newtheorem{remark}[theorem]{Remark}

\numberwithin{equation}{section}

\newcounter{lil1}
\newenvironment{step}
{\begin{list} { \bf Step (\Roman{lil1})}
{ \usecounter{lil1}
\setlength{\leftmargin}{0.0cm}
\setlength{\topsep}{0.2cm}
\setlength{\itemsep}{0.0cm}
\setlength{\parsep}{0.1cm}
\setlength{\itemindent}{0.8cm}
\setlength{\parskip}{0.0cm}}}
{\end{list}}

\newcounter{lil33}
\newenvironment{stepinner}
{\begin{list} { \bf Step (\alph{lil33})}
{ \usecounter{lil33}
\setlength{\leftmargin}{0.0cm}
\setlength{\topsep}{0.2cm}
\setlength{\itemsep}{0.0cm}
\setlength{\parsep}{0.1cm}
\setlength{\itemindent}{0.8cm}
\setlength{\parskip}{0.0cm}}}
{\end{list}}

\newcounter{lil1q}
\newenvironment{steps}
{\begin{list} { \bf Step (\Roman{lil1q}).}
{ \usecounter{lil1q}
\setlength{\leftmargin}{0.0cm}
\setlength{\topsep}{0.2cm}
\setlength{\itemsep}{0.0cm}
\setlength{\parsep}{0.1cm}
\setlength{\itemindent}{0.8cm}
\setlength{\parskip}{0.0cm}}}
{\end{list}}

\newcommand{\eps}{\varepsilon}

\newcommand{\MA}{\mathfrak{A}}

\usepackage{dsfont}
\usepackage{stmaryrd}
\usepackage{color}

\newcommand{\ul}{\underline}
\newcommand{\ol}{\overline}

\newcommand{\Reins}{R_0}
\newcommand{\Rzwei}{ R_1}
\newcommand{\Rdrei}{R_2}

\newcommand{\subX}{\mathcal{X}_\mathfrak{A}(\Reins,\Rzwei,\Rdrei)}

\newcommand{\stoppt}{\phi_\kk(h(\eta,\xi,t))\;}
\newcommand{\stopp}{\phi_\kk(h(\eta,\xi,s))\;}

\newcommand{\stose}{\phi_\kk(h(\eta_1,\xi_1,s))\;}
\newcommand{\stosz}{\phi_\kk(h(\eta_2,\xi_2,s))\;}

\newcommand{\vardelta}{\nu}
\newcommand{\sig}{1}

\newcommand{\pp}{p+1}

\newcommand{\TInt}{T}

\newcommand{\kk}{\kappa}
\newcommand{\buk}{\bar u_\kk}
\newcommand{\bvk}{\bar v_\kk}
\newcommand{\uk}{u_\kk}
\newcommand{\vk}{v_\kk}

\newcommand{\wk}{v_\kk}
\newcommand{\textkappa}{\mu}

\newcommand{\baray}{\begin{array}{rcl}}
\newcommand{\earay}{\end{array}}
\newcommand{\barray}{\begin{array}{rcl}}
\newcommand{\earray}{\end{array}}

\newcommand\dela[1]{}

\newcommand{\bcase}{\begin{cases}}
\newcommand{\ecase}{\end{cases}}

\newcommand{\ns}{{n^\ast}}

 \newcommand{\LLm}{{2}}
\newcommand\BY{\mathbb{Y}}
\newcommand\BZ{\mathbb{Z}}
\newcommand\BH{\mathbb{H_\rho}}
\newcommand\BX{\mathbb{Z}}

\newcommand{\ppp}{m}
\newcommand{\DeltaA}{A}
\newcommand\del[1]{}
\newcommand{\pst}{{p^\star}}

\del{

\newcommand\del[1]{}
}

\def\eps{\varepsilon}

\newcommand{\Law}{\mbox{Law}}

\newcommand{\lqq}{\lefteqn}

\newcommand{\la}{\langle}
\newcommand{\ra}{\rangle}

\newcommand{\LL}{{\rm I \kern -0.2em L}}

\newcommand{\ep} {\varepsilon }

\newcommand{\be} {\begin{enumerate} }
\newcommand{\ee} {\end{enumerate} }

\newcommand{\CO}{{{ \mathcal O }}}

\newcommand{\CH}{{{ \mathcal H }}}
\newcommand{\CS}{{{ \mathcal S }}}
\newcommand{\CG}{{{ \mathcal G }}}

\newcommand{\CB}{{{ \mathcal B }}}

\newcommand{\CM}{{{ \mathcal M }}}

\newcommand{\BF}{{{ \mathbb{F} }}}
\newcommand{\CF}{{{ \mathcal F }}}

\newcommand{\CN}{{{ \mathcal N }}}

\newcommand{\WW}{{\mathbb{W}}}

\newcommand{\NN}{\mathbb{N}}

\newcommand{\PP}{{\mathbb{P}}}

\newcommand{\EE}{ \mathbb{E} }

\newcommand{\DEQS}{\begin{eqnarray*} }
\newcommand{\EEQS}{\end{eqnarray*} }
\newcommand{\DEQSZ}{\begin{eqnarray} }
\newcommand{\EEQSZ}{\end{eqnarray} }
\newcommand{\DEQ}{\begin{eqnarray}}
\newcommand{\EEQ}{\end{eqnarray}}

\newcommand{\Dcal} {{\mathcal D}}

\newcommand{\Fcal} {{\mathcal F}}
\newcommand{\Gcal} {{\mathcal G}}
\newcommand{\Hcal} {{\mathcal H}}

\newcommand{\Mcal} {{\mathcal M}}

\newcommand{\Ocal} {{\mathcal O}}

\newcommand{\Vcal} {{\mathcal V}}

\newcommand{\Xcal} {{\mathcal X}}

\newcommand{\Afrak} {{\mathfrak A}}

\newcommand{\N}{\mathbb{N}}

\renewcommand{\P}{\mathbb{P}}
\newcommand{\Eb}{\mathbb{E}}

\newcommand{\X}{\mathbb{X}}
\newcommand{\D}{\mathbb{D}}

\usepackage{amsmath}
\usepackage{amssymb, color}
\usepackage{latexsym}

\usepackage{mathrsfs}

\usepackage{enumerate}

\usepackage[T1]{fontenc}
\usepackage{lipsum}
\usepackage{amsfonts}
\usepackage{graphicx}
\usepackage{epstopdf}
\usepackage{algorithmic}
\usepackage{bbm}

\usepackage{url}

\AtBeginDocument{  }

\begin{document}

\title[The stochastic Klausmeier system]{The stochastic Klausmeier system and a stochastic Schauder-Tychonoff type theorem}
\author[E. Hausenblas]{Erika Hausenblas}
\address{Montanuniversit\"{a}t Leoben\\
Department Mathematik und Informationstechnologie\\
Franz Josef Stra{\ss}e 18\\
8700 Leoben\\
Austria}
\email{erika.hausenblas@unileoben.ac.at}

\author[J. M. T\"olle]{Jonas M. T\"{o}lle}
\address{Aalto University\\
Department of Mathematics and Systems Analysis\\
PO Box 11100 (Otakaari 1, Espoo)\\
00076 Aalto\\
Finland}
\email{jonas.tolle@aalto.fi}
\date{\today}
\begin{abstract}
On the one hand, we investigate the existence and pathwise uniqueness of a nonnegative martingale solution to the stochastic evolution system of nonlinear advection-diffusion equations proposed by Klausmeier with Gaussian multiplicative noise. On the other hand, we present and verify a general stochastic version of the Schauder-Tychonoff fixed point theorem, as its application is an essential step for showing existence of the solution to the stochastic Klausmeier system. The analysis of the system is based both on variational and semigroup techniques. We also discuss additional regularity properties of the solution.
\end{abstract}


\thanks{E.H. acknowledges partial funding by FWF (Austrian Science Foundation) grant P 28010 ``Mathematische Analyse von Fl\"ussigkeitskristallen mit stochastischer St\"orung''.
J.M.T. acknowledges support by the Academy of Finland and the European Research Council (ERC) under the European Union's Horizon 2020 research and innovation programme (grant agreements no. 741487 and no. 818437).\\
This work is licensed under the Creative Commons Attribution 4.0 International License. To view a copy of this license, visit \url{http://creativecommons.org/licenses/by/4.0/} or send a letter to Creative Commons, PO Box 1866, Mountain View, CA 94042, USA. \includegraphics[height=1em]{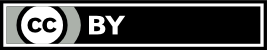}
Original work published in Potential Analysis (2023+), 62 pp., \url{https://doi.org/10.1007/s11118-023-10107-3}.}

\maketitle

{\small \tableofcontents}

\section{Introduction}

Pattern formation at the ecosystem level has recently gained a lot of attention in spatial ecology and its mathematical modeling.
Theoretical models are a widely used tool for studying e.g.\ banded vegetation patterns. One important model is the system of advection-diffusion equations proposed by Klausmeier \cite{klausmeier}.
This model for vegetation dynamics in semi-deserted areas is based on the ``water redistribution hypothesis'', using the idea that rain water in dry regions is eventually infiltrated into the ground. Water falling down onto bare ground mostly runs off downhill toward the next patch of vegetation which provides a better infiltration capacity. The soil in such regions of the world as Australia, Africa, and Southwestern North America is prone to nonlocality of water uptake due to the semi-arid environment. Studies of the properties of the system
and further developments can be found in e.g. \cite{sherratt,nadia,klaus1,Sherratt:2005,Sherratt:2010,Sherratt:2011}.

The Klausmeier system is a generalization of the so-called Gray-Scott system \cite{GrayScott_original} (see also \cite{segel,skt} for earlier accounts employing similar models) which already exhibits effects similar to Turing patterns \cite{turing,kepper,Maini,Dillon}, see for instance the discussion in \cite{vanderStelt:2012gs}. We refer to \cite{murray1,murray2,levin2} for further reading on pattern formation in biology. A discussion of reaction-diffusion type equations with motivation from biology can be found in \cite{Perthame}.

The underlying mathematics of this model is given by a pair of solutions $(u,v)$ to a partial differential equation system coupled by a nonlinearity. The function $u$ represents the surface water content  and $v$ represents the biomass density of the plants.
In order to model the spread of water on a terrain without a specific preference for the direction in which the water flows, the original models were extended by replacing the diffusion operator by a nonlinear porous media operator, which represents the situation that the ground is partially filled by interconnected pores conveying fluid under an applied pressure gradient.

To this end, let $\CO\subset \mathbb{R}^d$ be a bounded domain, $d=1,2,3$, having $C^2$-boundary or $\CO=[0,1]^d$.
Consider the following problem
\begin{align} \label{equ1s}
\left\{ \barray
 \dot {u}(t)  &=& r_u   \Delta   u^{[\gamma]}(t)
 - \chi u(t)\,v^2(t) +k-fu(t) , \quad  t> 0,
\quad u(0)= u_0 ,  
\\
\dot{v}(t) &=& r_v \Delta v(t) + u(t)\,v^2(t) -gv(t) ,\quad   t> 0,
\quad
{v}(0) = v_0 ,
\earray\right.
\end{align} 
with Neumann (or periodic when $\CO$ is a rectangular domain) boundary conditions and initial conditions $u(0)=u_0$
and $v(0)=v_0$. Here, $z^{[\gamma]}:=\vert z\vert^{\gamma-1}z$, $\gamma>1$, $z\in\mathbb{R}$, and further,
$r_u$, $r_v$, $\chi$, $k$, $f$ and $g$ denote positive constants.

The deterministic or macroscopic model is derived from the limiting behavior of interacting diffusions --- the so-called microscopic model, see \cite{kotelenez}. When applying the strong law of large numbers and passing from the microscopic to the macroscopic equation, one is neglecting the random fluctuations. In order to get a more realistic model, it is necessary to add noise, which represents the randomness of the natural environment or the fluctuation of parameters in the model. The introduction to stochasticity to ecological models is supported by the arguments of \cite{Hastings}.
Due to the Wong-Zakai principle, this leads to the representation of the noise as a Gaussian stochastic perturbation with Stratonovich increment, see \cite[Section 3.4]{Flandoli} and \cite{MaZhu2019,WongZakai}. One important consequence is the preservation of energy in the noisy system.

In practice, we are investigating the system \eqref{equ1s}
driven by a multiplicative infinite dimensional Wiener process. Under suitable regularity assumptions on the initial data $(u_0,v_0)$, on $\gamma$ and on the perturbation by noise which we shall specify later, we find that there exists a nonnegative (martingale) solution to the system \eqref{equ1n} in dimensions $d=1,2,3$, see our main result Theorem \ref{mainresult}. More precisely, we prove existence of nonnegative solutions to the noisy systems \eqref{eq:uIto}--\eqref{eq:vIto} (It\^o noise), \eqref{equ1ss}--\eqref{eqv1ss} (Stratonovich noise) respectively. For $d=1$, we show that the solution to \eqref{eq:uIto}--\eqref{eq:vIto} is pathwise unique and that there exists a strong solution in the stochastic sense.
Thus, we are actually considering the following system for $\gamma>1$
\begin{align} \label{equ1n}
\left\{ \barray
 \dot {u}(t)  &=& r_u   \Delta   u^{[\gamma]}(t)  - \chi u(t)\,v^2(t) +u(t)\, \Xi_1
 (t), \quad  t> 0,\quad u(0)=u_0,
\\  \dot{v}(t) &=& r_v \Delta v(t) + u(t)\,v^2(t)+ v(t)\, \Xi_2(t)
,\quad   t> 0,\quad v(0)=v_0.
\earray\right.
\end{align} 
Here, $\Xi_1$ and $\Xi_2$ denote independent random Gaussian noises specified to be certain Banach space valued linear Stratonovich Wiener noises later on. Similar stochastic equations with L\'evy noise have been studied e.g. in \cite{DebusscheHoegeleImkeller2013,reactdiff}. We note that due to the coupling and the nonlinearity, the solutions are not expected to be stochastically independent.

Nonlinear diffusion systems perturbed by noise exhibit certain improved well-posedness and regularity properties when compared to their deterministic counterparts,
see, among others, \cite{Bianchi_etal, CrauelFlandoli, Flandoli}. Another feature of linear multiplicative noise is that it preserves the nonnegativity of the initial condition \cite{TessitoreZabczyk1998,BDPR3}. Some applications may demand more complicated noises with nonlinear structure or coupling, however, we choose a linear dependence on the noise as a first step toward more general models.

The challenging problem in the system given by \eqref{equ1s}, respective the noisy system \eqref{equ1n}, is the nonlinearity appearing once with a negative sign and once with a positive sign.
Choosing different speeds of diffusion seems rather natural, as it well known that the characteristic pattern
formation may not take place if $r_u=r_v$, cf. \cite{grayscott}. The stochastic perturbation, however, does not restrict pattern formation \cite{Cao,Woolley}, noting that our choice of noise exhibits rather small intensity and damped high frequency modes.
On the other hand, the nonlinearity is not of variational structure, such that energy methods are not available for the analysis, and neither the maximum principle nor  Gronwall type arguments work.
More direct deterministic (pathwise) methods may also fail in general as the equation for $v$ is non-monotone.

Another difficulty is posed by the nonlinear porous media diffusion operator in the equation for $u$. Typically, it is studied with the help of variational monotonicity methods, see e.g. \cite{weiroeckner,BDPR2016}. To the nonlinear diffusion, neither semigroup methods nor Galerkin approximation can be applied directly without greater effort. Other approaches are given for instance in \cite{Gess2017,vazquez2007porous,Otto,FehrmanGess2021}, to mention a few.

Our motivation to prove a probabilistic Schauder-Tychonoff type fixed point theorem originates from the aim to show existence of a solution to the stochastic counterpart \eqref{equ1n} of the system \eqref{equ1s}. Our result is for instance also applied in \cite{HausenblasPanda1,HausenblasPanda2}.
As standard methods for showing existence and uniqueness of solutions to stochastic partial differential equations cannot be applied directly, we perform a fixed point iteration using a mix of so-called variational and semigroup methods with nonlinear perturbation. Here, a precise analysis of regularity properties of nonnegative solutions to regularized and localized
subsystems with truncated nonlinearities needs to be conducted. Together with the continuous dependence of each subsystem on the other one, we obtain the fixed point, locally in time, by weak compactness and an appropriate choice of energy spaces. The result is completed by a stopping time gluing argument.

Apart from the probabilistic structure,
the main novelty is that we construct the fixed point in the nonlinearity and not in the noise coefficient. This works well because our system is coupled precisely in the nonlinearity. The main task consists then in the analysis of regularity and invariance properties of a regularized system with ``frozen'' nonlinearity (not being fully linearized, however, as the porous media operator remains).

It may be possible to apply this method in the future to other nonlinear systems, as for example, systems with nonlinear convection-terms as systems with transport or Navier-Stokes type systems. It is certainly possible to apply the method to linear cross diffusion systems.
See \cite{boundednessbyentropymethod} for a recent work proving existence of martingale solutions to stochastic cross diffusion systems, however, their approach relies on other methods that do not cover the porous media case. See \cite{MASLOWSKI2003} for a previous work using the classical Schauder theorem for stochastic evolution equations with fractional Gaussian noise.

Given higher regularity of the initial data, we also show the pathwise uniqueness of the solution to \eqref{equ1n} for spatial dimension $d=1$.
As a consequence of the celebrated result by Yamada and Watanabe \cite{YW1,WY2}, we obtain existence and uniqueness of a strong solution, see Corollary \ref{cor:strong}. We refer to \cite{Schmalfuss1997,Capinski1993,CP1997,BM2013} for previous works that employ a similar strategy for proving the existence of a unique strong solution.

\subsection*{Structure of the paper}

The stochastic Schauder-Tychonoff type Theorem \ref{ther_main} is presented in Section \ref{schauder}. In the subsequent section, i.e. Section \ref{klausmeier}, we apply this fixed point theorem to show the existence of a martingale solution to the stochastic counterpart \eqref{equ1n} of the system \eqref{equ1s}. Section \ref{klausmeier} contains our main result Theorem \ref{mainresult}. Section \ref{sec:proof_of_schauder} is devoted to the proof of Theorem \ref{ther_main}.
In Section \ref{technics}, we prove several technical propositions that are need for the main result in Section \ref{klausmeier}. Section \ref{pathwise} contains the proof of pathwise uniqueness, that is, Theorem \ref{thm_path_uniq}. Some auxiliary results are collected in Appendices \ref{app:A}, \ref{dbouley-space} and \ref{sec:BDG}.

\section{The stochastic Schauder-Tychonoff type theorem}\label{schauder}

Let us fix some notation.
Let
$U$ be a Banach space. Let $\Ocal\subset \mathbb{R}^d$ be an open domain, $d\ge 1$. Let $\X\subset\{\eta:[0,T]\to E\subset\Dcal^{\prime}(\Ocal)\}$
be a Banach function space\footnote{Here, $\Dcal^\prime(\Ocal)$ denotes the space of Schwartz distributions on $\Ocal$, that is, the topological dual space of smooth functions with compact support $\Dcal(\Ocal)=C_0^\infty(\Ocal)$.}, let $\X^{\prime}\subset\{\eta:[0,T]\to E\subset\Dcal^{\prime}(\Ocal)\}$
be a reflexive Banach function space embedded compactly and densely into $\X$. In both cases, the trajectories take values in a Banach function space $E$ over the spatial domain $\Ocal$, where we assume that $E$ has the UMD property, see \cite{vanneerven}.
Let $\Afrak=(\Omega,\Fcal,\mathbb{F},\P)$ be a filtered probability space
with filtration $\mathbb{F}=(\Fcal_{t})_{t\in [0,T]}$ satisfying the usual conditions.
Let $H$ be a separable Hilbert space and $(W(t))_{t\in [0,T]}$ be a Wiener process\footnote{That is, a $Q$-Wiener process, see e.g. \cite{DaPrZa:2nd} for this notion.} in $H$ with a linear, nonnegative definite, symmetric trace class covariance
operator $Q:H\to H$ such that $W$ has the representation
$$
W(t)=\sum_{i\in\mathbb{I}} Q^\frac 12 \psi_i\beta_i(t),\quad t\in [0,T],
$$
where $\{\psi_i:i\in \mathbb{I}\}$ is a complete orthonormal system in $H$, $\mathbb{I}$ a suitably chosen countable index set, and $\{\beta_i:i\in\mathbb{I}\}$ a family of independent real-valued standard Brownian motions on $[0,T]$ modeled in $\Afrak=(\Omega,\Fcal,\mathbb{F},\P)$. Due to \cite[Proposition 4.7, p. 85]{DaPrZa:2nd}, this representation does not pose a restriction.

For $m\ge1$, define the collection of processes
\begin{equation}\label{eq:MMdef}
\begin{aligned}  \Mcal_{\Afrak}^{m}(\X)
:= & \Big\{ \xi:\Omega\times[0,T]\to E\;\colon\;\\
&\qquad\text{\ensuremath{\xi} is \ensuremath{\mathbb{F}}-progressively measurable}\;\text{and}\;\Eb\vert\xi\vert_{\X}^{m}<\infty\Big\}
\end{aligned}
\end{equation}
equipped with the semi-norm
\[
\vert\xi\vert_{\Mcal_{\Afrak}^{m}(\X)}:=(\Eb\vert\xi\vert_{\X}^{m})^{1/m},\quad\xi\in\Mcal_{\Afrak}^{m}(\X).
\]

For fixed $\Afrak$, $W$, $m>1$, we define the operator $\Vcal=\Vcal_{\Afrak,W}:\Mcal_{\Afrak}^{m}(\X)\times L^m(\Omega,\Fcal_0,\P;E)\to\Mcal_{\Afrak}^{m}(\X)$
for $\xi\in\Mcal_{\MA}^m(\X)$ via
$$\Vcal(\xi):=\Vcal(\xi,w_0):=w,$$
where $w$ is the solution to the following It\^o stochastic partial differential equation (SPDE)
\begin{align} \label{spdes}
dw(t) =&\left(\DeltaA w(t)+ F(\xi,t)\right)\, dt +\Sigma(w(t))\,dW(t),\quad w(0)=w_0\in E.
\end{align} 
For convenience, we drop the dependence on the initial datum $w_0$ in the notation $\Vcal(\xi)$.
Here, we implicitly assume that \eqref{spdes} is well-posed and a unique strong solution (in the stochastic sense) $w\in\Mcal_{\MA}^m(\X)$ exists for $\xi\in\Mcal_{\MA}^m(\X)$.
Here, we shall also assume that $\DeltaA:D(\DeltaA)\subset E\to E$ is a possibly nonlinear and measurable (single-valued) operator and $F:\X\times [0,T]\to E$ a (strongly) measurable map such that
\[\PP\left(\int_0^T\left(\left\vert\DeltaA w (s)\right\vert_E+\left\vert F(\xi,s)\right\vert_E\right)\,ds<\infty\right)=1,\]
and assume that $\Sigma:E \to \gamma(H,E)$ is strongly measurable. Here, $\gamma(H,E)$ denotes the space of $\gamma$-radonifying operators from $H$ to $E$,
as defined in the beginning of Section \ref{technics}, which coincides with the space of Hilbert-Schmidt operators $L_{\textup{HS}}(H,E)$ if $E$ is a
separable Hilbert space.

We are ready to formulate our main tool for proving the existence of nonnegative martingale solutions to
\eqref{equ1n}, that is, a stochastic variant of the (deterministic) Schauder-Tychonoff fixed point theorem(s) from \cite[§ 6--7]{granas}.

\begin{theorem}\label{ther_main}
Let $H$ be a Hilbert space, $Q:H\to H$ such that $Q$ is linear, symmetric, nonnegative definite and of trace class, let $U$ be a Banach space, and let us assume that we have a compact and dense embedding $\X^{\prime}\hookrightarrow\X$
as above. Let $m>1$. Suppose that
for any filtered probability space $\Afrak=(\Omega,\Fcal,\mathbb{F},\P)$
and for any $Q$-Wiener process $W$ with values in $H$ that is modeled
on $\Afrak$ the following holds.

Suppose that there exist constants $R_1,\ldots,R_K>0$, $K\in\N$, continuous functions $\Psi_i:\X\to [0,\infty)$, $1\le i\le K$, measurable functions $\Theta_i:\X\to[0,\infty]$, $1\le i\le K$ with closed sublevel sets $\Theta_i^{-1}([0,\alpha])$, $\alpha\ge 0$, $1\le i\le K$, and
a nonempty, sequentially weak$^\ast$-closed, measurable and bounded subset\footnote{Here, the notation $\Xcal(\Afrak)$ means that $\text{Law}(\xi)=\text{Law}(\tilde{\xi})$
on $\X$ for $\xi\in\Xcal(\Afrak)$ and $\tilde{\xi}\in\Mcal_{\tilde{{\Afrak}}}^{m}$
implies $\tilde{\xi}\in\Xcal(\tilde{\Afrak})$.} $\Xcal_{R_1,\ldots,R_K}(\Afrak)$ of $\Mcal_{\MA}^m(\X)$ such that:
\begin{enumerate}[(a)]
\item $\mathbb{E}[\Psi_i(\xi)]\le R_i$, for every $\xi\in \Xcal_{R_1,\ldots,R_K}(\Afrak)$ and every $1\le i\le K$,
\item $\mathbb{P}(\{\Theta_i(\xi)<\infty,\;\xi\in\Xcal_{R_1,\ldots,R_K}(\Afrak)\})=1$, for every $1\le i\le K$.
\end{enumerate}
Let us assume that the operator $\Vcal_{\Afrak,W}$, defined by \eqref{spdes}, restricted to $\Xcal_{R_1,\ldots,R_K}(\Afrak)$
satisfies the following
properties:
\begin{enumerate}[(i)]
\item the operator $\Vcal_{\Afrak,W}$ is well-defined on $\Xcal_{R_1,\ldots,R_K}(\Afrak)$ for all choices of $R_i>0$, $1\le i\le K$,
\item there exist constants $R^0_i>0$, $1\le i\le K$ such that 
\[\Vcal_{\Afrak,W}(\Xcal_{R_1,\ldots,R_K}(\Afrak))\subset \Xcal_{R_1,\ldots,R_K}(\Afrak)\]
for all $R_i\ge R^0_i$ and all $1\le i\le K$,
\item for all choices of $R_i>0$, $1\le i\le K$, the restriction $\Vcal_{\Afrak,W}\big\vert_{\Xcal_{R_1,\ldots,R_K}(\Afrak)}$ is
uniformly continuous on bounded subsets w.r.t. the strong topology of $\Mcal_{\Afrak}^{m}(\X)$,
\item there exist constants $R>0$, $m_0\ge m$ such that\footnote{Note that we can relax assumption (iv) to $m_0>1$ if we assume additionally that there exist constants $R>0$, $m_1 > m$ such that $\mathbb{E}\left[\vert\mathcal{V}_{\mathfrak{A},W}(\xi)\vert_{\mathbb{X}}^{m_1}\right]\le R$ for every $\xi\in\mathcal{X}_{R_1,\ldots,R_K}(\mathfrak{A})$ for all $R_i>0$ and all $1\le i\le K$.}
\[
\Eb\left[\vert\Vcal_{\Afrak,W}(\xi)\vert_{\X^{\prime}}^{m_0}\right]\le R\quad\text{for every}\quad \xi\in\Xcal_{R_1,\ldots,R_K}(\Afrak),
\]
for all $R_i>0$ and all $1\le i\le K$.
\item for all $R_i>0$ and all $1\le i\le K$, $\Vcal_{\Afrak,W}(\Xcal_{R_1,\ldots,R_K}(\Afrak))\subset\D([0,T];U)$ $\P$-a.s.\footnote{Here, $\mathbb{D}([0,T];U)$ denotes the
Skorokhod space of c\`adl\`ag paths in $U$ endowed with the Skorokhod $J_1$-topolgy, see \cite[Appendix A2]{Kallenberg}.}
\end{enumerate}
Then, there exists a filtered probability space $\Afrak^\ast=(\Omega^\ast,\Fcal^\ast,\mathbb{F}^\ast,\P^\ast)$
(that satisfies the usual
conditions)
together with a $Q$-Wiener process $W^\ast$ modeled on $\Afrak^\ast$
and an element $w^\ast\in\Mcal_{\Afrak^\ast}^{m}(\X)$ such that
for all $t\in[0,T]$, $\P^\ast$-a.s.
\[
\Vcal_{\Afrak^\ast,W^\ast}(w^\ast,w^\ast_0)(t)=w^\ast(t)
\]
for any initial datum $w_0\in L^m(\Omega,\Fcal_0,\P;E)$,
where $w^\ast_0\in L^m(\Omega^\ast,\Fcal_0^\ast,\P^\ast;E)$ satisfies $\operatorname{Law}(w^\ast_0)=\operatorname{Law}(w_0)$.
\end{theorem}

The proof of Theorem \ref{ther_main} is postponed to Section \ref{sec:proof_of_schauder}. We note that
by construction, we get that $w^\ast$ solves
 \begin{align*} 
dw^\ast(t) =&\left(\DeltaA w^\ast(t)+ F(w^\ast,t)\right)\, dt +\Sigma(w^\ast(t))\,dW^\ast(t),\quad w^\ast(0)=w_0.
\end{align*} 
on $\Afrak^\ast$.

\section{Existence of a solution to the stochastic Klausmeier system}\label{klausmeier}

In this section, we shall prove the existence of a nonnegative solution to the stochastic Klausmeier system. First, we will introduce some notation and the definition of a (martingale) solution.
After fixing the function spaces and the main hypotheses on the parameters of those, we present our main result Theorem \ref{mainresult}. The following proof consists mainly of a verification of the conditions for Theorem \ref{ther_main} and a stopping time localization procedure. As pointed out before, we are using compactness arguments to show the existence, which leads to the loss of the initial stochastic basis.

Let $H_1$ and $H_2$ be a pair of separable Hilbert spaces,
let $\mathfrak{A}=(\Omega,\CF,\mathbb{F},\PP)$ be a filtered probability space and let $W_1$, $W_2$ be a pair of independent Wiener process modeled on $\mathfrak{A}$ taking values in $H_1$ and $H_2$, respectively, with covariance operators $Q_1$ and $Q_2$, respectively.
We are interested in the solution to the following reduced Klausmeier system for $x\in \CO$ and $t>0$,
\begin{equation}\label{equ1ss}
 d {u}(t,x)  =  ( r_u\Delta   u^{[\gamma]}(t,x) -\chi  u(t,x)\,v^2(t,x))\, dt +\sigma_1 u(t,x)\circ  dW_1(t,x),
\end{equation}
and,
\begin{equation}\label{eqv1ss}
d{v}(t,x) = ( r_v\Delta v(t,x)  + u(t,x)\,v^2(t,x))\, dt  +\sigma_2 v(t,x)\circ d W_2(t,x),
\end{equation}
with Neumann (or periodic if $\Ocal=[0,1]^d$) boundary conditions and initial conditions $u(0)=u_0$
and $v(0)=v_0$. Let $r_u,r_v,\chi>0$ be positive constants. Here, we use the abbreviation $x^{[\gamma]}:=\vert x\vert^{\gamma-1}x$ for $\gamma>1$.
The hypotheses on the linear
noise coefficient maps $\sigma_1,\sigma_2$ are specified below.

Due to the nonlinear porous media term, we do neither use solutions in the strong stochastic sense, nor mild solutions, that is, solutions in the sense of stochastic convolutions. Let us define what we mean with a solution on a fixed stochastic basis. The function spaces $H_2^{-1}(\CO)$ and $H^{\rho}_2(\CO)$ used in the following definition are discussed in Appendix \ref{dbouley-space}.

\begin{definition}\label{def:singleSLN}
A couple $(u,v)$ of stochastic processes on $\Afrak$ is called \emph{solution to the system} \eqref{equ1ss}--\eqref{eqv1ss} for initial data $(u_0,v_0)$ if there exists $\rho\in\mathbb{R}$ such that
\[u\in L^2(\Omega;C([0,T];H_2^{-1}(\CO)))\cap L^{\gamma+1}(\Omega\times (0,T)\times \Ocal)\]
and
\[v\in L^2(\Omega;C([0,T];H^{\rho}_2(\CO)))\cap L^2(\Omega\times (0,T);H_2^{\rho+1}(\CO)) \]
such that $u$ and $v$ are $\mathbb{F}=(\Fcal_t)_{t\in [0,T]}$-adapted, and satisfy
\begin{equation}\label{eq:uStratonovich}u(t)=u_0+\int_0^t (r_u\Delta  {u}^{[\gamma]}(s)-\chi u(s) v^2(s))\,ds+\int_0^t \sigma_1 u(s)\circ dW_1(s)\end{equation}
for every $t\in [0,T]$, $\P$-a.s. in $H^{-1}_2(\Ocal)$,
and
\begin{equation}\label{eq:vStratonovich}v(t)=v_0+\int_0^t (r_v\Delta v(s)+u(s) v^2(s))\,ds+\int_0^t \sigma_2 v(s)\circ dW_2(s)\end{equation}
for every $t\in [0,T]$, $\P$-a.s. in $H_2^{\rho}(\Ocal)$.
\end{definition}

Here, $V:=L^{\gamma+1}(\Ocal)$ is dualized in a Gelfand triple over $\Hcal:=H^{-1}_2(\Ocal)$ which, in turn, is identified with its own dual by the Riesz isometry. Thus $V^\ast$ is not to be mistaken to be equal to be the usual Banach space dual of $V$, namely $L^{\frac{\gamma+1}{\gamma}}(\Ocal)$, see the proof of Theorem \ref{theou1} for details.
Note that the above $ds$-integral for $u$ in \eqref{eq:uStratonovich} is initially a $V^\ast$-valued Bochner integral, however seen to be in fact $H^{-1}_2(\Ocal)$-valued, see the discussion in \cite[Section 4.2]{weiroeckner} for further details. 

As mentioned before, due to the loss of the original probability space, we are considering solutions in the weak probabilistic sense.

\begin{definition}\label{Def:mart-sol}
A {\sl  martingale solution} to the problem
\eqref{equ1ss}--\eqref{eqv1ss} for initial data $(u_0,v_0)$ is a system
\begin{equation}
\left(\Omega ,{{\mathcal{F}}},{\mathbb{F}},\mathbb{P},
(W_1,W_2), (u,v)\right)
\label{mart-system}
\end{equation}
such that
\begin{enumerate}
\item the quadruple $\mathfrak{A}:=(\Omega ,{{\mathcal{F}}},{\mathbb{F}},\mathbb{P})$ is a complete filtered
probability space with a filtration ${\mathbb{F}}=(\mathcal{F}_t)_{t\in [0,T]}$ satisfying the usual conditions,
\item $W_1$ and $W_2$ are independent $H_1$-valued, respectively, $H_2$-valued Wiener processes over the probability space
$\mathfrak{A}$ with covariance operators $Q_1$ and $Q_2$, respectively,
\item both $u:[0,T]\times \Omega \to H^{-1}_2(\CO)$ and $v:[0,T]\times \Omega \to H_2^{\rho}(\CO)$ are ${\mathbb{F}}$-adapted
processes such that the couple $(u,v)$ is a solution to the system \eqref{equ1ss} and \eqref{eqv1ss} over the probability space $\mathfrak{A}$ in the sense of Definition \ref{def:singleSLN} for some $\rho\in\mathbb{R}$.
 \end{enumerate}
\end{definition}

\begin{remark}
For our purposes, instead of the Stratonovich formulation the system \eqref{equ1ss}--\eqref{eqv1ss}, it is convenient to consider the equations in It{\^o} form:
\begin{equation}\label{eq:uIto}
 d {u}(t,x)  =  ( r_u\Delta   u^{[\gamma]}(t,x) -\chi  u(t,x)\,v^2(t,x))\, dt +\sigma_1 u(t,x)\,  dW_1(t,x),
\end{equation}
and,
\begin{equation}\label{eq:vIto}
d{v}(t,x) = ( r_v\Delta v(t,x)  + u(t,x)\,v^2(t,x))\, dt  +\sigma_2 v(t,x)\, d W_2(t,x),
\end{equation}
with initial data $u(0)=u_0$ and $v(0)=v_0$.
One reason for this is that the stochastic integral then becomes a local martingale.  In order to show the existence of a solution to the Stratonovich system, one would have to incorporate the It\^o-Stratonovich conversion term (cf. \cite{evans2012introduction}), which, due to the linear multiplicative noise, is a linear term being just a scalar multiple of $u$, $v$, respectively.
If one is interested in the exact form of the correction term, we refer to \cite{duan}. We also refer
to the discussion in \cite{grayscott}, where the constant accounting for the correction term is computed explicitly.
\end{remark}

From now on, we shall consider the system \eqref{equ1ss}--\eqref{eqv1ss} in It{\^o} form, that is, as in \eqref{eq:uIto}--\eqref{eq:vIto}. Definitions \ref{def:singleSLN} and \ref{Def:mart-sol} are supposed to be adjusted in the obvious way, keeping the statement on the regularity of the solution unchanged.
We note that we can solve \eqref{equ1ss}--\eqref{eqv1ss} by a straightforward modification of the proof given here.

Before presenting our main result, we will first introduce the hypotheses on $d$, $\gamma$, $\rho$ and the initial conditions $u_0$ and $v_0$, and on
the multiplication operators  $\sigma_1$, $\sigma_2$. Most of the hypotheses are technical in nature, as they lead to several different embeddings for function spaces and interpolation spaces that we need to use in our proofs.

\begin{hypo}[Existence]
\label{init}
Let $d\in\{1,2,3\}$.
Let $\gamma>1$, $\rho\in\mathbb{R}$,  and $m>2$, $m_0>\frac {2(\gamma+1)}\gamma$ such that
$$\frac d2 -\frac 2m-\frac d {m_0} \le\rho<1-\frac d2
$$
and $p^\ast\ge 2$ and $p_0^\ast\ge 2$ such that
$$
\frac 1{p^\ast}+\frac 2m< 1\quad\mbox{and}\quad \frac {2}{ m_0}+\frac 1{p_0^\ast}< \frac {\gamma}{\gamma+1}
.
$$
{Let
$$l> 1+\left(1-\frac{d}{2}-\rho\right)^{-1}
$$}

Let us assume that the initial conditions $u_0$, $v_0$ are $\Fcal_0$ measurable and satisfy
$$\EE \vert u_0\vert_{L^{2}}^l \vee\EE \vert u_0\vert_{L^{p^\ast}}^{p^\ast_0}\vee\EE \vert u_0\vert_{L^{\gamma+1}}^{\gamma+1}<\infty,\quad
 \mbox{and}\quad
\EE \vert v_0\vert_{H^{\rho}_2}^{m_0}\vee\EE \vert v_0\vert_{H^{-\delta_0}_m}^{m_0}<\infty, \quad \delta_0<\frac 1{m_0},
  $$
and that $u_0$ and $v_0$ are a.e. nonnegative functions (nonnegative Borel measures that are finite on compact subsets, respectively).
\end{hypo}

\begin{hypo}[Noise]\label{wiener}
Let  $\{\psi_k:k\in\mathbb{Z}\}$ be the eigenfunctions of $-\Delta$ and $\{\nu_k:k\in\mathbb{Z}\}$ the corresponding eigenvalues.
Let $W_1$ and $W_2$ be a pair of Wiener processes given by
$$
{W}_j(t,x)=\sum_{k\in\mathbb{Z}} \lambda_k^{(j)} \psi_k(x)\beta^{(j)}_k(t), \quad t\in[0,T],\quad j=1,2,\quad x\in\Ocal,
$$
where $\{\lambda_k^{(j)}\;:\;k\in\mathbb{Z}\}$, $j=1,2$, is a pair of nonnegative sequences belonging to $\ell^2(\mathbb{Z})$,
$\{ \beta_k^{(j)}:[0,T]\times\Omega\to\mathbb{R}\;:\;(k,j)\in\mathbb{Z}\times\{1,2\}\}$
is a family of
independent standard real-valued Brownian motions.
We assume that $\lambda^{(j)}_k\le C_j \vert\nu_k\vert^{-\delta_j}$, $j=1,2$, where $\delta_1> \frac{1}{2}\vee \left(\frac{d}{2}-\frac{1}{4}\right)$, $\delta_2>\frac{1}{2}\vee
\left( \frac{d}{2}-\frac{1}{4}\right)$, and $C_1>0$, $C_2>0$.

Compare also with Subsection \ref{subsec:noise} for details.
\end{hypo}

Under these hypotheses, the existence of a martingale solution can be shown.

\begin{theorem}\label{mainresult}
Assume that Hypotheses \ref{init}--\ref{wiener} hold.
Then there exists a martingale solution to system \eqref{eq:uIto}--\eqref{eq:vIto} satisfying the following properties
\begin{enumerate}[(i)]
\item $u(t,x)\ge 0$ and $v(t,x)\ge 0$ $\PP$-a.s., for a.e. $t\in [0,T]$ and a.e. $x\in\Ocal$,
  \item Let  $p\ge 1$ and $\EE\vert u_0\vert^{p+1}_{L^{p+1}}<\infty$. Then there exists a constant $C_0(p,T)>0$ such that
\begin{align*}  \nonumber
\lqq{\EE \left[\sup_{0\le s\le T} \vert u (s)\vert_{L^{p+1} }^{p+1}\right]
+ \gamma p(p+1)r_u\EE \int_0^ \TInt\int_\CO  \vert u(s,x)\vert^{p+\gamma-2} \vert\nabla u (s,x)\vert^2\, dx \, ds} &
\\\nonumber &{} +(p+1)\chi\EE \int_0^\TInt \int_\CO  \vert u(s,x) \vert^{p+1} \vert v(s,x)\vert^2\, dx \, ds \le C_0(p,T)\,\left( \EE\vert u_0\vert_{L^{p+1}}^{p+1}+1\right).
\end{align*} 
  \item for any choice of parameters $\rho$, $m_0$, $l$ as in Hypothesis \ref{init}, there exists a constant $C_2(T)>0$ such  that
  \begin{align*}&\EE \left[\sup_{0\le s\le T} \vert v(s)\vert_{H^{\rho}_{2}}^{m_0}\right]+\EE\left(\int_0^T \vert v(s)\vert_{H^{\rho+1}_2}^2\,ds\right)^{m_0/2}\\
 \le & C_2(T)\left(1+  \EE \vert v_0\vert_{H^{\rho}_2}^{{m_0}}+\EE \vert u_0\vert_{L^{2}}^{l}\right).
 \end{align*}
\end{enumerate}
\end{theorem}

We shall also collect a standard notion on uniqueness.

\begin{definition} The system \eqref{eq:uIto}--\eqref{eq:vIto} is called \emph{pathwise unique} if, whenever $(u_i,v_i)$, $i=1,2$ are martingale solutions to the system \eqref{eq:uIto}--\eqref{eq:vIto} on a common stochastic basis
$(\Omega, \CF,\BF,\PP,(W_1,W_2))$, such that
$\PP(u_1(0)=u_2(0))=1$ and $\PP(v_1(0)=v_2(0))=1$, then
\[
\PP(u_1(t)=u_2(t))=1 \quad \mbox{and} \quad
\PP(v_1(t)=v_2(t))=1,
 \quad \mbox{for every } t \in [0,T].
\]
\end{definition}

The theorem of Yamada and Watanabe \cite{YW1} asserts that
(weak) existence and pathwise uniqueness of the solution to a stochastic equation is equivalent to the existence of a unique strong solution.
Therefore, showing pathwise uniqueness is a fundamental step for obtaining the existence of a unique strong solution.
Under certain additional conditions, as collected below, we are able to prove pathwise uniqueness.
{
\begin{hypo}[Uniqueness]\label{Uniqueness}
Let $\delta_0\in(0,\frac 1 \gamma)$, and $
\rho\ge \frac d2-\frac 12$.
\end{hypo}
\begin{remark}\label{rem:dimone}
The condition $\delta_0>0$ implies that the Hypothesis \ref{Uniqueness} can only be satisfied if  $d=1$.
\end{remark}

\begin{theorem}\label{thm_path_uniq}
Let $T>0$ and let $\mathfrak{A}=(\Omega,\CF,(\CF_t)_{t\in[0,T]},\PP)$ a probability space satisfying the usual conditions, let $(W_1,W_2)$ be a Wiener process  over $\MA$
 on $H_1$ and $H_2$ and satisfying Hypothesis \ref{wiener}.
 Let initial data  $u_0 \in H^{-1}_2(\CO)$ and $v_0 \in L^{2}(\CO)$
 together with the parameters $m,m_0,p^\ast,p^\ast_0,l,\delta_0$ and $\rho$  satisfying Hypothesis \ref{init} and Hypothesis \ref{Uniqueness}, where $\Ocal\subset\mathbb{R}$. Assume further that Hypothesis \ref{wiener} holds.
  Let $(u_i,v_i)$, $i=1,2$, be two solutions to the system \eqref{eq:uIto}--\eqref{eq:vIto} with initial data $(u_0,v_0)$ and belonging $\PP$-a.s. to $C([0,T];H^{-1}_2(\CO))\times C([0,T];H^{-\delta_0}_2(\CO))$. Then, the processes $(u_1,v_1)$ and $(u_2,v_2)$ are indistinguishable in $H^{-1}_2(\CO)\times H^{-\delta_0}_2(\CO)$.
\end{theorem}}
\begin{remark}
It can be shown that for dimension $d=1$
such indices  $m^\ast,m_0,p^\ast,p^\ast_0,l,\delta_0$ and $\rho$  satisfying Hypothesis \ref{init} and Hypothesis \ref{Uniqueness}
can be found. Because of our condition $\rho\ge\frac d2-\frac 12$, we cannot handle dimensions $d\in \{2,3\}$ as Hypothesis \ref{init} then implies that $\rho<0$.
\end{remark}

\begin{proof}
The proof is postponed to Section \ref{pathwise}.
\end{proof}

By this Theorem at hand, the following corollary is a consequence of the Theorem of Yamada-Watanabe.
\begin{corollary}\label{cor:strong}
There exists a unique strong solution (in the stochastic sense) to the system \eqref{eq:uIto}--\eqref{eq:vIto}.
\end{corollary}
\begin{proof}
This follows from Theorem \ref{mainresult} in combination with Theorem \ref{thm_path_uniq} and the Yamada-Watanabe theorem, see \cite[Appendix E]{weiroeckner} and \cite{Qiao2010,Kurtz2007,Ondrejat2004}.
\end{proof}

The proof of Theorem \ref{mainresult} is an application of the Schauder-Tychonoff-type Theorem \ref{ther_main} and consists of the following five steps to verify the conditions of Theorem \ref{ther_main}.

In the first step, we are specifying the underlying Banach spaces. In the second step, we shall construct the operator $\mathcal{V}$ for a truncated system and show that the operator $\mathcal{V}$ satisfies the assumptions of Theorem \ref{ther_main}. In the third step, we localize via stopping times and glue the fixed point solutions together, which exist by Theorem \ref{ther_main}. In the fourth step, we prove that the stopping times are uniformly bounded. In the fifth step, we show that we indeed yield a martingale solution satisfying the above properties. However, to keep the proof itself simple, we will postpone several technical a priori estimates and further regularity results which are collected in Section \ref{technics}.

\begin{proof}[Proof of Theorem \ref{mainresult}]
\begin{steps}
\item{\bf The underlying space(s).}
Here we define the spaces on which the operator $\mathcal{V}$ will act.
Let the probability space $\mathfrak{A}=(\Omega,\CF,\mathbb{F},\PP)$ be given and let $W_1$ and $W_2$ be two independent $H_1$ and $H_2$-valued Wiener processes defined over $\mathfrak{A}$ with covariances $Q_1$ and $Q_2$. Let $W=(W_1,W_2)$, $H=H_1\times H_2$, with covariance operator
$$ Q=\left(\begin{matrix} Q_1 & 0 \\0 & Q_2\end{matrix}\right).
$$

Let us define the Banach space
$$\mathbb{Y}= L^{\gamma+1}(0,T; {L^{\gamma+1}}(\CO))\cap L^\infty(0,T; {H^{-1}_2}(\CO))$$
equipped with the norm
$$
\Vert\eta\Vert_{\mathbb{Y}}:= \Vert\eta\Vert_{L^{\gamma+1}(0,T; {L^{\gamma+1}})}+ \Vert\eta\Vert_{L^\infty(0,T; {H^{-1}_2})},\quad\eta\in \mathbb{Y},
$$
and the reflexive Banach space
$$\mathbb{Z}:=L^{m_0}(0,T;L^m(\CO)),$$
equipped with the norm
$$
\Vert\xi\Vert_{\mathbb{Z}}:= \Vert\xi\Vert_{L^{m_0}(0,T;L^m)},\quad\xi\in \mathbb{Z}.
$$

Finally, let us fix
an auxiliary  Banach space $\BH := L^2(0,T; H^{\rho+1}_2(\CO))\cap L^\infty(0,T; {H^{\rho}_2}(\CO))$
equipped with the norm
\begin{equation}\label{eq:BZeq}
\Vert\xi\Vert_{\BH}:= \Vert\xi\Vert_{L^2(0,T; H^{\rho+1}_2)}+ \Vert\xi\Vert_{L^\infty(0,T; {H^{\rho}_2})},\quad\xi\in \BH.
\end{equation}

\begin{remark}\label{variationalremark}
If
$$\frac d2 -\rho\le \frac 2m+\frac d {m_0},$$
then one can show by Sobolev embedding and interpolation theorems (see Proposition \ref{interp_rho} in the appendix) that
there exists a constant $C>0$ such that
$$
\Vert\xi\Vert_{\mathbb{Z}}\le C\Vert\xi\Vert_{\BH },\quad \xi\in\BH .
$$
\end{remark}
Let us fix the compactly embedded reflexive Banach subspace of $\mathbb{Z}$ by\footnote{For the definition of $\mathbb{W}^{\alpha}_{m_0}$, we refer to Appendix \ref{dbouley-space}.}
$$\mathbb{Z}':=L^{m_0}(0,T;H^\sigma_m(\CO))\cap \mathbb{W}^{\alpha}_{m_0}(0,T;H^{-\delta}_m(\CO))$$
equipped with the norm
$$
\Vert\xi\Vert_{\mathbb{Z}'}:= \left(\Vert\xi\Vert_{L^{m_0}(0,T;H^\sigma_m)}^{m_0}+\Vert\xi\Vert_{\mathbb{W}^{\alpha}_{m_0}(0,T;H^{-\delta}_m)}^{m_0}\right)^{1/m_0},
$$
compare with Appendix \ref{dbouley-space}, where the compact embedding $\mathbb{Z'}\hookrightarrow\mathbb{Z}$ is discussed. The parameters $\sigma>0$, $\alpha\in (0,1)$, $\delta>0$ are specified in Proposition \ref{semigroup}.

Now,
denote the space of progressively measurable (pairs of) processes $\mathcal{M}_{\mathfrak{A}}^{2,m_0}(\mathbb{Y},\BX)$ by
\begin{align*} 
\mathcal{M}_{\mathfrak{A}}^{2,m_0}(\mathbb{Y},\BX):= &\Big\{ (\eta,\xi) \;\colon\; \eta,\xi:[0,T ]\times   \Omega \to \Dcal^\prime(\Ocal) \;\mbox{such that} \\
&\qquad\mbox{$\eta$ and $\xi$ are progressively measurable on $\mathfrak{A}$ and}\\
   &\qquad\EE \Vert\eta\Vert^2_{\mathbb{Y}} <\infty \; \mbox{and}\;\EE \Vert\xi\Vert^{m_0}_{\BX }<\infty\Big\}
\end{align*} 
equipped with the norm
\begin{equation}\label{eq:MMnorm}\Vert (\eta,\xi)\Vert_{\mathcal{M}_{\mathfrak{A}}^{2,m_0}(\mathbb{Y},\BX )}:=\left(  \EE \Vert\eta\Vert^2_\mathbb{Y}\right)^\frac 12+\left(  \EE \Vert\xi\Vert^{m_0}_\BX \right)^\frac 1{m_0},
\quad (\eta,\xi)\in\mathcal{M}^{2,m_0}_{\mathfrak{A}}(\mathbb{Y},\BX ).
\end{equation}
Note that here, progressive measurability is meant relative to the Borel $\sigma$-fields of the target spaces $H_2^{-1}(\Ocal)$ and $H^{\rho}_2(\Ocal)$ respectively.

Finally, for fixed $\Reins,\Rzwei,\Rdrei>0$, let us define the subspace $\mathcal{X}_{\mathfrak{A}}=\subX$
by
\bigskip\bigskip
\begin{align*} 
&\subX\\
:=&\Bigg\{ (\eta,\xi)\in\CM_\MA^{2,m_0}(\mathbb{Y},\BX )\;\colon\\
& \qquad
 \EE \left[\sup_{0\le s\le T} \vert\eta(s)\vert_{L^{p^\ast}}^{p_0^\ast}\right]^\frac 1{{p_0^\ast}}\le \Reins,\;  \EE \Vert\xi\Vert^{m_0}_{\BX}\le \Rzwei,\;\EE\Vert\xi\Vert_{\BH}^{m_0}\le \Rdrei, \;\mbox{and}\\
 &\qquad \quad\eta\;\mbox{and}\;\xi\;\mbox{are nonnegative}\;\P\otimes\mbox{Leb-a.e. in}\;\Dcal^\prime(\Ocal)\Bigg\}.
 \end{align*} 

It is easy to verify that $\mathcal{X}_{\mathfrak{A}}$ is a sequentially weak$^\ast$-closed and bounded subset of $\mathcal{M}^{2,m_0}_{\mathfrak{A}}(\mathbb{Y},\BX)$. The continuous functions $\Psi_i$, $i=0,1,2$, satisfying assumption (a) of Theorem \ref{ther_main} can be defined in the obvious way. Also, it is easy to find measurable functions $\Theta_i$, $i=0,1,2$, with closed sublevel sets such that assumption (b) of Theorem \ref{ther_main} is satisfied, and can be used to capture the nonnegativity by setting e.g. $\Theta_i((\eta,\xi)):=\infty$ if $\eta$ or $\xi$ is negative, $i=0,1,2$.

Every $(\eta,\xi)\in\mathcal{X}_{\mathfrak{A}}$ is a pair of a (equivalence class of a) nonnegative function (or a
nonnegative Borel measure that is finite on compact subsets of $\Ocal$).

\begin{remark}
Note that in order to apply Theorem \ref{ther_main} formally, we will assume the obvious modification (or extension) of its statement and proof, such that we can
treat pairs of spaces with different exponents like $\CM_\MA^{2,m_0}(\mathbb{Y},\BX )$.\end{remark}

\item  {The truncated system.}

Let $\phi \in \Dcal(\mathbb{R})$ be a smooth cutoff function that satisfies
$$
\phi(x) \bcase =0, &\mbox{ if } \vert x\vert\ge 2,
\\ \in [0,1], &\mbox{ if } 1<\vert x\vert<2,
\\=1, &\mbox{ if } \vert x\vert\le 1,
\ecase
$$
and let $\phi_\kk(x):= \phi(x/\kk)$, $x\in\mathbb{R},\,\kk \in \mathbb{N}$.
In addition, for any progressively measurable pair of processes $(\eta,\xi)\in\CM_\MA^{2,m_0}(\BY,\mathbb{Z})$
let us define for $t \in [0,T]$
\begin{align*} 
h(\eta,\xi,t ):=& \sup_{0\le s\le t}\vert\eta(s)\vert^2_{H^{-1}_2}+\int_0^ t\vert\eta(s)\vert_{L^{\gamma+1}}^{\gamma+1}\, ds+\Vert\xi\mathbbm{1}_{[0,t]}\Vert_{\BH}
, \end{align*} 
where $\nu\in(0,1]$ is chosen such that $\frac 1{p_0^\ast}+\frac {\nu m_0}{\gamma+1}\le 1$ and $\frac 1{p_0^\ast}+\frac {\nu }{m_0}\le 1 $.

Let us consider the truncated system given by
\begin{equation}\label{eq:cutoffu}
\left\{ \barray  d {\uk }(t) & =&  \left[r_u \Delta   (\uk (t))^{[\gamma]} -\chi \phi_\kappa(h (u_\kappa,v_\kappa,t))u_\kappa (t) v_\kappa^2(t) \right]\, dt
+\sigma_1\uk (t)  dW_1(t),
 \\
  \uk (0) &=& u_0  , \phantom{\Big\vert}\earray\right.
\end{equation}
and
\begin{equation}\label{eq:cutoffv}
\left\{ \barray
d{\vk }(t) &=& \left[ r_v\Delta \vk (t)  +  \phi_\kappa (h(u_\kappa,v_\kappa,t))u_\kappa (t) v_\kappa^2(t) \right]\,dt
 +\sigma_2 \vk (t)d W_2(t) .
\\
 {\vk }(0) &=& v_0 . \phantom{\Big\vert}\earray\right.
\end{equation}

We shall show the existence of a martingale solution to system \eqref{eq:cutoffu}--\eqref{eq:cutoffv}.
\begin{proposition}\label{prop.exist.n}
For any $\kk\in\NN$, there exists constants $\Reins>0$, $\Rzwei>0$, $\Rdrei>0$, depending on $\kappa$, such that
there exists a martingale solution $(\uk,\wk)$ to system \eqref{eq:cutoffu}--\eqref{eq:cutoffv} contained in $\subX\subset\CM_\MA^{2,m_0}(\mathbb{Y},\BX )$.
\end{proposition}

\begin{proof}[Proof of Proposition \ref{prop.exist.n}]
The proof consists of several steps. First, we shall define an operator denoted by $\mathcal{V}_\kk$ which satisfies the assumptions of Theorem
\ref{ther_main}, yielding the existence of a martingale solution.

\begin{stepinner}
\item {\bf Definition of the operator $\Vcal_\kappa$.}

First, define
$$\Vcal_\kk:=\mathcal{V}_{\kk,\Afrak}:\subX\subset\CM_\MA^{2,m_0}(\BY,\BX)\to\CM_\MA^{2,m_0}(\BY,\BX)$$
by
\[\mathcal{V}_\kk(\eta,\xi):=(\uk,\wk),\quad\text{for}\quad(\eta,\xi)\in\subX\]
where $\uk$ is a solution to
\begin{equation}\label{eq:cutoffuu}
\left\{ \barray  d {\uk }(t) &=&  \left[r_u \Delta   (\uk (t))^{[\gamma]} -\chi \stoppt\uk (t)\,\xi ^2(t) \right]\, dt
+\sigma_1\uk (t) \, dW_1(t),
 \\
  \uk (0) &=& u_0  , \phantom{\Big\vert}\earray\right.
\end{equation}
and $\vk$ is a solution to
\begin{equation}\label{eq:cutoffvv}
\left\{ \barray
d{\vk }(t) &=& \left[ r_v\Delta \vk (t)  + \stoppt \eta (t)\,\xi ^2(t) \right]\,dt
 +\sigma_2 \vk (t)\, d W_2(t) ,
\\
 {\vk }(0) &=& v_0 . \phantom{\Big\vert}\earray\right.
\end{equation}

The operator $\mathcal{V}_\kk$ is well-defined for $(\eta,\xi)\in\subX$. In fact, by Theorem \ref{theou1} and Proposition \ref{positivityu}, given such a pair of processes $(\eta,\xi)$, the existence of a nonnegative unique solution $\uk$ to \eqref{eq:cutoffuu} for nonnegative initial data $u_0$ with
\begin{align} \label{constantu}
 \EE\Vert\uk\Vert_{\BY}^2 \le& C_1(\kk,T)
\end{align} 
follows. By Proposition \ref{propvarational} and Proposition \ref{positivityv}, the existence of a unique solution $v_\kappa$ to \eqref{eq:cutoffvv} with
$$ \EE \Vert\vk\Vert_{\BH}^{m_0}
\le \EE \vert v_0\vert_{H^\rho_2}^{m_0} +C(\kk,T)\Reins,
$$
follows.

\item {\bf The $\Vcal_\kappa$-invariant set $\Xcal$.}

Let
\begin{align*} 
C_0(T)\Big(\EE \vert u_0\vert^{p^\ast_0}_{L^{p^\ast}}+1\Big)\le\Reins ,
\end{align*} 
where the constant $C_0(\gamma,T)$ is as in \eqref{uniformlpbounds}.
Then, due to Proposition \ref{uniformlpbounds}, we know that
\begin{align*} 
\EE\left[ \sup_{0\le s\le T} \vert\uk(s)\vert^{p^\ast_0}_{L^{p^\ast}}\right]\vee\EE\left[ \sup_{0\le s\le T} \vert\uk(s)\vert_{L^{\gamma+1}}^{\gamma+1}\right]
 \le \Reins.
\end{align*} 
Furthermore, by Proposition \ref{propvarational}, the existence of a unique
solution to \eqref{eq:cutoffvv} such that
\begin{align} \label{constantvkk}
\EE \Vert\wk\Vert_{\BH}^{m_0}\le&  \EE\vert v_0\vert^{m_0}_{H^\rho_2}+R_0 C_2(\kk,T)
\end{align} 
follows.

Let $\Rdrei:= \EE\vert v_0\vert^{m_0}_{H^\rho_2}+R_0 C_2(\kk,T)$.
Finally, by Proposition  \ref{semigroup} we can show that for $\Rzwei\ge C(T)\left\{ \EE\vert v_0\vert_{H_m^{-\delta_0}}^{m_0}+C(\kappa)\right\}$
(here the constants are given by Proposition   \ref{semigroup})
\begin{align} \label{constantvkk1}
\EE \Vert\wk\Vert_{\BZ}^{m_0}\le&  \Rzwei.
\end{align} 

Then,
$\Vcal_\kk$ maps $\mathcal{X}_\MA(\Reins,\Rzwei,\Rdrei)$ into itself.
Hence, assumption (i) of Theorem \ref{ther_main} is satisfied.

\item {\bf Continuity of $\Vcal_\kappa$ on $\Xcal$.}

Next we show that assumption (ii) of Theorem \ref{ther_main} is satisfied.
We need to show that
the restriction of the operator $\mathcal{V}_\kappa$ to $\subX$ is continuous in the strong topology of $\CM_{\mathfrak{A}}^{2,m_0}(\BY,\BX)$.
Let $\{(\eta^{(n)},\xi^{(n)}):n\in\NN\}\subset \subX$ be  a sequence converging strongly to $(\eta,\xi)$ in $\CM_{\mathfrak{A}}^{2,m_0}(\BY,\BX)$.

Firstly, due to Proposition \ref{continuity} and Remark \ref{variationalremark}, we know that the sequence $\{\uk^{(n)}:n\in\NN\}$ where $\uk^{(n)}$ solves
\begin{equation}\label{eqv1kk}
\left\{ \barray
d{\uk ^{(n)}}(t) &=& \left[ r_u\Delta (\uk^{(n)}(t))^{[\gamma]}   -\chi  \phi_\kappa(h(\eta^{(n)},\xi^{(n)},t))\, \uk^{(n)} (t)\,(\xi^{(n)}(t)) ^2 \right]\,dt\\
&& \quad +\sigma_1 \uk ^{(n)}(t)\,d W_1(t) ,
\\
 u_\kappa^{(n)}(0) &=& u_0 . \phantom{\Big\vert}\earray\right.
\end{equation}
converges in $\CM_{\mathfrak{A}}^{2,m_0}(\BY,\BX)$ to $\uk$ which solves \eqref{eq:cutoffuu}.

Secondly, due to Proposition \ref{semigroup_continuous} and keeping Remark \ref{variationalremark} in mind, it follows that
the sequence $\{\vk^{(n)}:n\in\NN\}$, where $\vk^{(n)}$ solves
\begin{equation}\label{equ1kk}
\left\{ \barray
d{\vk ^{(n)}}(t) &=& \left[ r_v\Delta \vk^{(n)} (t)  + \phi_\kappa(h(\eta^{(n)},\xi^{(n)},t))\, \eta^{(n)} (t)\,(\xi^{(n)}(t)) ^2 \right]\,dt\\
&& \quad +\sigma_2 \vk ^{(n)}(t)\,d W_2(t) ,
\\
 v_\kk^{(n)}(0) &=& v_0 . \phantom{\Big\vert}\earray\right.
\end{equation}
converges in $\CM_{\mathfrak{A}}^{2,m_0}(\BY,\BX)$ to $\vk$, where $\vk$ solves \eqref{eq:cutoffvv}.

This gives that the operator
$$\Vcal_\kk\vert_{\subX}:\subX\subset\CM_{\mathfrak{A}}^{2,m_0}(\BY,\BX)\to\CM_{\mathfrak{A}}^{2,m_0}(\BY,\BX)$$
is continuous.

\item {\bf Tightness.}

Next we show that assumption (iii) of Theorem \ref{ther_main} is satisfied.

First, note that the embedding ${\mathbb{Z}'}\hookrightarrow {\mathbb{Z}}$
is compact by the Aubin-Lions-Simon Lemma \ref{th-gutman}.

Second, by Proposition \ref{semigroup} it follows that for any $\kk\in\NN$
there exists a constant $C=C(\kk,T)>0$ such that for all $(\eta,\xi)\in\subX$
$$
\EE\, \vert\vk\vert^{m_0}_{\mathbb{Z}'}\le C,
$$
where $\vk$ solves \eqref{eq:cutoffvv}. Observe that we have the compact embedding $L^{\gamma+1}(\CO)\hookrightarrow H^{-1}_2(\CO)$. Hence, by Proposition \ref{CC2} and by standard arguments (cf. \cite{ethier}), we know that
the laws of the set $\{ \uk:(\eta,\xi)\in\subX\}$ are tight on $C([0,T];H^{-1}_2(\CO))$.
By Proposition \ref{CC2}  and the Aubin-Lions-Simon Lemma \ref{th-gutman},
the laws of  $\{ \uk:(\eta,\xi)\in\subX\}$ are tight on $L^{\gamma+1}(0,T;L^{\gamma+1}(\CO))$.
Therefore, it follows that the set $\{ \mathcal{V}_\kk(\eta,\xi):(\eta,\xi)\in\subX\}$ is tight on $\BY\times \BX$,
which implies assumption (iii) of Theorem \ref{ther_main} for
$$\X^\prime:=L^{\gamma+1}(0,T;L^{\gamma+1}(\CO))\times\mathbb{Z}^\prime\hookrightarrow \mathbb{Y}\times \BX=:\X$$

\item {\bf Continuity of paths.}

To show  assumption (iv) of Theorem \ref{ther_main}, we notice that
by choosing
$$U:=H^{-1}_2(\Ocal)\times H_2^{\rho}(\CO),$$
it follows from Proposition \ref{CC2} and standard arguments on continuity of solutions to nonlinear SPDEs
that
$$\mathcal{V}_\kk(\eta,\xi)\in C([0,T];U)\subset \mathbb{D}([0,T];U)\quad\PP\text{-a.s.}$$
for all
$(\eta,\xi)\in\subX$.

\item {\bf Existence of a fixed point.}

Finally, by Theorem \ref{ther_main}, for each $\kappa\in\N$, there exists a probability space
$\tilde{\mathfrak{A}}_\kk=(\tilde \Omega_\kk,\tilde{\CF}_\kk,\tilde{\mathbb{F}}_\kk,\tilde{\PP}_\kk)$, a Wiener process $\tilde W^\kk=(\tilde W^\kk_1,\tilde W^\kk_2)$, modeled on $\tilde {\mathfrak{A}}_\kk$, and elements
$$(\tilde u_{\kappa},\tilde v_{\kappa})\in\mathcal{X}_{\tilde{\mathfrak{A}}_\kk}(R_0,R_1,R_2)\cap C([0,T];U)$$
such that we have a $\tilde \PP$-a.s. fixed point
$$
\mathcal{V}_{\kk,\tilde{\Afrak}_\kappa}(\tilde u_{\kappa}(t),\tilde v_{\kappa}(t))=(\tilde u_{\kappa}(t),\tilde v_{\kappa}(t))
$$
for every $t\in[0,T]$.
Due to the construction of
$\mathcal{V}_{\tilde{\Afrak}}$, the pair $(\tilde u_{\kappa},\tilde v_{\kappa})$
solves the system \eqref{eq:cutoffuu}--\eqref{eq:cutoffvv} over the stochastic basis $\tilde{\Afrak}_\kk$ with the Wiener noise $\tilde W^\kk$, see Section \ref{sec:proof_of_schauder} for details.
\end{stepinner}
\end{proof}

{
\item
Here, we will construct a  family of solutions $\{(\bar u_{\kappa} ,\bar v_{\kappa} ):\kappa\in\NN\}$  following the solution to the  original problem until a stopping time ${\bar \tau}_\kappa$. In particular,  we will introduce for each $\kappa\in\NN$ a new pair of processes $({\bar u}_{\kappa} , {\bar v}_{\kappa} )$ following the Klausmeier system up to the stopping time ${\bar \tau}_\kappa$. Besides, we will have
$({\bar u}_{\kappa} , {\bar v}_{\kappa} )\vert_{[0,{\bar \tau}_\kappa)}=({\bar u}_{\kk+1},{\bar v}_{\kk+1})\vert_{[0,{\bar \tau}_\kappa)}$.

Let us start with $\kappa =1$. From Proposition \ref{prop.exist.n}, we know there exists a martingale solution consisting of a probability space $\mathfrak{A}_1=(\Omega_1,\CF_1,\mathbb{F}_1,\PP_1)$, two independent Wiener processes $({W}^1_1,{W}^1_2)$
defined over  $\mathfrak{A}_1$, and
 a couple of processes $(u_1,v_1)$ solving $\PP_1$--a.s.\ the system
 \begin{align*} 
\left\{ \barray  du_1(t)&=&\left[r_u \Delta u_1^{[\gamma]}(t) -\chi \,\phi_1(h(u_1,v_1,t))\, u_1 (t)\,v_1^2(t)\right] dt+ \sigma_1 u_1(t)\,d {W}^1_1(t),
 \\
 dv_1(t)&=&\left[r_v \Delta v_1(t)+ \,\phi_1(h(u_1,v_1,t))\, u_1 (t)\,v_1^2(t)
\right]\,dt+\sigma_2 v_1(t)\, d {W}^1_2(t),
 \\ (u_1(0),v_1(0))&=&(u_0,v_0).
 \earray\right.\end{align*} 
 Let us define now the stopping time
\begin{align*}
\tau_1^\ast:=\inf \{s\ge 0\;\colon\; h(u_1,w_1,s) \ge 1 \}.
\end{align*}
Observe, on the time
interval $[0,\tau_1^\ast)$, the pair $(u_1,v_1)$ solves the system given in
\eqref{equ1n}.
Now, we define a new pair of processes $(\bar u_1,\bar v_1)$  following $(u_1,v_1)$  on $[0,\tau_1^\ast)$ and extend this processes
 to the whole interval $[0,T]$ in the following way.
First, we put $\overline{\mathfrak{A}}_1:=\mathfrak{A}_1$ and $\bar {{W}}_j^1:={W}_j^1$, $j=1,2$,
and let us introduce the processes $y_1$ and $y_2$
being a strong solution over $\overline{\mathfrak{A}}_1$ to
 \begin{eqnarray} 
d y_1(t, u_1(\tau^\ast_1),\sigma) &=&  r_u\Delta y_1^{[\gamma]}(t, u_1(\tau^\ast_1),\sigma)\,dt + \sigma_1 y_1(t, u_1( \tau^\ast_1),\sigma) \,d(\theta_{\sigma} \bar{{W}}^1_1)(t)
  \nonumber \\ \label{eq1} y_1(0,u_1(\tau^\ast_1))&=&u_1(\tau^\ast_1){}
,\quad t\ge 0 ,\\ \nonumber
 y_2(t, v_1( \tau^\ast _1),\sigma) &=&e^{r_v\Delta t} v_1(\tau^\ast_1)
 \\
 &&+ \int_{0}^t  e^{(r_v\Delta )(t-s)} \sigma_2 y_2 (s, v_1(\tau^\ast _1),\sigma) \,d(\theta_{\sigma} \bar{{W}}^1_2)(s)
,\quad t\ge 0.  \label{eq2}
\end{eqnarray} 
Here, $\theta_\sigma$ is the shift operator which maps ${W}_j(t)$ to ${W}_j(t+\sigma)-W_j(\sigma)$.
Since the couple $(u_1,v_1)$  is continuous in $H_2^{-1}(\CO)\times H^\rho_2(\CO)$, we know that  $u_1(\tau^\ast_1)$ and $v_1(\tau^\ast_1)$ are well-defined random variables belonging $\PP_1$-a.s.\ to $L^2(\CO)$ and $H^\rho_2(\CO)$, respectively.
By \cite[Theorem 2.5.1]{BDPR2016} the existence of a unique solution $y_1$ over $\overline{\mathfrak{A}}_1$ to \eqref{eq1} in $H^{-1}_2(\CO)$ is given. Since $(e^{t(r_v\Delta -\tilde{\alpha} \operatorname{Id})})_{t\ge 0}$ is an  analytic semigroup on  $H^\rho_2(\CO)$ for $\tilde{\alpha}>0$,
the existence of a unique solution $y_2$  over $\overline{\mathfrak{A}}_1$ to \eqref{eq2}  in $H^\rho_2(\CO)$
can be shown by standard methods, cf. \cite{DaPrZa:2nd}. Furthermore, it is straightforward to verify that the assumptions on the initial conditions are satisfied.
Now, let us define two  processes $\bar u_1 $ and $\bar v_1$
being identical to $u_1$ and $v_1$, respectively,  on the time interval $[0,\tau^\ast_1)$ and
following the porous media, respective, the heat equation (with lower order terms) with noise and without nonlinearity, i.e.,   $y_1(\cdot,u_1(\tau^\ast_1),\tau^\ast_1)$ and $y_2(\cdot,v_1(\tau^\ast_1),\tau^\ast_1)$, afterwards.
In particular, let
$$
\bar u_1  (t) = \bcase u_1(t) & \mbox{ for } 0\le t< \tau^\ast_1,\\
y _1(t,u_1(\tau^\ast_1),\tau^\ast_1) & \mbox{ for } \tau^\ast_1\le  t \le T,\ecase
$$
and
$$
\bar v_1  (t) = \bcase v_1 (t) & \mbox{ for } 0\le t< \tau_1^\ast,\\
y_2 (t,v_1(\tau^\ast_1),\tau_1^\ast
)  & \mbox{ for } \tau_1^\ast \le  t \le T.\ecase
$$
Let us now construct the probability space and the processes for the next time interval. First, let $(u_1(\tau^\ast_1),v_1(\tau^\ast_1))$ have probability law $\mu_1$ on $H_2^{-1}(\CO)\times H^\rho_2(\CO)$. Again, from Proposition \ref{prop.exist.n}, we know
there is a martingale solution consisting of a probability space $\mathfrak{A}_2=(\Omega_2,\CF_2,\mathbb{F}_2,\PP_2)$, a pair of independent Wiener processes $({W}_1^2,{W}_2^2)$ such that $({\bar{W}}_1^1,{\bar{W}}_2^1,{W}_1^2,{W}_2^2)$
are independent as well, a couple of processes $(u_2,v_2)$ solving $\PP_2$-a.s.\ the system
 \begin{align*} 
\left\{ \barray  du_2(t)&=&\left[r_u \Delta  u^{[\gamma]}_2(t) -\chi \,\phi_2(h(u_2,v_2,t))\, u_2 (t)\,v_2^2(t)\right] dt+ \sigma_2 u_2(t) \, d{W}^2_1(t),
 \\
 dv_2(t)&=&\left[r_v \Delta v_2(t)+\phi_2(h(u_2,v_2,t))\, u_2 (t)\,v_2^2(t)
\right]\,dt+\sigma_2 v_2(t)\, d {W}^2_2(t),
 \\&&(u_2(0),v_2(0))\sim\mu_1.
 \earray\right.\end{align*} 
 with initial condition $(u_2(0),v_2(0))$ having law $\mu_1$.
 Let us define now the stopping times on $\mathfrak{A}_2$,
\begin{align*}
\tau_2^\ast:=\inf \{s\ge 0:h(u_2,v_2,s) \ge 2 \}.
\end{align*}
Let $\overline{\mathfrak{A}}_1:=(\overline{\Omega}_1,\overline{\CF}_1,\overline{\mathbb{F}}_1,\overline{\PP}_1):=\mathfrak{A}_1$, with $\overline{\mathbb{F}}_1:=(\overline{\CF}^1_t)_{t\in [0,T]}$.
Let $\overline{\Omega}_2:=\overline{\Omega}_1\times\Omega_2$, $\overline{\CF}_2:=\overline{\CF}_1\otimes \CF_2$, $\overline{\PP}_2:=\overline{\PP}_1\otimes\PP_2$ and let
$\overline{\mathbb{F}}_2:=(\overline{\CF}^2_t)_{t\in [0,T]}$, where
\begin{align*}&\overline{\CF}^2_{t}:=
\sigma\left(\left\{A\cap\{t<\tau^\ast_1\}\;\colon\;A\in\overline{\Fcal}_t^1\right\}\cup\overline{\Fcal}_0^1\cup\left\{A\cap\{t\ge \tau^\ast_1\}\;\colon\;A\in\Fcal_{t-\tau_1^\ast}^2\right\}\right),\end{align*}
where
\[\Fcal^2_{t-\tau_1^\ast}:=\left\{A\in\Fcal^2_\infty\;\colon\;A\cap\{t-\tau_1^\ast\le s\}\in\Fcal_s^2\;\text{for all }s\ge 0\right\},\]
and where $\Fcal^2_\infty:=\sigma\left(\bigcup_{t\ge 0}\Fcal^2_t\right)$.
Let $\overline{\mathfrak{A}}_2:=(\overline{\Omega}_2,\overline{\CF}_2,\overline{\mathbb{F}}_2,\overline{\PP}_2)$.
Finally, let us set for $j=1,2$
$$\bar{{W}}_j^2(t):=\bcase \bar{{W}}^1_j(t), & \mbox{if} \quad t<\tau_1^\ast,
\\   {W}^2_j({t-\tau_1^\ast})+\bar{{W}}^1_j(\tau_1^\ast), & \mbox{if}\quad  t\ge \tau_1^\ast.
\ecase
$$
which give independent Wiener processes for $j=1,2$, w.r.t. the filtration $\overline{\mathbb{F}}_2$.

Now, let us define two  processes $\bar u_2 $ and $\bar v_2$
being identical to $\bar u_1$ and $\bar v_1$, respectively,  on the time interval $[0,\tau^\ast_1)$, being identical to $u_2$ and $v_2$ on the time interval $[\tau^\ast_1,\tau^\ast_1+\tau_2^\ast)$ and
following the porous media, respective, the heat equation  (with lower order terms) with multiplicative noise, afterwards.
Let us note, for any initial condition having distribution equal to $u_2(\tau_2^\ast)$ and $v_2(\tau^\ast_2)$ that there exists a strong solutions  $y_1(\cdot ,\cdot ,\tau^\ast_2+\tau^\ast_1)$ and $y_2(\cdot,\cdot ,\tau^\ast_2+\tau^\ast_1)$ of the systems \eqref{eq1} and \eqref{eq2}, respectively, on $\overline{\mathfrak{A}}_2$.
Let for $t\in[0,T]$
$$
\bar u_2  (t) = \bcase \bar u_1(t) & \mbox{ for } 0\le t< \tau^\ast_1,\\
  u_2(t-\tau_1^\ast ) & \mbox{ for }  \tau^\ast_1\le t\le \tau^\ast_1+\tau_2^\ast,\\
y _1(t-(\tau^\ast_1+\tau^\ast_2),u_2(\tau^\ast_2),\tau^\ast_1+\tau^\ast_2) & \mbox{ for } \tau^\ast_2+\tau^\ast_1\le  t \le T,\ecase
$$
$$
\bar v_2  (t) = \bcase v_1 (t) & \mbox{ for } 0\le t< \tau_1^\ast,\\
  v_2(t-\tau_1^\ast ) & \mbox{ for } \tau^\ast_1\le t\le \tau^\ast_1+\tau_2^\ast,\\
y_2 (t-(\tau^\ast_1+\tau^\ast_2),v_2(\tau^\ast_2),\tau_1^\ast+\tau^\ast_2
)  & \mbox{ for } \tau_1^\ast+\tau^\ast_2 \le  t \le T.\ecase
$$
In the same way we will construct for any $\kappa \in\NN$ a probability space $\overline{\mathfrak{A}}_\kappa $, a pair of independent Wiener processes $(\bar{{W}}^1_\kappa ,\bar{{W}}^2_\kappa )$, over $\overline{\mathfrak{A}}_\kappa $ and a pair of processes $(\bar{u}_\kappa,\bar{v}_\kappa)$ starting at $(u_0,v_0)$ and solving system  \eqref{eq:cutoffu}-\eqref{eq:cutoffv} up to time

\begin{align} \label{stopp_time}\bar \tau_\kappa :=\tau_1^\ast+\cdots +\tau_\kappa ^\ast
\end{align}  and following  the porous media, respective, the heat equation  afterwards.
 Besides, we know that
  $(\bar u_{\kappa} ,\bar v_{\kappa} )\vert_{[0,\bar \tau _{\kappa -1})}=(\bar u_{\kappa -1},\bar v_{\kappa -1})\vert_{[0,\bar \tau_{\kappa -1})}$.

\item {\bf Uniform bounds on the stopping time.}

Let us consider the family $\{ (\bar{u}_\kappa,\bar{v}_\kappa)\;\colon\; \kk\in\NN\}$. The next aim is to show that there exists $\Reins,\Rzwei$ and $\Rdrei>0$ independent of $\kk\in\NN$ such that
$\{ (\bar{u}_\kappa,\bar{v}_\kappa)\;\colon\; \kk\in\NN\}\subset \subX$.

First, that due to Proposition \ref{uniform1}  there exists a constant $C_0(1,T)>0$ such that
\begin{align} \label{est2}
\\\nonumber
\lqq{
\EE \left[\sup_{0\le s\le \bar\tau_\kappa\wedge T} \vert\buk (s)\vert_{L^{2} }^{2}\right]
+ 2\gamma r_u\EE \int_0^ {\bar\tau_\kappa\wedge T}\int_\CO  \vert\buk (s,x )\vert^{\gamma-1}\vert\nabla \buk (s,x)\vert^2\, dx \, ds} &
\\\nonumber &{} +2\chi\EE \int_0^{\bar\tau_\kappa \wedge T} \int_\CO  \vert\buk(s,x )\vert^2\vert \bvk (s,x )\vert^2\, dx \, ds \le C_0(1,T)\,\left( \EE\vert u_0\vert_{L^{2}}^{2}+1\right).
\end{align} 
From above we can conclude that
there exists a constant $C>0$ such that
\begin{align*} 
\EE \Vert\buk\mathbbm{1}_{[0,\bar\tau_\kappa\wedge T]}\Vert_{\BY}^2\le C.
\end{align*} 
Here, it is important that
 $\buk\ge 0$.
This estimate can be extended to the time interval $[0,T]$ by standard results (see e.g.\ \cite{BDPR2016} and  Proposition \ref{uniform1}).
{Next, let us assume  $\Rzwei>0$ is that large that
\begin{align} \label{constantR2}
C(T)\left(
 \EE\vert v_0\vert_{H^{-\delta_0}_{m}}^{m_0}+ \Reins^{\delta_1}\,
\Rdrei^{\delta_1}\right)\le \Rdrei,
\end{align} 
where the constants $C(T)$, $\delta_0$, $\delta_1$ are given in Proposition \ref{propvarational2}.
Observe, by Proposition \ref{propvarational2},
we have for $l$ as in Hypothesis \ref{init}
\begin{align}  \label{constantR22}\hspace{+6ex}
\EE  \Vert\bvk\Vert_{\mathbb{Z}}^{m_0}\le C(T)\left(
 \EE\vert v_0\vert_{H^{-\delta_0}_{m}}^{m_0}+\Bigg(\EE \left(\int_0^{\bar\tau_\kk \wedge T}  \uk^{2}(s)\vk^2(s)\,ds
\right)^l\Bigg)^{\delta_1}\,
\Rzwei^{\delta_1}\right).
\end{align} 
Then,  by Proposition \ref{uniform1}, we know that
for $p=1$,
 the term
$$
\EE \left(\int_0^{\bar\tau_\kk \wedge T}  \uk^{2}(s)\vk^2(s)\,ds
\right)^l\le\EE \vert u_0\vert_{L^{2}}^l,
$$
which can be uniformly bounded by $\Reins^l$. Hence, we conclude by \eqref{constantR2} that $\EE  \Vert\bvk\Vert_{\BZ}^{m_0}\le \Rzwei$
and the choice of $\Rdrei$ that $\EE  \Vert\bvk\Vert_{\BH}^{m_0}\le \Rdrei$.
Finally, by Proposition \ref{interp_rho} we know for $\frac d2 -\rho\le \frac 2m+\frac d {m_0}$ that
$\BH\hookrightarrow \mathbb{Z}$. Due to the choice of $m$ and $m_0$, we know that the inequality is satisfied.
Hence, setting
\begin{align} \label{constantR3}
\Rzwei\ge&\Rdrei,
\end{align} 
we know that
$\EE  \Vert\bvk\Vert_{\mathbb{Z}}^{m_0}\le \Rzwei$.
In particular, there exists $\Reins,\Rzwei$ and $\Rdrei$ such that for all $\kk\in\NN$  we have
\begin{equation}\label{constantR1}\EE\left[\sup_{0\le s\le T}\vert\bar u_\kappa(s)\vert_{L^{p^\ast}}^{p^\ast_0}\right]\le \Reins^{p_0^\ast},
\end{equation}
  $\EE \Vert\bvk\Vert_{\BZ}^{m_0}\le \Rzwei$, and
$\EE  \Vert\bvk\Vert_{\BH}^{m_0}\le \Rdrei$. 
}

}

\item {\bf Passing on to the limit.}

In the final step, we will show that $\PP$-a.s. a martingale  solution to \eqref{eq:uIto}-\eqref{eq:vIto} exists.

\begin{claim} Let
\[(\Omega,\Fcal,\mathbb{F},\P):=\left(\prod_{\kappa\in\N}\overline{\Omega}_\kappa,\bigotimes_{\kappa\in\N}\overline{\Fcal}_\kappa,\overline{\mathbb{F}}_\infty,\bigotimes_{\kappa\in\N}\overline{\P}_\kappa\right),\]
where
$\overline{\mathbb{F}}_\infty:=(\overline{\CF}^\infty_t)_{t\in [0,T]}$, where, setting $\tau^\ast_0:=0$,
$$\overline{\CF}^\infty_t:= \bigcap_{\kappa\in\N}\overline{\CF}^\kappa_t.
$$
There exists a measurable set $\Omega_0\subset \Omega$ with $\PP(\Omega_0)=1$ such that a martingale solution $(u,v)$ to the
system \eqref{eq:uIto}-\eqref{eq:vIto} exists on $(\Omega_0,\Fcal\cap\Omega_0,\mathbb{F}\cap \Omega_0,\P\vert_{\Omega_0},(\bar{W}_1^\infty\vert_{\Omega_0},\bar{W}_2^\infty\vert_{\Omega_0}))$,
where for $j=1,2$, setting $\tau^\ast_0:=0$,
$$\bar{W}_j^\infty(t):= \bar{W}_j^\kappa(t), \text{\;if} \quad t\in [\tau_{\kappa-1}^\ast,\tau_\kappa^\ast),\quad\kappa\in\N.
$$
\end{claim}
\begin{proof}
For any $\kappa\in\NN$,
let us define the set
		$$A_\kappa :=
		\left\{ \omega\in \Omega\;\colon\;\bar\tau_\kappa(\omega)\ge T\right\}.
		$$
		It can be clearly observed that
		there exists a progressively measurable process $(u,v)$ over $\mathfrak{A}$
such that $(u,v)$ solves $\PP$--a.s. the integral equation given by
	\eqref{eq:uIto}-\eqref{eq:vIto} up to time $T$. In particular, we have for the conditional probability
		$$
		\PP\left( \{ \mbox{there exists a solution $(u,v)$ to \eqref{eq:uIto}-\eqref{eq:vIto}} \} \mid A_\kappa\right)=1.
		$$
		Set $\Omega_0=\bigcup_{\kappa=1}^{\infty} A_\kappa$. Then, it is elementary to verify that
		\begin{align*} \lqq{
			\qquad\PP\left( \left\{ \mbox{there exists a solution $(u,v)$ to \eqref{eq:uIto}-\eqref{eq:vIto}}
			\right\}\cap\Omega_0\right) }\\
			=& \lim_{\kappa\to \infty} \PP\left(  \{ \mbox{there exists a solution $(u,v)$ to \eqref{eq:uIto}-\eqref{eq:vIto}} \} \cap A_\kappa \right)\\
		=& \lim_{\kappa\to \infty} \left[\PP\left(  \{ \mbox{there exists a solution $(u,v)$ to \eqref{eq:uIto}-\eqref{eq:vIto}} \} \mid A_\kappa \right)\PP\left( A_\kappa\right)\right].
		\end{align*} 
		Since $ \PP\left(  \{ \mbox{a solution $(u,v)$ to \eqref{eq:uIto}-\eqref{eq:vIto} exists} \} \mid A_\kappa\right)=1$, it remains to show that
		${\lim_{\kappa\to\infty}}$ $\PP( A_\kappa)=1$. Then, as $A_\kappa\subset A_{\kappa+1}$ for $\kappa\in\NN$, it follows that
		$$\PP\left( \left\{ \mbox{there exists a solution $(u,v)$ to \eqref{eq:uIto}-\eqref{eq:vIto}}
		\right\}\cap\Omega_0\right)=1.
		$$
		However, due to Step IV, there exists a constant $C(T)>0$ such that
		$$
		\EE\left[ h(\buk,\bvk,t)\right]\le C(T),\quad t\in[0,T],\;\kk\in\NN,
		$$
		thus by the Markov inequality,
		$$\PP\left(\Omega\setminus  A_\kappa\right) \le \frac {C(T)}{\kappa }\to 0,
$$ as $\kappa\to\infty$.
		Thus the solution process is well-defined on $\Omega_0=\bigcup_{\kappa=1}^{\infty} A_\kappa$ with $\PP(\Omega_0)=1$.
		\end{proof}

\end{steps}

The proof of Theorem \ref{mainresult} is complete.

\end{proof}

\section{Proof of the stochastic Schauder-Tychonoff theorem}\label{sec:proof_of_schauder}

\begin{proof}[Proof of Theorem \ref{ther_main}]

Fix
$\Afrak$ and $W$, and $R_1,\ldots,R_K>0$, $K\in\N$. In addition, for simplification, we shall omit $R_1,\ldots,R_K$ in the notation for $ \Xcal_{R_1,\ldots,R_K}(\Afrak)$ and write $ \Xcal(\Afrak)$ instead of $ \Xcal_{R_1,\ldots,R_K}(\Afrak)$. Fix an initial datum $w_0\in L^m(\Omega,\Fcal_0,\mathbb{P};E)$.
\begin{step}
\item

In the first step we will approximate the operator $\mathcal{V}_{\Afrak,W}$. We shall discretize time, as we would like apply the classical Schauder-Tychonoff theorem in a compact subset which is given by a tight collection of laws on a finite time grid. 
Let us fix a sequence $\{\ep_\iota:\iota\in\N\}$ such that $\ep_\iota\to 0$.

First, let us introduce a dyadic time grid $\pi_n=\{t_0=0<t_1<t_2<\cdots <t_{2^n}=T\}$
by $t_k= T\frac{k}{ 2^n}$, $k=0,1,2,\ldots,2^n$.
The stochastic process will be approximated by an averaging operator over the dyadic time interval.
To this end, let us define a step-function
$\phi_n:[0,T]\to[0,T]$ by $\phi_n(s)=T\frac{k}{ 2^n}$, if
$k=0,1,2,\ldots,2^n-1$ and $T\frac{k}{ 2^n}\le s<T\frac{k+1}{2^n}$, i.e.,
$\phi_n(s)= T2 ^ {-n}\lfloor 2 ^ ns\rfloor$, $s\ge 0$, where $\lfloor t\rfloor$ is the largest
integer that is less or equal $t\in \mathbb{R}$. Let $\{w_n:n\in\NN\}\subset L^m(\Omega,\Fcal_0,\mathbb{P};E)$ be a sequence, such that $w_n\to w_0$ in $L^m(\Omega,\Fcal_0,\mathbb{P};E)$ and $$\Vert w_n-w_0\Vert_{L^m(\Omega,\Fcal_0,\mathbb{P};E)}\le \frac {\ep_n}n.$$

For a function $\xi\in\Mcal_{\Afrak}^{m}(\X)$, we define
\newcommand{\Pro}{\mbox{Proj}}
\begin{eqnarray}\label{hatdefined} \Pro_n(\xi)(s):=
\bcase w_n, &
\mbox{ if } s\in [0,T2^{-n}),
\\
 \frac{2^n}{T}\int_{\phi_n(s)-T2^{-n}}^{\phi_n(s)} \xi(r)\: dr, &\mbox{ if
} s\geq T2^{-n}.\ecase
\end{eqnarray}
Note, that $\Pro_n(\xi)$ is a
progressively measurable, $\P$-a.s. piecewise constant, $U$-valued stochastic process.
\begin{remark}\label{projection}
Observe that the projection operator satisfies
\begin{enumerate}[(i)]
\item $\Pro_n$ is a linear bounded contraction operator from $\mathbb{X}$ into $\mathbb{X}$;
\item If $B$ is a bounded subset of $\mathbb{X}$,
then for all $\ep>0$ there exists a $n_0\in\NN$ such that
$$\Vert \Pro_n (\xi) - \xi\Vert_{\mathbb{X}}<\ep,\quad \xi\in B,\quad n\ge n_0,
$$
\end{enumerate}
cf. \cite[Appendix B]{BHM}.
\end{remark}
Thus, due to Remark \ref{projection} and the uniform continuity of $\mathcal{V}_{\mathfrak{A},W}$ on the bounded set $\mathcal{X}(\mathfrak{A})$, we know that for any $\ep>0$, there exists some $n_0\in\NN$ such that for every $\xi$ uniformly in
$\mathcal{V}_{\mathfrak{A},W}\left( \mathcal{X}(\mathfrak{A})\right)$ such that for any $r\in (1,m]$,
$$
\left( \EE\Vert\Pro_n(\xi)-\xi\Vert^r_{\mathbb{X}}\right)^\frac 1r \le  \ep ,\quad \forall \xi\in \mathcal{V}_{\mathfrak{A},W}\left( \mathcal{X}(\mathfrak{A})\right),\quad \forall \, n\ge n_0.
$$
Let $\{n_\iota:\iota\in\NN\}$ be a sequence such that
\begin{equation}\label{eq:pointwiseeps}
\left( \EE \Vert\Pro_{n_\iota}(\xi)-\xi\Vert_{\mathbb{X}}^{m}\right)^\frac{1}{m}\le \frac {\ep_\iota} 3,\quad \forall \xi\in \mathcal{V}_{\mathfrak{A},W}\left( \mathcal{X}(\mathfrak{A})\right).
\end{equation}
Finally, let us define the operator
$$
\mathcal{V}^\iota _{\mathfrak{A},W}(\xi):=  (\Pro_{n_\iota} \circ\mathcal{V}_{\mathfrak{A},W})(\xi),\quad \xi\in \mathcal{X}(\mathfrak{A}).
$$

\item
Denote $\mathbb{U}:=\mathbb{D}([0,T];U)$, the Skorokhod space of c\`adl\`ag paths in $U$ endowed with the Skorokhod $J_1$-topolgy, see \cite[Appendix A2]{Kallenberg}. Given the probability space $\mathfrak{A}=\left(\Omega,\CF,\mathbb{F},\PP\right)$,
for any $\iota \in\NN$ this operator $\mathcal{V}^\iota _{\mathfrak{A},W}$
induces  an operator $\mathscr{V}_\iota$ on the set of Borel probability measures on $\mathbb{X}\cap\mathbb{U}$, denoted by  $\mathscr{M}_1(\mathbb{X}\cap\mathbb{U})$.
The construction of the operator $\mathscr{V}_\iota$ is done in the this step.

Define
the subset of probability measures $\mathscr{X}$ given by
\begin{align*}
\lqq{ \mathscr{X}:=\left\{ \mu\in \mathscr{M}_1(\mathbb{X}\cap\mathbb{U}):\int_{\mathbb{X}}\Psi_j(\xi)\,\mu(d\xi)\le R_j,\right.
}
\\
&& \left. \phantom{\int_{\mathbb{X}}}
\qquad
\mbox{ and } \mu\left(\{\Theta_j<\infty\}\right)=1\, \quad \forall j=1,\ldots,K\right\}.
\end{align*}
Now, let $\mu\in\mathscr{X}$. Then, by the Skorokhod lemma \cite[Theorem 4.30]{Kallenberg}, we know that there exists a
probability space  $\mathfrak{A}_0=(\Omega_0,\CF^0,\PP_0)$ and a random variable
$\xi:\Omega_0\to \mathbb{X}\cap\mathbb{U}$ such that the law of $\xi$ coincides with $\mu$.
In particular,
the probability measure $\nu_\xi:\CB(\mathbb{X}\cap\mathbb{U})\to[0,1]$ induced by $\xi$ and given by
$$
\nu_\xi:\CB(\mathbb{X}\cap\mathbb{U})\ni A\mapsto \PP_0\left( \left\{ \omega:\xi(\omega)\in A\right\}\right)
$$
coincides with the probability measure $\mu$.

Due to the definition of $\mathbb{U}$, we know that $\xi$ is a progressively measurable stochastic process, in particular, $\xi:\Omega_0\times [0,T]\to U$ such that $\PP_0(\xi\in\mathbb{X}\cap\mathbb{U})=1$.
Let
				$$
			 \CG^0_t:=\sigma \left( \left\{\,\xi (s)\;\colon\; 0\le s\le t \, \right\}\cup \CN_0\right),\quad t\in [0,T],
				$$
				where $\CN_0$ denotes the collection of zero sets of ${\mathfrak{A}}_0$.
				Set $\mathfrak{A}_0:=(\Omega_0,\CF^0,(\CG_t^0)_{t\in[0,T]},\PP_0)$.

Next, we have to construct the  Wiener process and extend the probability space.
				 Now, let $\mathfrak{A}_1=(\Omega_1,\CF^1,(\CF_t)_{t\in[0,T]},\PP_1)$ be a probability space
				where a cylindrical Wiener process $W$ on $H$ is defined,
				and let $\mathfrak{A}_{\mu}$ the product probability space of $\mathfrak{A}_0$ and $\mathfrak{A}_1$.
				In particular, we set
\begin{align*}
&\Omega_\mu =  \Omega_0\times \Omega_1,
\\
&\CF_\mu = \CF^0\otimes \CF^1,
\\
&\CG^\mu_t = \CG^0_t\otimes \CG^1_t,\quad t\in [0,T],
\\
\mbox{and}\quad &\PP_\mu = \PP_0\otimes \PP_1.
\end{align*}
Here, we know that $\xi(t)$, $t\in [0,T]$, is independent of the increments $W(t')-W(t)$, $t'>t$ and $\{\CG_t^\mu\}_{t\in [0,t]}$-progressively measurable.
Since $\mu\in\mathscr{X}$, we know $\xi\in\mathcal{X}(\mathfrak{A}_\mu)$.

Next, we have to verify if the family    operators
$$
\left\{ \Vcal^\iota_{\mathfrak{A}_\mu,{W}}: \iota\in\NN\right\}
$$
is well-defined. However, this is follows from assumption (i),
and since $\xi\in\mathcal{X}(\mathfrak{A}_\mu)$. In fact, the $\mu$-dependence of the stochastic basis can be removed by lifting to the space of probability measures (path laws). We aim to find a fixed point in the space of probability measures.

Now, for $A\in\CB(\mathbb{X}\cap\mathbb{U})$. Then, let us define the mapping $\mathscr{V}_\iota$ that maps
the probability measure $\mu$, in other words, the probability measure $\nu_\xi:\CB(\mathbb{X}\cap\mathbb{U})\to[0,1]$ that is induced by $\xi$ to the probability measure $\nu_{\mathcal{V}_{\Afrak_\mu,W}^\iota(\xi)}:\CB(\mathbb{X}\cap\mathbb{U})\to[0,1]$ given by
$$
\nu_{\mathcal{V}_{\Afrak_\mu,W}^\iota(\xi)}(A):=\PP_\mu\left( \left\{ \omega\in\Omega_\mu: \Vcal^\iota_{\mathfrak{A}_\mu,W}(\xi(\omega))\in A\right\}\right),\quad \mathscr{V}_\iota(\mu):=\nu_{\mathcal{V}_{\Afrak_\mu,W}^\iota(\xi)}.
$$
Note, since $\mathbb{X}\cap\mathbb{U}$ is a complete metric space,
the space of probability measures over $\mathbb{X}\cap\mathbb{U}$ equipped with the
Prokhorov metric\footnote{Let $\mathscr{M}_1(X)$ be the set of Borel probability measures on the metric space $(X,d)$ equipped with the weak topology. Let $\nu,\mu\in\mathscr{M}_1(X)$. Then the weak topology can be metrized by the \emph{Prokhorov metric}, cf. \cite{dudley2002},
$$d_\alpha(\mu,\nu):=\inf\{\alpha>0: \mu(A)\le \nu(A_\alpha)+\alpha \mbox{ and }
\nu(A)\le \mu(A_\alpha)+\alpha \mbox{ for all } A\in\CB(X)\}.
$$
Here $A_\alpha:=\{ x\in X: d(x,A)<\alpha$\}.}   is complete.

The following points can be easily verified.
\begin{enumerate}[(1)]
\item  $\mathscr{X}$  is invariant under $\mathscr{V}_\iota$.
This follows directly from assumption (ii)
and the properties of the projection $\Pro_{n_\iota}$, see also Remark \ref{projection}.
\item Due to the fact that $\Vcal^\iota_{\Afrak_{\mu},W}$ restricted to $\mathcal{X}(\Afrak_\mu)$ is uniformly continuous, $\mathscr{V}_\iota$ restricted to $\mathscr{X}$ is continuous on $\mathscr{M}_1(\mathbb{X}\cap\mathbb{U}) $ in the Prokhorov metric
by \cite[Theorem 11.7.1]{dudley2002}.
\item Note that by assumption (v), $\mathcal{V}^\iota_{\mathfrak{A}_\mu,W}(\xi)\in \mathbb{U}$ for $\xi\in\mathcal{X}(\mathfrak{A}_\mu)$. We claim that $\mathscr{V}_\iota$ restricted to $\mathscr{X}$ is compact on $\mathscr{M}_1(\mathbb{X}\cap\mathbb{U})$. 
In particular, it maps bounded sets into compact sets.
In fact, we have to show that for all $\iota\in\NN$ and $\ep>0$ there exists a compact subset $K_\ep\subset\mathbb{X}\cap\mathbb{U}$ such that
$$
\nu_ {\mathscr{V}_\iota(\xi)}\left ((\mathbb{X}\cap\mathbb{U})\setminus K_\ep\right)< \ep, \quad \forall \nu_\xi\in \mathscr{X}\quad \mbox{and}\quad \nu_ {\mathscr{V}_\iota(\xi)}:=\mathscr{V}_\iota(\nu_\xi).
$$
However, by assumption (iv)
there exists a constant $R>0$ with
$$
\EE\Vert\mathcal{V}^\iota_{\mathfrak{A}_\mu,W}(\xi)\Vert_{\mathbb{X}'}^{m_0}\le R,\quad \xi\in\mathcal{X}(\mathfrak{A}_\mu).
$$
Let $\tilde{R}>R^{1/{m_0}}\ep^{-1/{m_0}}$
and let $K_\ep:=\{x\in\mathbb{X}\cap\mathbb{U}:\Vert x\Vert_{\mathbb{X}'}\le \tilde{R}\}$.
Due to the construction of the operator $\mathscr{V}_\iota$, we have
that the law is preserved. In particular,
$$
\PP_\mu\left(  \left\{
x\in\mathbb{X}\cap\mathbb{U}\cap \mathcal{X}(\mathfrak{A}_\mu) :\Vert\mathcal{V}_{\mathfrak{A}_\mu,W}^\iota(x)\Vert_{\mathbb{X}'} \ge \tilde{R}\right\} \right)
=\nu_ {\mathscr{V}_\iota(\xi)}\left ( { \left\{
x\in\mathscr{X}:\Vert x\Vert_{\mathbb{X}'} \ge \tilde{R}\right\}}\right).
$$
Next, by Chebyshev's inequality, we get that
$$\nu_{ \mathscr{V}_\iota(\xi)}\left ((\mathbb{X}\cap\mathbb{U})\setminus K_\ep\right)
=
\nu_ {\mathscr{V}_\iota(\xi)}\left ({ \left\{
x\in\mathscr{X}:\Vert x\Vert_{\mathbb{X}'} \ge \tilde{R}\right\}}\right)
< \ep.$$
Since $\mathbb{X}'\hookrightarrow \mathbb{X}$ compactly,
we have proved the tightness.
\item $\mathscr{X}$ is a convex subset of $\mathscr{M}_1(\mathbb{X}\cap\mathbb{U})$.
Let $\nu,\mu\in\mathscr{X}$, we have to show that for any $\alpha\in(0,1)$ we have $\alpha \nu+(1-\alpha)\mu\in\mathscr{X}$. 
First, analyzing the expectation with respect to $\Psi_1,\ldots,\Psi_K$, this follows by the linearity of the expectation value. Secondly, since $\nu,\mu\in\mathscr{X}$
we know that $\nu\left(\{\Theta<\infty\}\right)=1$ and $\mu\left(\{\Theta<\infty\}\right)=1$, Let $\alpha\in(0,1)$. Then
\begin{align*}
&\left(\alpha \nu+(1-\alpha)\mu\right)\left(\{\Theta<\infty\}\right)
\\
&=\alpha \underbrace{\nu\left(\{\Theta<\infty\}\right)}_{=1}
+(1-\alpha)\underbrace{\mu\left(\{\Theta<\infty\}\right)}_{=1}=1.
\end{align*}
\end{enumerate}
In particular, the mapping $\mathscr{V}_\iota$ restricted to $\mathscr{X}$
satisfies all assumptions of the classical Schauder-Tychonoff theorem, see \cite[§ 7, Theorem 1.13,. p. 148]{granas}:
\begin{lemma}[Schauder-Tychonoff]
Let $\mathcal{C}$ be a nonempty convex subset of a locally convex linear topological space $\mathcal{E}$, and let $F:\mathcal{C}\to\mathcal{C}$ be a compact map, i.e., $F(\mathcal{C})$ is contained in a compact subset of $\mathcal{C}$. Then $F$ has a fixed point.
\end{lemma}

Hence, for any $\iota\in\NN$ there exists a probability measure $\nu^\ast_\iota\in\mathscr{X}$ such that $$\mathscr{V_\iota}(\nu^\ast_\iota)=\nu^\ast_\iota.$$

				\item
Note, that the tightness argument in the previous step is independent of $\iota$, thus the set
$$
\left\{ \nu^\ast_\iota:\iota\in\NN\right\}
$$				
is tight, therefore there exists a subsequence $\{\iota_j:j\in\NN\}$ and a Borel probability measure $\nu^\ast$ such that
$ \nu_{\iota_j}^\ast\to 	 \nu^\ast$, as $j\to\infty$.	
In this step, we will construct from the family of probability measures $\{\nu^\ast_{\iota_j}:j\in\NN\}$ and $\nu^\ast\in\mathscr{X}$, a filtered probability space $\mathfrak{A}^\ast$, a Wiener process
${W}^\ast$, a
progressively measurable process $w^\ast$, and
a family of progressively measurable processes  $\{ w_{\iota_j}:j\in\NN\}$ that
			are $\PP^\ast$-a.s. contained in $\mathbb{X}\cap\mathbb{U}$
	  over $\mathfrak{A}^\ast$ such that these objects have probability measures $\{\nu^\ast_{\iota_j}:j\in\NN\}$ and $\nu^\ast_{\iota_j}\in\mathscr{X}\cap\mathbb{U}$.
	
By the Skorokhod lemma \cite[Theorem 4.30]{Kallenberg},
there exists a probability space $\mathfrak{A}^\ast_0=(\Omega^\ast_0,\CF^{\ast}_0,\PP^\ast_0)$ and a  sequence of $\mathbb{X}$-valued random variables $\{ {w}^\ast_{\iota_j}:j\in\NN\}$
and ${w}^\ast_{\iota_j}$ where
the random variable $w^\ast_{\iota_j}:\Omega^\ast_0\to\mathbb{X}\cap\mathbb{U}$ has the  law $\nu^\ast_{\iota_j}$ in $\mathbb{X}\cap\mathbb{U}$.
In addition, by tightness and the Skorokhod lemma, we have
\begin{equation}\label{eq:PwwP}
{w}^\ast_{\iota_j}\to{w}^\ast\quad \mbox{ as $j\to \infty$ }\quad  {\P}^\ast_0\text{-a.s.}
\end{equation}
on $\X$.
Moreover, let us introduce the filtration $\mathbb{G}^\ast_0=(\mathcal{F}_t^{\ast,0})_{t\in[0,T]}$  given by
				$$
			 \CF^{\ast,0}_t:=\sigma \left( \left\{\,(w _{\iota_j}^\ast(s),w^\ast(s))\;\colon\; 0\le s\le t, \,j\in\NN \right\}\cup \CN_0^\ast\right),\quad t\in [0,T],
				$$
				where $\CN_0^\ast$ denotes the collection of zero sets of ${\mathfrak{A}}_0^\ast$.
				
Next, similarly as above, let us construct the Wiener process.
Let 
$$
\mathfrak{A}_1^\ast=\left(\Omega_1^\ast,\CF_1^\ast,(\mathcal{G}_t^{1,\ast})_{t\in [0,T]},\PP_1^\ast\right).
$$
be a filtered probability space with a cylindrical Wiener process $W^\ast$ on $H$ being adapted to the filtration $(\mathcal{G}_t^{1,\ast})_{t\in [0,T]}$. 				Let $\mathfrak{A}^\ast:=\mathfrak{A}_0^\ast\times \mathfrak{A}_1$.	
				In particular, we put
\begin{align*}
&\Omega ^\ast=  \Omega_0^\ast\times \Omega^\ast_1,
\\
&\CF^\ast = \CF_0^\ast\otimes \CF_1^\ast,
\\
&\CG^\ast_t = \CG^{0,\ast}_t\otimes \CG^{1,\ast}_t,\quad t\in [0,T],
\\
\mbox{and}\quad &\PP^\ast = \PP^\ast_0\otimes \PP_1^\ast.
\end{align*}
In addition, $\mathcal{X}(\mathfrak{A}^\ast)$ can be defined, and also the operators $\mathcal{V}^\iota_{\mathfrak{A}^\ast,W^\ast}$
and ${\mathcal{V}}_{\mathfrak{A}^\ast,W^\ast}$
				
				\item
				
				In this step, we mimic an explicit Euler scheme, to construct a $\P^\ast$-a.s. piecewise constant and $\{\CG^\ast_t\}_{t\in [0,T]}$-progressively measurable process that is a fixed point for the operator $\mathcal{V}_{\mathfrak{A}^\ast,W^\ast}^{\iota_j}$.
			
Since $\nu^\ast_{\iota_j}\in \mathscr{X}$, the process $w^\ast_{\iota_j}\in\mathcal{X}(\mathfrak{A}^\ast)$ and, hence, 
 $\mathcal{V}_{\mathfrak{A}^\ast,W^\ast}^{\iota_j}(w^\ast_{\iota_j})$ is well-defined.
 Since $\mathscr{V}_{\iota_j}(\nu^\ast_{\iota_j})=\nu^\ast_{\iota_j}$, $\Law(w^\ast_{\iota_j})=\nu^\ast_{\iota_j}$.
However, we do not know if
 the process $w^\ast_{\iota_j}$ satisfies
\begin{align*}
\PP^\ast\left(
\mathcal{V}_{\mathfrak{A}^\ast,W^\ast}^{\iota_j}\left(w^\ast_{\iota_j}\right) (s) =w^\ast_{\iota_j}(s)\right)=&1 \quad \mbox{for} \quad 0\le s\le T
.
\end{align*}

In this step, we will construct here a fixed point for the operator $\mathcal{V}_{\mathfrak{A}^\ast,W^\ast}^{\iota_j}$.
Let us define a new process by induction. To start with, let
\begin{equation}\label{nummer1}
w^\ast_{\iota_j,1}(s) :=  \begin{cases}
w^\ast_{\iota_j}(s)& \mbox{ if } 0\le s<t_1,
\\
\left(\mathcal{V}_{\mathfrak{A}^\ast,W^\ast}^{\iota_j}(w^\ast_{\iota_j})\right)(s)& \mbox{ if } t_1\le s\le T,
\end{cases}
\end{equation}
and
\begin{equation}\label{nummer2}
w^\ast_{\iota_j,2}(s) :=\begin{cases}
w^\ast_{\iota_j,1}(s)& \mbox{ if } 0\le s<t_2,
\\
\left(\mathcal{V}_{\mathfrak{A}^\ast,W^\ast}^{\iota_j}(w^\ast_{\iota_j,1})\right)(s)& \mbox{ if } t_2\le s\le T.
\end{cases}
\end{equation}
Now, having defined $w^\ast_{\iota_j,k}$, let
\begin{equation}\label{nummerk}
w^\ast_{\iota_j,k+1}(s) :=  \begin{cases}
w^\ast_{\iota_j,k}(s)& \mbox{ if } 0\le s<t_{k+1},
\\
\left(\mathcal{V}_{\mathfrak{A}^\ast,W^\ast}^{\iota_j}(w^\ast_{\iota_j,k})\right)(s)& \mbox{ if } t_{k+1}\le s\le T,
\end{cases}
\end{equation}
where $t_k\in\pi_n$ are dyadic time points.
Let us put $  w^\ast_{\iota_j,0}(s)=w_0^\ast$ for $0\le s\le T$, where $ w_0^\ast$ is a $\Gcal^\ast_0$-measurable version of $w_0$, and
\begin{equation}\label{definfty}
w_{\iota_j,\infty}^\ast(s) :=   w^\ast_{\iota_j,k}(s),\quad \mbox{if}\quad t^{\iota_j}_{k-1}\le s< t^{\iota_j}_{k},\,\,k=1,\ldots, 2^{\iota_j}.
\end{equation}
We claim that the process $w^\ast_{\iota_j,\infty}$ satisfies
\begin{equation}\label{isasolution}
\PP^\ast\left(
\mathcal{V}_{\mathfrak{A}^\ast,W^\ast}^{\iota_j}\left(w^\ast_{\iota_j,\infty}\right) (s) =w^\ast_{\iota_j,\infty }(s)\right)=1 \quad \mbox{for} \quad 0\le s\le T
.
\end{equation}
Note, that
by the definition of $\Pro_{\iota_j}$ on $[0,t_1^{\iota_j})$
the process on $[0,t_{1}^{\iota_j})$ is defined by the initial data.
In fact, we have
for $0\le s< t^{\iota_j}_1$
$$
\mathcal{V}_{\mathfrak{A}^\ast,W^\ast}^{\iota_j}\left(w^\ast_{\iota_j,\infty}\right)(s)= w^\ast_0.
$$
By equation \eqref{definfty} and equation \eqref{nummer1} we have $w^\ast_{\iota_j,\infty}=w^\ast_{\iota_j,1}=w^\ast_0$ 
for $0\le s< t^{\iota_j}_1$.
In particular, the process on $[0,t_{1}^{\iota_j})$ is defined by the initial data
and we have  $\PP^\ast$-a.s.
$$\mathcal{V}_{\mathfrak{A}^\ast,W^\ast}^{\iota_j}\left(w^\ast_{\iota_j,\infty}\right)(s)=w^\ast_{\iota_j,\infty}(s)
, \quad \mbox{for} \quad 0\le s<t^{\iota_j}_1.
$$
At time $t^{\iota_j}_1$, we have by equation \eqref{definfty}
and equation \eqref{nummer1}, 
$$
\mathcal{V}_{\mathfrak{A}^\ast,W^\ast}^{\iota_j}\left(w^\ast_{\iota_j,\infty }\right)(t^{\iota_j}_1)=\left(\mathcal{V}_{\mathfrak{A}^\ast,W^\ast}^{\iota_j}  w^\ast_{\iota_j}\right)(t^{\iota_j}_1)=w^\ast_{\iota_j,1}(t^{\iota_j}_1).
$$
However, we have $\tilde w^\ast_{\iota ,1}(t^{\iota_j} _1)= w^\ast_{{\iota_j },\infty}(t^{\iota_j}_1)$.
Let us analyze what is happening at the next time interval $[t^{\iota_j }_1,t^{\iota_j }_2)$. Her the process is constant and equals   $\PP^\ast$-a.s. the value at $t^{\iota_j }_1$, i.e.
$$
\mathcal{V}_{\mathfrak{A}^\ast,W^\ast}^{\iota_j}\left( w^\ast_{{\iota_j },\infty}\right)(s)
=w^\ast_{\iota_j ,1}(s),  \quad \mbox{for} \quad t^{\iota_j}_1\le s<t^{\iota_j}_2.
$$
Note, also that   $\PP^\ast$-a.s. we have $w^\ast_{\iota_j ,1}(s)=w^\ast_{\iota_j ,\infty}(s)$ for $ t^{\iota_j}_1\le s<t^{\iota_j}_2$, and, hence
$$
\PP^\ast\left( w^\ast_{\iota_j ,\infty }(s)= \mathcal{V}_{\iota_j }\left(w^\ast_{\iota_j ,\infty }\right)(s)\right) =1,  \quad \mbox{for} \quad t^{\iota_j}_1\le s<t^{\iota_j}_2.
$$
Let us analyze what happens in $t^{\iota_j}_2$. By equation \eqref{definfty}, we have
$$
\mathcal{V}_{\mathfrak{A}^\ast,W^\ast}^{\iota_j}\left(w^\ast_{\iota_j,\infty}\right)(t^{\iota_j}_2)=\mathcal{V}_{\mathfrak{A}^\ast,W^\ast}^{\iota_j}\left(w^\ast_{\iota_j,1}\right)(t^{\iota_j}_2).
$$
By equation of $ w^\ast_{\iota_j,2}$, i.e. \eqref{nummer2},  we have 
$$
\mathcal{V}_{\mathfrak{A}^\ast,W^\ast}^{\iota_j}\left(w^\ast_{\iota_j,\infty}\right)(t^{\iota_j}_2)=w^\ast_{\iota_j,2}(t^{\iota_j}_2).
$$
Now, we can proceed by induction. Let us assume that in $[0,t^{\iota_j}_k)$ we have shown that 
\begin{equation}\label{inductionstart}
\PP^\ast\left(
\mathcal{V}_{\mathfrak{A}^\ast,W^\ast}^{\iota_j}\left(w^\ast_{\iota_j,\infty }\right) (s) =w^\ast_{\iota_j,\infty}(s)\right)=1 \quad \mbox{for} \quad 0\le s\le t^{\iota_j}_k
.
\end{equation}
Then,  we have by equation \eqref{definfty} and equation \eqref{nummerk} we have for $t^{\iota_j}_k\le s< t^{\iota_j}_{k+1}$
$$
\mathcal{V}_{\mathfrak{A}^\ast,W^\ast}^{\iota_j}\left(w^\ast_{\iota_j,\infty}\right) (s) =\mathcal{V}_{\mathfrak{A}^\ast,W^\ast}^{\iota_j}\left(w^\ast_{\iota_j,k-1}\right)(s)=w^\ast_{\iota_j,k}(s).
$$

\item

Next, we verify a couple of statements with the goal to pass on to the limit. Here, we point out that the same construction as done for $w^\ast_{\iota_j,\infty}$ can be done on the initially given probability space $\mathfrak{A}$.
The resulting process is denoted by $w_{\iota_j,\infty}$.
 Due to the construction and by the properties of the projection, it is easy to see that the laws are preserved. In particular, $\Law(w_{\iota_j,\infty})=\Law(w^\ast_{\iota_j,\infty})$.
 \begin{claim}\label{claim222}
 We claim that
\begin{enumerate}[(1)]
\item there exists a constant $C>0$ such that $ \sup_{j\in{\mathbb{N}}}  \EE^\ast \left[ \Vert w^\ast_{\iota_j,\infty}
\Vert^{m_0}_{\mathbb{X}}\right]\le C$ and
\item  for any $r\in (1,m_0)$ we have
$$\lim_{j\to \infty}{{\mathbb{E}^\ast }}\left[   \Vert w^\ast_{\iota_j,\infty}- w^{\ast} \Vert_{\mathbb{X}}^r\right] = 0. $$
\end{enumerate}
\end{claim}
\begin{proof}[Proof of Claim \ref{claim222}:]
Clearly, since $\{  w^\ast_{\iota_j,\infty}\}_{j\in\NN}\subset \mathcal{X}( \MA^\ast)$, and $\mathcal{X}( \MA^\ast)$ is bounded in $\X$,
 we can conclude from the application of the Skorokhod lemma that
\[
\Eb\left\Vert w_{\iota_j,\infty}\right\Vert_{\X}^{r}={\Eb}^\ast \Vert{w}_{\iota_j,\infty}\Vert_{\X}^{r},
\]
for any $r\in [1,m_0]$,
so that we get by assumption (iv) that
\[
\sup_{j}{\Eb}^\ast \Vert w^\ast_{\iota_j,\infty}\Vert_{\X}^{m_0}\le R\tilde{C}=:C,
\]
where $\tilde{C}>0$ is a constant such that $\Vert\cdot\Vert_{\X}\le \tilde{C}\Vert\cdot\Vert_{\X^\prime}$ which exists by the compact and dense embedding $\mathbb{X}'\hookrightarrow \mathbb{X}$.

Hence, we know that $\{\Vert w^\ast_{\iota_j,\infty}\Vert_{\X}^{r}\}$
is uniformly integrable for any $r\in (1,m_0]$ w.r.t. the probability measure $\P^\ast$.
By \eqref{eq:PwwP}, there exists $w^\ast\in \mathcal{X}(\Afrak^\ast)$ with ${w}^\ast_{\iota_j,\infty}\to{w}^\ast$ ${\P}^\ast$-a.s., so
we get by the Vitali convergence theorem that
\begin{equation}
\lim_{j\to\infty}{\Eb}^\ast\left\Vert w^\ast _{\iota_j,\infty}-{w}^\ast\right\Vert_{\X}^{r}=0,\label{eq:strong-v}
\end{equation}
for any $r\in(1,m_0)$.
\end{proof}

\item
In this step we show that $w^\ast$ over $\MA^\ast$ together with the Wiener process $W^\ast$  is indeed a martingale solution to \eqref{spdes}.
We shall use an $\ep/3$-argument and expand
\begin{align*}
w^\ast-\mathcal{V}_{\mathfrak{A}^\ast,W^\ast}(w^\ast)
=&\underbrace{w^\ast-w_{\iota_j,\infty}^\ast}_{:=I}
+
\underbrace{ w_{\iota_j,\infty}^\ast-\mathcal{V}_{\mathfrak{A}^\ast,W^\ast}^{\iota_j}( w_{\iota_j,\infty}^\ast)}_{=:II}\\
&+
\underbrace{\mathcal{V}_{\mathfrak{A}^\ast,W^\ast}^{\iota_j}( w_{\iota_j,\infty}^\ast)-\mathcal{V}_{\mathfrak{A}^\ast,W^\ast}^{\iota_j}(w^\ast)}_{=:III}
+
\underbrace{\mathcal{V}_{\mathfrak{A}^\ast,W^\ast}^{\iota_j}(w^\ast)-\mathcal{V}_{\mathfrak{A}^\ast,W^\ast}(w^\ast)}_{=:IV}.
\end{align*}
Now, we analyze the terms $I$, $II$, $III$, and $IV$ separately.

Due to the convergence, we know that for any $r\in (1,m_0)$, $\ep>0$, there exists $j_0\in\N$,
$$\EE^\ast\Vert w^\ast-w^\ast_{\iota_j}\Vert_{\mathbb{X}}^r\le \frac \ep 3\quad\text{for all }j\ge j_0.
$$
Next, to tackle $II$, we know due to the well-posedness by the existence of fixed point in the step before,
$$
\mathcal{V}_{\mathfrak{A}^\ast,W^\ast}^{\iota_j}( w^\ast_{\iota_j,\infty})=w^\ast_{\iota_j,\infty}.
$$
To tackle $III$, due to the uniform continuity of the operator
$\mathcal{V}_{\mathfrak{A}^\ast,W^\ast}^{\iota_j}$, we know that there exists a $\delta=\delta(\ep)>0$ and $j_0\in\N$, such that, for any $r\in (1,m_0)$,
$$
\EE^\ast \left\Vert \mathcal{V}_{\mathfrak{A}^\ast,W^\ast}^{\iota_j}( w^\ast_{\iota_j,\infty})-\mathcal{V}_{\mathfrak{A}^\ast,W^\ast}^{\iota_j}(w^\ast)\right\Vert_{\mathbb{X}}^r
\le \frac{\ep}{3},
$$
whenever $j\ge j_0$, and
\[{\Eb}^\ast\left\Vert w^\ast _{\iota_j,\infty}-{w}^\ast\right\Vert_{\X}^{r}<\delta.\]
Finally, since $\mathcal{V}_{\mathfrak{A}^\ast,W^\ast}^{\iota_j}= \Pro_{n_{\iota_j}}\circ{\mathcal{V}_{\mathfrak{A}^\ast,W^\ast}}$, for $r\in (1,m)$, the difference
$$
\EE^\ast  \left\Vert\mathcal{V}_{\mathfrak{A}^\ast,W^\ast}^{\iota_j}(w^\ast)-\mathcal{V}_{\mathfrak{A}^\ast,W^\ast}(w^\ast)\right\Vert_{\mathbb{X}}^r
$$
tends to zero by uniform continuity, see \eqref{eq:pointwiseeps}. However, $m\le m_0$ by assumption (iv).
Thus, $IV$ tends to zero in $L^{r}(\Omega^\ast,\Fcal^\ast,\P^\ast;\X)$ for any $r\in (1,m)$.

As a consequence, we have
\[{\Vcal}_{\mathfrak{A}^\ast,W^\ast}({w}^\ast)={w}^\ast,\quad {\P}^\ast \text{-a.s.}\]
As seen above, ${w}^\ast\in\Xcal({\Afrak}^\ast)$,
so that by (v), ${\Vcal}_{\mathfrak{A}^\ast,W^\ast} ({w}^\ast)\in\D([0,T];U)$, and
therefore ${w}^\ast\in\D([0,T];U)$ ${\P^\ast}$-a.s. Hence for
all $t\in[0,T]$, ${\P^\ast}$-a.s.
\[
{\Vcal}({w}^\ast)(t)={w}^\ast(t)
\]
 and the proof is complete. By construction, we see that $w^\ast$ solves
 \begin{equation}\label{spdes_infact}
d w^\ast(t) =\left(\DeltaA  w^\ast(t)+ F( w^\ast,t)\right)\, dt +\Sigma( w^\ast(t))\,d{W}^\ast(t),\quad  w^\ast(0)= w^\ast_0,
\end{equation}
on $\Afrak^\ast$, where $ w_0^\ast$ is a $\Gcal^\ast_0$-measurable version of $w_0$.
\end{step}
\end{proof}

\section{Results on regularity and technical propositions}\label{technics}

This section contains the remaining results which are used in the proof of the main result Theorem \ref{mainresult}.

\subsection{Assumptions on the noise and consequences}\label{subsec:noise}

Let us recall, we denoted by $\{\psi_k:k\in\NN\}$ the eigenfunctions of the Laplace operator $-\Delta$ in $L^2(\CO)$ and by  $\{\nu_k:k\in\NN\}$ the corresponding eigenvalues,
where the enumeration is chosen in increasing order counting the multiplicity.

Let us characterize the asymptotic behavior of  $\{\nu_k:k\in\NN\}$  and $\{\psi_k:k\in\NN\}$ for an arbitrary domain $\CO$ with $C^2$--boundary. Here, we know by Weyl's law \cite{Weyl1911,Weyl1912} that there exist constants $c,C>0$ such that,
\begin{align} \label{abs_ev}
\# \{ j\;\colon\; \nu_j\le \lambda \} \le C \lambda^ \frac d2,
\end{align} 
compare with \cite{LiYau,Hoermander} and \cite[Corollary 17.8.5]{HoermanderIII}, and there exists a constant $C>0$ such that,
\begin{align} \label{abs_infty}
\sup_{x\in\CO}\vert\psi_k(x)\vert \le& C \nu_k^\frac {d-1}2, \quad k\in\NN,
\end{align} 
compare with \cite{grieser}.

\begin{remark}
If $\CO=[0,1]^d$ is a rectangular domain, then
a complete orthonormal system $\{ \psi_k:k\in\mathbb{Z} \}$ of the underlying Lebesgue space $L^2(\CO)$ is given by the trigonometric functions, see \cite[p.\ 352]{Garrett1989} (now, $\mathbb{Z}$ obviously denotes the integer numbers). Let us define firstly the eigenvalues for the Laplace on $L^2([0,1])$
\begin{equation}\label{ONS1}
e_k(x)=
\begin{cases}
{\sqrt{2}} \, \sin\big(2\pi{k} x\big) &\!\!\text{if } k\geq 1,\,x\in\CO, \\
{1} &\!\!\text{if } k = 0,\, x\in\CO\,, \\
{\sqrt{2}}\, \cos\big(2 \pi{k} x\big) &\!\!\text{if } k \leq - 1, x\in\CO.
\end{cases}
\end{equation}
For instance, the  eigenfunctions for $L^2([0,1]^2)$ for the multi-index $k=(k_1,k_2)\in \mathbb{Z}^2$   are given by the tensor product
$$
\psi_{k}(x_1,x_2)= e_{k_1}(x_1)e_{k_2}(x_2),\quad x\in [0,1]^2,
$$
the corresponding eigenvalues are given by $\nu_k=\pi^2\vert k\vert^2$, where $\vert k\vert=k_1+k_2$. The case $d=3$ works analogously.
In this special case the conditions on $\delta_1$ and $\delta_2$ in Hypothesis \ref{wiener} can be relaxed to
$$ \lambda_k^{(1)}\le C \nu_k^{-\delta_1},\quad \lambda^{(2)}_k\le C\nu_k^{-\delta_2},
\quad k=(k_1,k_2)
$$
to $\delta_1,\delta_2>\frac 12$. See also the discussion in \cite[Examples 1.2.1 and 2.1.2]{BDPR2016}.
\end{remark}

We begin with a remark on the noise coefficients. We partly work in the Banach spaces $L^m(\CO)$ and $H^\delta_m(\CO)$ for $\delta$ being arbitrary small but positive.

Given a Wiener process $W$ on $H=L^2(\CO)$  over $\MA=(\Omega,\Fcal,\mathbb{F},\PP)$, and a progressively measurable process $\xi\in\mathcal{M}^2_\MA(L^2(\CO))$,
$\rho\in[0,\frac 12]$,
let us define $\{Y(t):t\in[0,T]\}$ by
 $$
 Y(t):=\int_0^ t\sigma( \xi(s))\, d {W}(s), \quad t\in[0,T].
 $$
Let $E$ be a function space, specified later. Here, for each $\xi\in E$, $\sigma (\xi)$ is interpreted as a multiplication operator acting on the
elements of $H$, namely,
 $\sigma: \xi\mapsto \xi\psi\in \CS'(\Ocal)$, where $\CS'(\Ocal)$ denotes the space of \emph{tempered distributions} on $\Ocal$.
Let $E$ be a Banach space\footnote{For Banach spaces of martingale type $2$ we refer to \cite{vanneerven,brzezniak2}} of martingale type $2$.
An inequality needed in several places within the proof is the Burkholder-Davis-Gundy inequality given for $p\ge 1$ as follows
\begin{equation*}\label{BDG1}
\EE \left[\sup_{t \in [0, T]} \vert Y(t)\vert^p_{E}\right] \leq C(p) \, \EE \,\left[ \int_0^T \vert\xi(t)\vert^2_{\gamma(H,E )}\, dt\right]^\frac p2,
\end{equation*}
where $\gamma(H,E)$ denotes the space of $\gamma$-radonifying operators and $\vert\cdot\vert_{\gamma(H,E)}$ the corresponding norm, cf. \cite{brzezniak2,vanneerven}. In case $E$ is a Hilbert space, the $\gamma$-norm coincides with the Hilbert-Schmidt norm $\vert\cdot\vert_{L_{\textup{HS}}(H,E)}$. See Appendix \ref{sec:BDG} for further details.

Since $\vert\xi\psi_k\vert_{L^m}\le \vert\xi\vert_{L^m}\, \vert\psi_k\vert_{L^\infty}$ and \eqref{abs_infty},
we know that
\begin{align*} 
 \vert\xi\vert^2_{\gamma (H, L^m )}\le \sum_{k\in\NN} \vert\lambda^{(j)}_k\vert^2 \vert\xi\psi_k\vert_{L^m}^2\le \vert\xi\vert_{L^m} ^2 \sum_{k\in\NN}\nu_k^{-2\delta_j} \vert\psi_k\vert_{L^\infty}^2
 \le C  \vert\xi\vert_{L^m} ^2 \sum_{k\in\NN} \nu_k^{(d-1)-2\delta_j},
 \end{align*} 
which is finite if $\delta_j>\frac d2-\frac 14$.

Next, let $E=H^{-1}_2(\CO)$. Since
$$\vert\xi\psi\vert_{H^{-1}_2}\le  \vert\xi\vert_{H^{-1}_2}\vert\psi\vert_{H^\delta_2}\le  \vert\xi\vert_{H^{\rho}_2}\vert\psi\vert_{L^\infty},$$
we know that $\delta>\frac 12$, then
$$ \vert\sigma_1 (\xi)\vert_{\gamma(H,H^{\rho}_2)}\le C\, \vert\xi\vert_{H^{\rho}_2}.
$$
Furthermore,
let $E=H^{\rho}_2(\CO)$, where $\rho$ is given in Hypothesis \ref{init}. Since
$$\vert\xi\psi\vert_{H^{\rho}_2}\le  \vert\xi\vert_{H^{\rho}_2}\vert\psi\vert_{H^{2\vert\rho\vert}_2}\le  \vert\xi\vert_{H^{\rho}_2}\vert\psi\vert_{L^\infty},$$
we know that $\delta>\frac 12$, then
$$ \vert\sigma_2 (\xi)\vert_{\gamma(H,H^{\rho}_2)}\le C\, \vert\xi\vert_{H^{\rho}_2}.
$$

\begin{remark}\label{LSSigma}
From Hypotheses \ref{wiener},
one can infer that
there exists a constant $C>0$ such that
\[\sum_{k=1}^\infty [(\sigma_j (\xi) f_k)(x)]^2\le C\vert\xi(x)\vert_{L^2}^2,\quad \forall \xi\in L^2(\CO),\;x\in\Ocal,\mbox{ and } j=1,2.\]
Here $\{f_k\}$ is an orthonormal basis in $H^{-1}_2(\CO)$, compare with \cite[Hypothesis 3, p.\ 42]{BDPR2016}.
In addition,
note that
\begin{itemize}
\item Firstly, $\sigma_j:H^{-1}_2(\CO)\to L_{\textup{HS}}(H_j,H^{-1}_2(\CO))$
is of linear growth and Lipschitz continuous.
In particular, there exists constants $C_1,L_1>0$ such that
\begin{align*} 
\vert\sigma_j(\xi)\vert_{L_{\textup{HS}}(H_j,H^{-1}_2)} \le& C_1(1+\vert\xi\vert_{H^{-1}_2}),\quad \xi\in H^{-1}_2(\CO),
\\ \vert\sigma_j(\xi)-\sigma_j(\eta)\vert_{L_{\textup{HS}}(H_j,H^{-1}_2)}\le& L_1 \vert\xi-\eta\vert_{H^{-1}_2}, \quad  \xi,\eta\in H^{-1}_2(\CO).
\end{align*} 

\item Secondly, $\sigma_j:L^2(\CO)\to L_{\textup{HS}}(H_j,L^2(\CO))$
is of linear growth and Lipschitz continuous.
In particular, there exists constants $C_2,L_2>0$ such that
\begin{align*} 
\vert\sigma_j(\xi)\vert_{L_{\textup{HS}}(H_j,L^2)} \le& C_2(1+\vert\eta\vert_{L^2}),\quad \eta\in L^2(\CO),
\\ \vert\sigma_j(\xi)-\sigma_j(\eta)\vert_{L_{\textup{HS}}(H_j,L^2)}\le& L_2 \vert\xi-\eta\vert_{L^2}, \quad  \xi,\eta\in L^2(\CO).
\end{align*} 
\item
Similarly, straightforward computations and using the fact that $\vert f_k\vert_{L^\infty}\le \nu_k^\frac {d-1}2 $, see \eqref{abs_infty} and \cite[p.\ 46]{BDPR2016}, we get that
\begin{align*} 
\vert\sigma_2(\xi)\vert_{\gamma(H_2,L^m)} \le& C_1(1+\vert\xi\vert_{L^m}),\quad \xi\in L^m(\CO),
\\ \vert\sigma_2(\xi)-\sigma_2(\eta)\vert_{{\gamma}(H_2,L^m)}\le& L_1 \vert\xi-\eta\vert_{L^m}, \quad  \xi,\eta\in L^m(\CO),
\end{align*} 
where $\gamma(H_2,L^m)$ denotes the space of $\gamma$-radonifying operators, cf. \cite{brzezniak2,vanneerven}.
\end{itemize}
\end{remark}

\subsection{Properties of equation \eqref{eq:cutoffuu}}
In this subsection, we are analyzing equation \eqref{eq:cutoffuu}, where $\kappa\in\N$ and the couple  $(\xi,\eta)\in\subX\subset\CM_\MA^{2,m_0}(\BY,\mathbb{Z})$ are given. The constants $\Reins,\Rzwei$ and $\Rdrei$ are given as in \eqref{constantR1},  \eqref{constantR2}, and  \eqref{constantR3}, respectively.

First, we will show that a unique solution $u_\kappa$ to the system \eqref{eq:cutoffuu} exists and is nonnegative. Second, we will show by variational methods that this solution satisfies some integrability properties, given $u_0\in L^ {\pp}(\Ocal)$, $p\ge 1$. Let us start with proving an inequality that we are going to use in the sequel.

\begin{lemma}\label{lem:useful}
Assume that the Hypotheses of Theorem \ref{mainresult} hold.
For any $\eps_1,\eps_2>0$ and any $(\xi,\eta)\in\subX\subset\CM_\MA^{2,m_0}(\BY,\mathbb{Z})$, $u,w\in L^{\gamma+1}(\Ocal)$, $\kappa\in\N$, there exists a constant $C=C(\kappa,\chi,\eps_1,\eps_2)>0$ such that
\[\chi\left\vert\int_{\Ocal}(-\Delta)^{-1}(\phi_\kappa(h(\eta,\xi,t))\xi^{2}(t,x)\, u)w\,dx\right\vert\le C\left(1+\vert u\vert_{H_2^{-1}}^2\right)+\eps_1\vert u\vert_{L^{\gamma+1}}^{\gamma+1}+\eps_2\vert w\vert_{L^{\gamma+1}}^{\gamma+1}.\]
\end{lemma}
\begin{proof}
By H\"{o}lder's inequality,
we have for fixed $\omega\in\Omega$ and fixed $t\in [0,T]$,
\begin{align*} 
& \left\vert\int_{\Ocal}(-\Delta)^{-1}(\phi_\kappa(h(\eta,\xi,t))\xi^{2}(t,x)\, u)w\,dx\right\vert
\\\nonumber
\le&
\vert(-\Delta)^{-1}(\phi_\kappa(h(\eta,\xi,t))\xi^{2}(t)u)\vert_{L^{(\gamma+1)/\gamma}}\vert w\vert_{L^{\gamma+1}}
\\\nonumber
\le&
\vert\phi_\kappa(h(\eta,\xi,t))\xi^{2}(t)u\vert_{H^{-2}_\frac {\gamma+1}\gamma }   \vert w\vert_{L^{\gamma+1}}
\end{align*} 
Now, pick $0>s_1>-\frac{3}{\gamma+1}$ with small modulus and $s_2=\frac{3}{\gamma+1}>-s_1$ and note that $s_2\ge s_1$, $s_1+s_2\le \frac{d\gamma}{\gamma+1}$, $s_1+s_2-\frac{d\gamma}{\gamma+1}\ge -2$, and thus we can apply Proposition \ref{prop:runst190}.
Then we know
\begin{align*} 
\vert\phi_\kappa(h(\eta,\xi,t))\xi^{2}(t)u)\vert_{H^{-2}_\frac {\gamma+1}\gamma}\le C\vert\phi_\kappa(h(\eta,\xi,t))\xi^{2}(t)\vert_{H^{s_2}_{\frac {\gamma+1}\gamma}}
\vert u\vert_{H^{s_1}_{\frac {\gamma+1}\gamma}}
.\end{align*} 
Let $l\ge 2$
such that
$$
\frac dl =\frac 12\left(\frac {d\gamma}{\gamma+1}- s_2\right)=\frac{d\gamma}{2(\gamma+1)}- \frac 3{2(\gamma+1)}. $$
Then, we know that (see \cite[p. 364]{runst}) there exists $c>0$ with
$$
\vert\phi_\kappa(h(\eta,\xi,t))\xi^{2}(t)\vert_{H^{s_2}_{\frac {\gamma+1}\gamma}}
\le c
\vert(\phi_\kappa(h(\eta,\xi,t)))^{1/2}\xi(t)\vert^2_{H^{s_2}_{l}}.
$$
Next, if $\delta\ge \frac d2 -\frac dl+s_2=\frac{d}{2}+\frac{d\gamma}{2(\gamma+1)}+ \frac{3}{2(\gamma+1)}$ we get
$$
\vert\phi_\kappa(h(\eta,\xi,t))\xi^{2}(t)\vert_{H^{s_2}_{\frac {\gamma+1}\gamma}}
\le c
\vert(\phi_\kappa(h(\eta,\xi,t)))^{1/2}\xi(t)\vert^2_{H^{\delta}_{2}}.
$$
Applying the Young inequality we know for any $\ep_2>0$ there exists a constant $C(\chi,\ep_2)>0$ such that
\begin{align*}  
\lqq{ \chi\left\vert\int_{\Ocal}(-\Delta)^{-1}(\phi_\kappa(h(\eta,\xi,t))\xi^{2}(t,x)\, u)w\,dx\right\vert}
\\\nonumber
\le&
C(\chi,\ep_2)\vert(\phi_\kappa(h(\eta,\xi,t)))^{1/2}\xi(t)\vert^{2\frac {\gamma+1}\gamma}_{_{H^{\delta}_{2}}}  \vert u\vert^{\frac {\gamma+1}\gamma}_{H^{s_1}_{\frac {\gamma+1}\gamma}}+\ep_2 \vert w\vert^{\gamma+1}_{L^{\gamma+1}}
\end{align*} 
Let us fix $r\in (2,\gamma+1)$, such that
$$
\frac 1r=\frac \theta\infty+\frac{1-\theta}{\gamma+1},\quad s_1=\theta(-1)+(1-\theta)0\quad\Rightarrow\quad \frac 1r=\frac {1+s_1}{\gamma+1}.
$$
Thus for any $\ep,\ep_2>0$, there exists a constant $C(\chi,\ep,\ep_2)>0$
\begin{align*}  
&\qquad\lqq{ \chi \left\vert\int_{\Ocal}(-\Delta)^{-1}(\phi_\kappa(h(\eta,\xi,t))\xi^{2}(t,x)\, u)w\,dx\right\vert}
\\\nonumber
\le&
C(\chi,\ep,\ep_2)
\vert(\phi_\kappa(h(\eta,\xi,t)))^{1/2}\xi(t)\vert^{\frac {2r(\gamma+1)}{r\gamma-(\gamma+1)}}_{H^{\delta}_{2}} +\ep \vert u\vert^{r}_{H^{s_1}_{r}}+\ep_2 \vert w\vert^{\gamma+1}_{L^{\gamma+1}}
\end{align*} 
To apply the the cut-off function $\phi_\kappa^{1/2}$, we need
$$
\delta=\rho+(1-\theta),\quad \frac {r\gamma-(\gamma+1)} {2r(\gamma+1)}=\frac \theta\infty+\frac {1-\theta}{2}.
$$
In fact, if
$$\rho>\frac {15+3\gamma-16\gamma^2+8d(1+\gamma)}{16(1+\gamma)^2}
$$
then,
$$\Vert \xi\Vert_{L^{\frac {2r(\gamma+1)}{r\gamma-(\gamma+1)}}(0,T;{H^{\delta}_{2}})}\le c
\Vert\xi\Vert_{\BH},
$$
 and
$$
\Vert \xi\Vert_{L^{\frac {2r(\gamma+1)}{r\gamma-(\gamma+1)}}(0,T;{H^{\delta}_{2}})}\le C(\kappa).
$$
Summarizing the calculations above
we know for any $\ep,\ep_2>0$ there exists a constant $C(\kappa,\chi,\ep,\ep_2)>0$ such that
\begin{align}\label{summ}
   &  \chi\left\vert\int_{\Ocal}(-\Delta)^{-1}(\phi_\kappa(h(\eta,\xi,t))\xi^{2}(t,x)\, u)w\,dx\right\vert
  \nonumber
    \\
  & C(\kappa,\chi,\ep,\ep_2)
+\ep \vert u\vert^{r}_{H^{s_1}_{r}}+\ep_2 \vert w\vert^{\gamma+1}_{L^{\gamma+1}}
.
\end{align}
In addition, by interpolation, see e.g. \cite{bergh}, we get that for any $\ep_1>0$, there exists $\eps>0$ and a constant $C(\ep_1,\eps)>0$ such that
\[ \ep\vert u\vert^{r}_{H^{s_1}_{r}}\le \ep_1\vert u\vert^{\gamma+1}_{L^{\gamma+1}} +C(\ep_1,\eps)\vert u\vert_{H^{-1}_2}^2.\]
Setting $C:=C(\ep_1,\eps)\vee C(\kappa,\chi,\ep,\ep_2)$ completes the proof.
\end{proof}

\begin{theorem}\label{theou1}
Assume that the Hypotheses of Theorem \ref{mainresult} hold.
For any $(\xi,\eta)\in\subX\subset\CM_\MA^{2,m_0}(\BY,\BX)$, and
for any $u_{0}\in L^{2}(\Omega;\Fcal_{0},\P;H^{-1}_2(\CO))$ the system
\begin{equation}
d\uk(t)=r_u\Delta \uk ^{[\gamma]}(t)\, dt - \chi\phi_\kk (h(\eta,\xi,t))\, \uk\,\xi^2(t)\,dt +\sigma_1 \uk (t)\,dW_1(t),
\label{eq:u-spde}
\end{equation}
with $\uk(0)=u_0$ has a unique solution $\uk$ on the time interval $[0,T]$ which is $\P$-a.s. continuous in $H_2^{-1}$ and satisfies
$$\EE \left[\sup_{0\le s\le T} \vert\uk(s)\vert^2_{H^{-1}_2}\right]+\EE \int_0^T\vert\uk(s)\vert_{L^{\gamma+1}}^{\gamma+1}\, ds
 <\infty.
$$
In particular, there exists a constant $C(T,\kk)>0$ such that
$$\EE\Vert\uk\Vert^2_{\BY}\le C(T,\kk).
$$
\end{theorem}

\begin{proof}
Before starting with the proof,
we introduce the setting used by the book by Barbu, Da Prato and R\"ockner \cite{BDPR2016}, and the book by Liu and R\"ockner \cite{weiroeckner}, respectively.
Let $\CH:=H_{2}^{-1}(\Ocal)$, the dual space of $\CH^\ast={H}_{2}^{1}(\Ocal)$
(corresponding to Neumann boundary conditions). Let $V:=L^{\gamma+1}(\Ocal)$.
By the Sobolev embedding theorem, $\Hcal^{\ast}\hookrightarrow L^{(\gamma+1)/\gamma}(\Ocal)$
densely and compactly. Thus, upon identifying $\Hcal$ with its own dual space
via the Riesz-map $(-\Delta)^{-1}$ of $\Hcal$, we have a Gelfand triple
\[
V\subset \CH\equiv\CH^{\ast}\subset V^{\ast}.
\]
We set
\[
A(t,u,\omega):=r_u\Delta u^{[\gamma]}-\chi\phi_\kappa(h(\eta(\omega),\xi(\omega),t))\xi^{2}(t,\omega) u.
\]
Note, due to the assumption $(\xi,\eta)\in\subX\subset\CM_\MA^{2,m_0}(\BY,\BX)$,  $\xi$ is  adapted such that for any $q_2\in [1,m_0]$,
\begin{equation}
\left( \int_{0}^{T}\vert\phi_\kappa(h(\eta,\xi,t))\xi(t)\vert_{L^{m}}^{m_0}
\,dt\right)\le (2\kappa)^{1/\nu}.\label{eq:v-bound-L1}
\end{equation}
Now,
given \eqref{eq:v-bound-L1}, we shall verify the conditions
of \cite[Theorem 5.1.3]{weiroeckner} for
$\gamma>1$.
First note that for fixed $t\in[0,T]$ and fixed $\omega\in\Omega$,
$A$ maps from $V$ to $V^{\ast}$. In particular, by arguments as in the proof of Lemma \ref{lem:useful},
we have for $\gamma>1$
\begin{align} \label{eq:A-bounded}
&\left\vert\mathbin{_{V^{\ast}}\langle A(t,u),w\rangle_{V}}\right\vert \\
=& \nonumber\left\vert\int_{\Ocal}\left[r_u u^{[\gamma]}(x)w(x)+(-\Delta)^{-1}(\chi\phi_\kappa(h(\eta,\xi,t))\xi^{2}(t,x)u(x)) w(x)\right]\,dx\right\vert
\\
\nonumber \le&\left\vert r_u\vert u\vert_{L^{\gamma+1}}^{\gamma}\vert w\vert_{L^{\gamma+1}}+\chi\vert(-\Delta)^{-1}(\phi_\kappa(h(\eta,\xi,t))\xi^2(t) u\vert_{L^{(\gamma+1)/\gamma}}\vert w\vert_{L^{\gamma+1}}\right\vert\\
\nonumber \le&\left[r_u\vert u\vert_{L^{\gamma+1}}^{\gamma}+
C(\kappa,\ep_1,\chi) + \ep_1\vert u\vert^{\gamma+1}_{L^{\gamma+1}} +C(\kappa,\ep_1,\chi)\vert u\vert_{H^{-1}_2}^2 \right]\vert w\vert_{L^{\gamma+1}},
\end{align} 
where only $\xi$,  and $h(\eta,\xi,\cdot)$ depend on $t$ and $\omega$.

Next, we verify Hypothesis (H1), (H2$^{\prime}$), (H3), and  (H4$^{\prime}$) of \cite[Theorem 5.1.3]{weiroeckner}.

\begin{enumerate}
\item[(H1):]  For $\lambda\in\mathbb{R}$, $u_{1},u_{2},w\in L^{\gamma+1}(\Ocal)$, $t\in[0,T]$,
$\omega\in\Omega$ consider the map
\[
\lambda\mapsto\langle A(t,u_{1}+\lambda u_{2}),w\rangle
\]
and show its continuity.
Note, that we have
\begin{align*}
&\langle A(t,u_{1}+\lambda u_{2}),w\rangle\\
=&-r_u\int_{\Ocal}(u_{1}(x)+\lambda u_{2}(x))^{[\gamma]}w(x)\,dx\\
&-\chi\int_{\Ocal}(-\Delta)^{-1}[\phi_\kappa(h(\eta,\xi,t))\xi^{2}(t,x)(u_1(x)+\lambda u_2(x)) ]w(x)\,dx,
\end{align*}
where only $\xi$, and $h(\eta,\xi,\cdot)$ depend on $t$ and $\omega$.
For the first integral
in the above identity, we prove continuity with the same arguments
as in \cite[Example 4.1.11, p. 87]{weiroeckner}. The second integral can be bounded by Lemma \ref{lem:useful} as follows
\begin{align*} 
\lqq{\chi\left\vert\int_{\Ocal}(-\Delta)^{-1}[\phi_\kappa(h(\eta,\xi,t))\xi^{2}(t,x)(u_1(x)+\lambda u_2(x)) ]w(x)\,dx\right\vert
}&\\
\le&
 C\left(1+\vert u_{1}+\lambda u_{2}\vert_{H_2^{-1}}^2\right)+\eps_1\vert u_{1}+\lambda u_{2}\vert_{L^{\gamma+1}}^{\gamma+1}+\eps_2\vert w\vert_{L^{\gamma+1}}^{\gamma+1}.
\end{align*} 
Hence, the assumption (H1) of \cite[Theorem 5.1.3]{weiroeckner} is satisfied.
\item[(H2$^{\prime}$):]  Let $u,w\in L^{\gamma+1}(\Ocal)$, $t\in[0,T]$, $\omega\in\Omega$. Take
\eqref{eq:porous-medium-inequality} into account and apply Lemma \ref{lem:useful} to get that
\begin{align*}
& \mathbin{_{V^{\ast}}\langle A(t,u)-A(t,w),u-w\rangle_{V}}
+\vert\sigma_1 u - \sigma_1 w\vert^2_{L_{\textup{HS}}(H_1,\CH)}\\
\le& -r_u\int_{\Ocal}(u^{[\gamma]}(x)-w^{[\gamma]}(x))(u(x)-w(x))\,dx+C\vert u-w\vert_{\CH}^{2}\\
&+ \eps\vert u-w\vert^{\gamma+1}_{L^{\gamma+1}} +C(\eps,\kappa,\chi)\left(1+\vert u-w\vert_{H^{-1}_2}^2\right)\\
\le & (\eps-2^{1-\gamma}r_u)\vert u-w\vert_{L^{\gamma+1}}^{\gamma+1}+C(1+\vert u-w\vert_{\CH}^{2})
\end{align*}
and (H2$^{\prime}$) of \cite[Theorem 5.1.3]{weiroeckner} is satisfied.

\item[(H3):]  Let $u\in L^{\gamma+1}(\Ocal)$, $t\in[0,T]$, $\omega\in\Omega$. Straightforward calculations, estimate \eqref{eq:A-bounded}, and taking into account Hypothesis \ref{init} give,
\begin{align*} 
\lqq{
 \mathbin{_{V^{\ast}}\langle A(t,u),u\rangle_{V}}+\vert\sigma_1 u\vert^2_{L_{\textup{HS}}(H_1,\Hcal)}}&
 \\
\le& -\int_{\Ocal}\left[r_u u^{[\gamma]}(x)u(x)+\chi(-\Delta)^{-1}(\phi_\kappa(h(t,\eta,\xi))\xi^{2}(t,x)u(x))u(x)\right]\,dx+C\vert u\vert_{\Hcal}^{2}
\\
\le& -r_u\vert u\vert_{L^{\gamma+1}}^{\gamma+1}+
C(\kappa,r_u,\chi) + \frac {r_u}4\vert u\vert^{\gamma+1}_{L^{\gamma+1}} +C(\kappa,r_u,\chi )\vert u\vert_{H^{-1}_2}^2 +\frac {r_u}4\vert u\vert^{\gamma+1}_{L^{\gamma+1}}.
\end{align*} 
As a consequence, (H3) of \cite[Theorem 5.1.3]{weiroeckner} holds with  $\theta:=r_u$.

\item[(H4$^{\prime}$):]  Let $u\in L^{\gamma+1}(\CO)$, $t\in[0,T]$, $\omega\in\Omega$. Then,
(H4$^{\prime}$) of \cite[Theorem 5.1.3]{weiroeckner} holds by \eqref{eq:A-bounded} with $\alpha:=\gamma+1$, and
$\beta:=0$.
\end{enumerate}
The rest of the proof follows by an application of \cite[Theorem 5.1.3]{weiroeckner}.
\end{proof}

\begin{proposition}\label{positivityu}
Under the conditions of Theorem \ref{theou1} and the additional conditions $u_0(x)\ge 0$ for all $x\in\CO$ and $\xi(t,x)\ge 0$,  for a.e. $(t,x)\in[0,T]\times \CO$, the solution $u_\kappa$
to \eqref{eq:cutoffuu} is a.e. nonnegative.
\end{proposition}
\begin{proof}
For the nonnegativity of the solution to \eqref{eq:cutoffuu}, we refer to the proof of positivity of the stochastic porous medium equation, see
Section 2.6 in \cite{BDPR2016} and see also \cite{BDPR3}. Mimicking the proof of nonnegativity  in \cite{BDPR2016} and applying a comparison principle \cite{Kotelenez1992} the nonnegativity of  \eqref{eq:cutoffuu} can be shown.
\end{proof}

In the next proposition we are using  variational methods to verify uniform bounds of $u_\kappa$, $\kappa\in\N$.
\begin{proposition}\label{uniformlpbounds}
Fix $p\ge 1$. Then, there exists a constant $C_0(p,T)>0$ such that for every $u_0\in L^{p+1}(\Omega,\Fcal_0,\PP;L^{p+1}(\CO ))$,
 and for every $\kk\in\NN$ and for every $(\xi,\eta)\in\subX$ such that $\uk$ solves \eqref{eq:cutoffuu},
 the following estimate
\begin{align*} \nonumber
\lqq{\EE \left[\sup_{0\le s\le T} \vert\uk (s)\vert_{L^{p+1} }^{p+1}\right]
+ \gamma p(p+1)r_u\EE \int_0^ \TInt\int_\CO  \vert\uk(s,x)\vert^{p+\gamma-2} \vert\nabla \uk (s,x)\vert^2\, dx \, ds} &
\\\nonumber &{} +(p+1)\chi\EE \int_0^\TInt \int_\CO  \uk^{[p]}(s,x) \phi_\kk (h(\eta,\xi,s))u_\kappa(s,x)\xi^2(s,x)\, dx \, ds
\\
\le& C_0(p,T)\,\left( \EE\vert u_0\vert_{L^{p+1}}^{p+1}+1\right)
\end{align*} 
is valid.
\end{proposition}
\begin{proof}
By It\^o's formula for the functional $u\mapsto\vert u\vert^{p+1}_{L^{p+1}}$, we get that $u_\kappa$ satisfies for $t\in [0,T]$,
\begin{align*}
&\vert u_\kappa (t)\vert^{p+1}_{L^{p+1}}-\vert u_0\vert^{p+1}_{L^{p+1}}\\
\le &(p+1)r_u \int_0^t\int_\Ocal \Delta(u_\kappa^{[\gamma]}(s,x))u^{[p]}_\kappa(s,x)\,dx\,ds\\
&-(p+1)\chi\int_0^t\int_\Ocal \phi_\kk (h(\eta,\xi,s))u_\kappa(s,x)\xi^2(s,x) u^{[p]}_\kappa(s,x)\,dx\,ds\\
&+M(t)+\frac{p}{2}(p+1)\int_0^t \vert u_\kappa(s)\vert^{p-1}_{L^{p+1}}\vert\sigma_1 u_\kappa(s)\vert_{\gamma(H_1,L^{p+1})}^2\,ds,
\end{align*}
where $t\mapsto M(t)$ is a local martingale. Integrating by parts, taking expectation (we may as well apply the Burkholder-Davis-Gundy inequality, see Section \ref{sec:BDG}) and rearranging yields,
\begin{align*}
&\EE\left[\vert u_\kappa (s)\vert^{p+1}_{L^{p+1}}\right]+(p+1)r_u \EE\int_0^t\int_\Ocal \nabla (u_\kappa^{[\gamma]}(s,x))\cdot\nabla (u_\kappa^{[p]}(s,x))\,dx\,ds\\
&+(p+1)\chi\EE\int_0^t\int_\Ocal \phi_\kk (h(\eta,\xi,s))u_\kappa(s,x)\xi^2(s,x) u^{[p]}_\kappa(s,x)\,dx\,ds\\
\le&\EE\vert u_0\vert^{p+1}_{L^{p+1}}+\frac{p}{2}(p+1)C(p)\EE\int_0^t \vert u_\kappa(s)\vert^{p+1}_{L^{p+1}}\,ds+\EE\int_0^t\vert\sigma_1 u_\kappa(s)\vert_{\gamma(H_1,L^{p+1})}^{(p+1)/(p-1)}\,ds,
\end{align*}
and hence by Remark \ref{LSSigma} and Gronwall's lemma,
\begin{align*}
&\EE\left[\sup_{0\le s\le T}\vert u_\kappa (s)\vert^{p+1}_{L^{p+1}}\right]+\gamma p(p+1)r_u \EE\int_0^T\int_\Ocal \vert\nabla u_\kappa(s,x)\vert^2 \vert u(s,x)\vert^{\gamma+p-2} dx\,ds\\
&+(p+1)\chi\EE\int_0^T\int_\Ocal \phi_\kk (h(\eta,\xi,s)) u_\kappa(s,x)\xi^2(s,x) u^{[p]}_\kappa(s,x)\,dx\,ds\\
\le&C_0(p,T)\left(\EE\vert u_0\vert^{p+1}_{L^{p+1}}+1\right).
\end{align*}
\end{proof}

We may remark that the above result now permits an application of Proposition \ref{runst1}. Note that the inequality becomes particularly useful, when
$u_\kappa\ge 0$ and $\eta\ge 0$.

\subsection{Properties of equation \eqref{eq:cutoffvv}}

Given the couple  $(\eta,\xi)\in\CM_\MA^{2,m_0}(\BY,\X)$,
we consider  the solution $\vk$  to the equation \eqref{eq:cutoffvv}.
First, we will
In the next proposition we investigate existence and uniqueness and the regularity of the process $\wk$.
The constants $\Reins,\Rzwei$ and $\Rdrei$ are given as in \eqref{constantR1},  \eqref{constantR2}, and  \eqref{constantR3}, respectively.

\begin{proposition}\label{propvarational}
Assume that the Hypotheses of Theorem \ref{mainresult} hold.
Then, for any pair $(\eta,\xi)\in\subX\subset\CM_\MA^{2,m_0}(\BY,\mathbb{Z})$ and for any initial datum $v_0$ as in Hypothesis \ref{init},
there exists a solution $\vk$ to \eqref{eq:cutoffvv} on the time interval $[0,T]$
such that $v_\kappa$ is continuous in $H^\rho_2(\Ocal)$ and such that
$$ \wk\in L^\infty(0,T;H^{\rho}_2(\CO))\cap L^2(0,T;H^{\rho+1}_2(\CO))\quad \PP\mbox{-a.s.}
$$
In addition, for all $\kk\in\NN$ there exists a constant $C=C(T,\kk)>0$ such that
\begin{align} \label{estimatesol2}
 \EE \Vert\vk\Vert_{\BH}^{m_0}
\le \EE \vert v_0\vert_{H^\rho_2}^{m_0} +C(\kk,T)\Reins.
\end{align} 
\end{proposition}
\begin{proof}
Let us consider the following equation with a locally integrable and progressively measurable $t\mapsto F(t)$
\begin{align} \label{wwwoben3}
 \qquad d w(t)=&r_v\Delta   w(t)\,dt+F(t)\,dt+\sigma_2 w(t)\, dW_2(t),\quad w(0)=w_0\in H^\rho_2(\CO),
\end{align} 
where the Laplace operator $\Delta$ is equipped with periodic boundary conditions on a box or Neumann boundary conditions and initial condition $w(0)= w_0$.
A solution to \eqref{wwwoben3} for $F\equiv 0$ is given by standard methods (see e.g. \cite[Chapter 6]{DaPrZa:2nd} or \cite[Theorem 4.2.4]{weiroeckner})
such that for any $q\ge 1$,
\begin{align} \label{erstes3}
\EE\Vert w\Vert^q_{C([0,T];H_2^\rho)}\le C_1(T)\quad  \mbox{and}\quad  \EE\Vert w\Vert^q_{L^2([0,T];H^{\rho+1}_2)}\le C_2(T).
\end{align} 
The solution to \eqref{wwwoben3} is given by
$$
\wk(t)=w(t)+\int_0^ t e^{r_v\Delta(t-s)} F(s)\,ds,\quad t\in[0,T].
$$

The Minkowski inequality, the smoothing property of the heat semigroup and the Sobolev embedding $L^1(\CO)\hookrightarrow H^{\rho-\delta}_2(\CO)$ for $\delta-\rho+\frac d2>d$ give
\begin{align} \label{zweites3}
\lqq{\left\vert\int_0^ t e^{r_v\Delta(t-s)}F(s)\,ds\right\vert_{H^\rho_2} \le\int_0^ t \left\vert e^{r_v\Delta(t-s)}F(s)\right\vert_{H^\rho_2}\, ds
}&
\\\nonumber
 \le& \int_0^ t (t-s)^{-\frac \delta2} \left\vert F(s)\right\vert_{H^{\rho-\delta}_2}\, ds
 \le C(T)\, \left(\int_0^ t  \left\vert F(s)\right\vert_{L^1}^\textkappa\, ds\right)^\frac 1 \textkappa.
 \end{align} 
 Setting
$$F=\phi_\kappa(h(\eta,\xi,\cdot))\eta\xi^2,$$
we obtain by the H\"older inequality  for $\delta/2<1-\frac 1\textkappa $ and $\frac {2\textkappa}{m_0}<1$
\begin{align} \label{zweites.2}
\lqq{\qquad\qquad\int_0^ T \vert F(s)\vert_{L^1}^\mu\, ds
}
&
\\
\le& \sup_{0\le s\le T} \vert\eta(s)\vert_{L^{p^\ast }}^{\mu} \int_0^ T (\phi_\kappa(h(\eta,\xi,t)))^{\mu} \vert \xi(s)\vert_{L^m}^{2\mu}\, ds\le (2\kk)^{\frac{2\mu}{m_0\vardelta}} \sup_{0\le s\le T} \vert\eta(s)\vert_{L^{p^\ast}}^{\mu}.\nonumber
\end{align} 
 Observe, since $\rho<2-\frac 4m-\frac d2$, such a $\delta$ and $\textkappa$ exists.
Taking expectation, we know due to the assumptions on $\eta$, that
the  term can be bounded.
{It remains to calculate the norm in $L^2(0,T;H^{\rho+1}_2(\CO))$.
By standard calculations (i.e. applying the smoothing property and the Young inequality for convolution) we get
for $\frac 32 \ge \frac 1\mu+\frac 1\kappa$ and $\delta \kappa/2<1$
\begin{align*} 
 \left\Vert\int_0^ \cdot e^{r_v\Delta(\cdot-s)}F(s)\,ds\right\Vert_{L^2(0,T;H_2^{\rho+1})}\le C\,  \Vert F\Vert_{L^\mu(0,T;H_2^{\rho+1-\delta})}
.
\end{align*} 
The embedding $L^1(\CO) \hookrightarrow H^{\rho+1-\delta}_2(\CO)$ for $\delta-(\rho+1)>\frac d2$
gives
\begin{align*} 
 \Vert F\Vert_{L^\mu(0,T;H_2^{\rho-1})}\le  \Vert F\Vert_{L^\mu(0,T;L^1)},
 \end{align*} 
 and similarly to before
\begin{align*} 
\EE \left\Vert\int_0^ \cdot e^{r_v\Delta(\cdot-s)}F(s)\,ds\right\Vert^{m_0}_{L^2(0,T;H_2^{\rho+1})}\le  C(\kk)\, \Reins
.
\end{align*} 
}%
\end{proof}

\begin{proposition}\label{positivityv}
Under the conditions of Theorem \ref{propvarational} and the additional condition $v_0(x)\ge 0$ for all $x\in\CO$ and $\eta(t,x)\ge 0$, $\xi(t,x)\ge 0$ for a.e. $(t,x)\in[0,T]\times \CO$, the solution $\wk$
to \eqref{eq:cutoffvv} is nonnegative a.e.
\end{proposition}
\begin{proof}
The heat semigroup, which is generated by the Laplace, maps nonnegative functions to nonnegative functions. In this way we refer to the proof of nonnegativity by Tessitore and Zabczyk \cite{TessitoreZabczyk1998}. The perturbation can be incorporated by comparison results, see \cite{Kotelenez1992}.
\end{proof}

\begin{proposition}\label{semigroup}
Assume that the Hypotheses of Theorem \ref{mainresult} hold.
In addition let us assume that we have for the couple $(\eta,\xi)\in\subX\subset\CM_\MA^{2,m_0}(\BY,\mathbb{Z})$.
Let $v_\kappa$ be a solution to
$$
d \wk(t)=[r_v\Delta \wk(t)+ \phi_\kk(h(\eta,\xi,t))\, \eta(t)\xi^2(t)]\,dt+\sigma_2 \wk(t)\, dW_2(t),\quad \wk(0)=v_0\in H^{-\delta_1}_\ppp(\CO).
$$
Then,
\begin{enumerate}[(i)]
\item
there exists a number $r_0>0$ and a constant $\tilde{C}_2(T)>0$  such that for any $r\le r_0$,
we have for all $\kk\in\NN$
\renewcommand{\LLm}{{m_0}}
\begin{align*} 
\EE\left\Vert \wk\right\Vert_{L^{m_0}(0,T;H^{r}_{\ppp})}^{\LLm}
\le& \tilde C_2(T)\bigg\{ \EE \vert v_0\vert^{\LLm}_{H^{-\delta_0}_{\ppp}} +  4\kk
\bigg\}.
\end{align*} 
\item
there exists a number $\alpha_0>0$ and a constant $C=C(T)>0$ such that for any $\alpha \le \alpha_0$ we have for all $\kk\in\NN$
\begin{align*} 
\EE\left\Vert \wk\right\Vert^{\LLm}_{\mathbb{W}_{m_0}^\alpha (0,T;H^{-\sigma}_m)}
\le& C(T)\bigg\{ \EE \vert v_0\vert^{\LLm}_{H^{-\delta_0}_{\ppp}} + 4\kk
\bigg\}.
\end{align*} 

\end{enumerate}

\end{proposition}

\begin{proof}
\renewcommand{\LLm}{{m_0}}
We start to show (i).
First, we get by the analyticity of the semigroup for $\delta_0,\delta_1\ge 0$
\begin{align*} 
\EE \left\Vert \wk \right\Vert_{L^{m_0}(0,T;H^{r}_{\ppp})}^{\LLm}
\le&\left[ \int_0^ T  \left\{  t^{-2{m_0}\delta_0} \EE \vert v_0\vert_{H^{-\delta_0}_{\ppp}}^{m_0}\right.\right.\\
&+\EE \left\vert\int_0^ t e^{r_v\Delta (t-s)}\phi_\kk(h(\eta,\xi,s))\,  \eta(s) \xi^ 2(s)\, ds\right\vert_{H^{r}_{\ppp}}^{m_0}
\\&+ \left.\left.\EE \left\vert\int_0^ t e^{r_v\Delta (t-s)} \sigma_2 \wk (s) dW_2(s)\right\vert_{H^{r}_{\ppp}}^{m_0}\,\right\} \,dt\right]
\\
=:& (I + II+III).
\end{align*} 
Since, by the hypotheses,
$$2{m_0}\delta_0<1,
$$
the first term, i.e.\ $I$ is bounded. In particular, we have
\begin{align} \label{zusammen1}
 I\le T^{\frac q{m_0}(1-2{m_0}\delta_0)} \EE \vert v_0\vert^{\LLm}_{H^{-\delta_0}_{\ppp}}.
 \end{align} 

Let us continue with the second term. The smoothing property of the semigroup gives for any $\delta_1\ge 0$
$$
II\le \EE\left[\int_0^ T \left(\int_0^ t (t-s)^ {-\frac 12(r+\delta_1)}\stopp\,  \left\vert \eta(s) \xi^ 2(s)\right\vert_{H^{-\delta_1}_{\ppp}}\, ds\right)^{m_0}\, dt\right].
$$
Using the  Sobolev embedding $L^ 1(\CO)\hookrightarrow  H^{-\delta_1}_{\ppp}(\CO)$, where $\delta_1\ge d(1-\frac 1m)$,
we get
\begin{align*} 
II
\le&\EE\left[\int_0^ T  \left( \int_0^ t (t-s)^ {-\frac 12(r+\delta_1)}\stopp\, \left\vert \eta(s) \xi^ 2(s)\right\vert_{L^\sig}\, ds\right)^{m_0}\, dt\right].
\end{align*} 
Supposing $l\frac { \delta_1+r_0}2<1$
the Young inequality for convolutions gives for
\begin{align} 
\label{Yinq}
\frac 1{m_0} +1=\frac 1l+\frac 1\textkappa\quad \mbox{and} \quad \beta_0=\frac 1l-\frac 12(r+\delta_1)
\end{align} 
and therefore,
\begin{align} \label{tocontinue}
II \le&
C_0 T^{\beta_0\frac \LLm{m_0}}\, \EE\left(\int_0^ T \left\vert \stopp\eta(s) \xi ^ 2(s)\right\vert_{L^\sig}^\textkappa \, ds\right)^ \frac \LLm\textkappa.
\end{align} 
Which can be bounded by H\"older inequality as follows,
\begin{align*} 
\lqq{ \int_0^ T \left\vert \stopp \eta(s) \xi ^ 2(s)\right\vert_{L^\sig}^\textkappa \, ds}
&
\\\le& \sup_{0\le s\le T} \vert\eta(s)\vert_{L^{p^\ast}}^{\mu}\int_0^ T \stopp\vert\xi(s)\vert_{L^m}^{m_0}\, ds\le C(\kk) \Reins.
\end{align*}

Next, let us investigate $III$. We treat the stochastic term by applying \cite[Corollary 3.5 (ii)]{brzezniak}, from which it follows for $\tilde \sigma+r<1$ and $\beta>0$
$$
\EE\left[\int_0^ T  \left\vert\int_0^ t e^{r_v\Delta (t-s)} \sigma_2 \wk (s) \, dW_2(s)\right\vert_{H^{r}_m}^{\LLm}\, dt \right]^\frac q{m_0}\le T^\beta \EE\left[ \int_0^ T \vert\wk (t)\vert_{H^{-\tilde \sigma}_m}^{m_0}\, dt \right].
$$
Due to the Sobolev embedding and interpolation, we know that  whenever
\begin{equation}\label{eq:rho}
\tilde \sigma>\frac d2-\frac dm-\tilde \rho,
\end{equation}
then there exists some $\theta\in(0,1)$ such that
$$
\vert\wk (t)\vert_{H^{-\tilde \sigma}_m}\le \vert\wk (t)\vert_{L^{m}} ^\theta \vert\wk (t)\vert_{H^{\tilde\rho}_{m}}^{1-\theta}\le C\,
\vert\wk (t)\vert_{L^{m}} ^\theta\vert\wk (t)\vert_{H^{\tilde\rho}_2}^{1-\theta}.
$$
Thus, if $\tilde \rho$ satisfies \eqref{eq:rho}, we get that
\begin{align*} 
\EE \left[\int_0^ T \vert\wk (t)\vert_{H^{-\tilde \sigma}_m}^{m_0}\, dt \right]^\frac q{m_0} \le& C\EE \left\{ \left[ \int_0^ T
 \vert\wk (t)\vert_{H^{\tilde \rho}_2}^{m_0}\, dt \right]^{1-\theta}\left[\int_0^ T \vert\wk (t)\vert_{L^m}^{ m_0}\, dt\right]^\theta \right\} ^\frac q{m_0}
\\
\le&C  \EE\left[ \int_0^ T
 \vert\wk (t)\vert_{H^{\tilde \rho}_2}^{m_0}\, dt \right]^{(1-\theta)\frac q{m_0}}\left[\EE \int_0^ T \vert\wk (t)\vert_{L^m}^{ m_0}\, dt\right]^{\theta\frac q{m_0}}
 \\ \le& C(\ep ) \EE\left[ \int_0^ T
 \vert\wk (t)\vert_{H^{\tilde \rho}_2}^{m_0}\, dt \right]^{\frac q{m_0}}+\ep   \EE \Vert\wk \Vert_{L^{m_0}(0,T;L^m)}^{q}.
\end{align*} 
Collecting everything together, choosing $\ep>0$ sufficiently small and subtracting $\ep   \EE \Vert\wk \Vert_{L^{m_0}(0,T;L^m)}^{ m_0}$
on both sides, (i) follows.

The rest of the proof is devoted to item (ii).
Note, that for $s<t$
\begin{align*} 
\lqq{ \wk (t)-\wk (s) =
(e^{r_v\Delta (t-s)}-\operatorname{Id}) \wk (s)} &
\\&{} + \int_s^t  e^{r_v\Delta(t-\tilde s )} \stopp\eta(\tilde s )\xi^2(\tilde s )\, d\tilde s +\int_0^s e^{r_v\Delta(t-\tilde s)}\wk (\tilde s)\, dW_2(\tilde s) .
\end{align*} 
Substituting it in the definition of $\WW^{\alpha}_{m_0} ([0,T];H^{-\sigma}_m(\CO))$, see Appendix \ref{dbouley-space}, we can write
\begin{align*} 
\lqq{ \int_0^T  \int_0^T  \frac{{\vert\wk (t)-\wk (s)\vert_{H^{-\sigma}_{\ppp}} ^ {m_0}}}{{ \vert t-s\vert ^ {1+\alpha {m_0}}}}\,ds\,dt}
&
\\
\le&
2\int_0^T  \int_0^t  \frac{\vert\wk (t)-\wk (s)\vert_{H^{-\sigma}_{\ppp}} ^ {m_0}}{\vert t-s\vert ^ {1+\alpha {m_0}}}\,ds\,dt
\\
\le&
2\int_0^T  \int_0^t  \frac{\vert[e^{r_v\Delta(t-s)}-\operatorname{Id}]\wk (s)\vert_{H^{-\sigma}_{\ppp}} ^ {m_0}}{ \vert t-s\vert ^ {1+\alpha {m_0}}}\,ds\,dt
\\&{}+
2\int_0^T  \int_0^t  \frac{\vert\int_{\tilde s}^t  e^{r_v\Delta(t- s  ) }\stopp\eta( s  )\xi^2( s  )\, d s   \vert_{H^{-\sigma}_{\ppp}} ^ {m_0}}{ \vert t-\tilde s\vert ^ {1+\alpha {m_0}}}\,d\tilde s\,dt
\\
&{}+
2\int_0^T  \int_0^t  \frac{\vert\int_s^t  e^{r_v\Delta(t-\tilde s  )} u(\tilde s)\, dW(\tilde s)   \vert_{H^{-\sigma}_{\ppp}} ^ {m_0}}{ \vert t-s\vert ^ {1+\alpha {m_0}}}\,ds\,dt
=: J+JJ+JJJ.
\end{align*} 
First, let us note that by part (i) we know that there exists a number $r>0$ and a constant $\tilde C(T)>0$ such that
\begin{align} \label{estimate_01}
\EE\left[\int_0^ T \vert\wk (s)\vert^{m_0}_{H^r_{\ppp}}\, ds\right]^\frac q{m_0} \le&
 \tilde C(T)\bigg\{ \EE \vert v_0\vert^{\LLm}_{H^{-\delta_0}_{\ppp}} +  4\kk  \bigg\}
.\end{align} 
Let us  consider the  term $J$.
Since
$$\vert[e^{r_v\Delta (t-s)}-\operatorname{Id}] x\vert_{H^{-\sigma}_{\ppp}}=C\int_0^{t-s} \tilde s^{2(r+\sigma)-1}\, d\tilde s \vert x\vert_{H^r_{\ppp}}\le C (t-s)^{2(r+\sigma)}\vert x\vert_{H^r_{\ppp}},$$ we can infer
\begin{align*} 
J \le&
2\int_0^T  \int_0^t \frac{(t-s)^ {2(r+\sigma){m_0}}\vert\wk (s)\vert_{H^r_\ppp} ^ {m_0}}{ \vert t-s\vert ^ {1+\alpha {m_0}}}\,ds\,dt
.
\end{align*} 
The Young inequality for convolutions (with all exponents equal to $1$) gives
\begin{align*} 
J \le& \left(\int_0^T s^{{m_0}2(r+\delta)-1-\alpha {m_0}}\, ds\right) \cdot \left(\int_0^T \vert\wk (s)\vert^ {{m_0}}_{H^r_{\ppp}}\, ds\right).
\end{align*} 
In particular, since $2(r+\delta)>\alpha$, the right hand side is bounded by estimate \eqref{estimate_01}.

Now, we consider the second term $JJ$.

Jensen's inequality gives for any $q>1$, $q':=\frac{q}{q-1}$,
\begin{align*} 
\lqq{  \left\vert\int_{\tilde s}^ t e^{r_v\Delta(t-s)}\stopp \eta(s) \xi ^ 2(s)\, ds\right\vert_{H^{-\sigma}_{\ppp}}^{q}
}&\\
\le&
 (t-\tilde s)^ {\frac {q}{q'}} \left(\int_{\tilde s}^ t \left\vert \stopp \eta ( s ) \xi^ 2( s)\right\vert^{q}_{H^{-\sigma}_{\ppp}}\, d s\right).
\end{align*} 
Hence,
\begin{align*} 
JJ\le&
\int_0^T  \int_0^t  \frac{\vert\int_{\tilde s}^t e^{r_v\Delta(t-s ) }\stopp \eta( s )\xi^2( s )\, d s  \vert_{H^{-\sigma}_{\ppp}} ^ {m_0}}{ \vert t-\tilde s\vert ^ {1+\alpha {m_0}}}\,d\tilde s\,dt
\\
\le&
\int_0^T  \int_0^t  \frac{(t-\tilde s)^{\frac {m_0} {{m_0}'}}  \int_{\tilde s}^t\stopp  \vert\eta( s )\xi^2( s )\vert^{m_0} _{H^{-\sigma}_{\ppp}}\, d s  }{ \vert t-\tilde s\vert ^ {1+\alpha {m_0}}}\,d\tilde s\,dt
\\
\le&
\left[ \int_0^T  \int_0^t  (t- s)^{\frac {m_0}{{m_0}'}-1-\alpha {m_0} }\,ds\,dt \right] \,\left[\int_0^T  \stopp \vert\eta ( s )\xi^2( s )\vert_{H^{-\sigma}_{\ppp}}^{m_0} \, d s \right].
\end{align*} 
Taking into account that $\alpha<\frac 1 {{m_0}'}$,
integration and the Sobolev embedding gives
\begin{align*} 
JJ \le&
\left[ \int_0^T   t^{\frac {m_0}{m_0'}-\alpha {m_0} }\,dt  \right]\int_0^T  \stopp \vert\eta(s)\xi^2(s)\vert^{m_0}_{H^{-\tilde r }_{\ppp}} \, ds
\\
\le& T   ^{\frac {m_0}{m_0'}+1-\alpha {m_0} }  \int_0^T  \stopp \vert\eta(s)\xi^2(s)\vert^{m_0} _{L^1}\, ds\le  4 T   ^{\frac {m_0}{m_0'}+1-\alpha {m_0} } C( \kk)\Reins.
\end{align*} 
It remains to give an estimate to the second term. We can show the assertion
by the same computations as in the part of the proof for (i), i.e.\ starting at  \eqref{tocontinue}.

Next, let us investigate $JJJ$. Here, again by \cite[Corollary 3.5 (ii)]{brzezniak}, we get for some $\beta>0$ that
$$
\EE \left\vert\int_s^ t  e^{r_v\Delta (t-\tilde s)} \sigma_2 \wk (\tilde s)\, dW_2(\tilde s)\right\vert^{ \LLm}_{H^{-\sigma}_m} \le (t-s)^\beta\,  \EE \left( \int_s^ t \vert\wk (\tilde s )\vert^{m_0}_{H^{-\tilde \sigma}_m}\, d\tilde s\right).
$$
In this way we get
\begin{align*} 
\EE \vert JJJ\vert\le&
2\int_0^T  \int_0^t  \frac{\vert t-s\vert^{m_0-2}}{ \vert t-s\vert ^ {1+\alpha {m_0}}} \EE \int_s^ t \vert\wk (\tilde s )\vert^{m_0}_{H^{-\tilde \sigma}_m}\, d\tilde s\,ds\,dt
\\
\le&
2\int_0^T  \int_0^t  \vert t-s\vert^{m_0(1-\alpha)-3}\left( \EE \int_s^ t \vert\wk (\tilde s )\vert^{m_0}_{H^{-\tilde \sigma}_m}\, d\tilde s\right) ds\,dt.
\end{align*} 
Applying the H\"older inequality, we show that
\begin{align*} 
\EE \vert JJJ\vert\le&
C(T)\, \EE \left(\int_0^ T \vert\wk ( s )\vert^{m_0}_{H^{-\tilde \sigma}_m}\, ds\right).
\end{align*} 
Collecting everything together gives (ii).
\end{proof}

\newcommand{\qst}{{q^\ast}}
\renewcommand{\qst}{{\alpha }}

\begin{tproposition}\label{techpropo}
Let us assume that $(\eta,\xi)\in\subX\subset\CM_\MA^{2,m_0}(\BY,\mathbb{Z})$.
Then, for any $\qst\in[1,\frac {m_0}2)$, $q^\ast\in(1,m_0]$, and $\kappa\in \NN$
 there exists some real numbers $\delta_1,\delta_2\in(0,1]$ and  constants $C_1(\kappa,\Reins),C_2(\kappa,\Reins)>0$ such that
we have
\begin{align*} 
\lqq{\EE \left\Vert  \eta _1\xi _1^2-\eta _2\xi ^2_2\right\Vert_{L^\qst(0,T;L^{1 })}^{q^\ast}}
&\\
\le& C_1(\kappa,\Reins) \times
 \left(\left(\EE \Vert\xi _1-\xi _2\Vert_{\BZ}^{m_0}\right)^ {\delta_1}+ \left(\EE\Vert\eta _1-\eta _2\Vert_{L^{\gamma+1}(0,T;L^{\gamma+1})}^{\gamma+1}\right)^{\delta_2}\right)
\\&{} +C_2(\kappa,\Reins) \times
\left(\EE \Vert\xi _1-\xi _2\Vert_{\BZ}^{m_0}+ \EE\Vert\eta _1-\eta _2\Vert_{L^{\gamma+1}(0,T;L^{\gamma+1})}^{\gamma+1}\right)
.
\end{align*} 
\end{tproposition}
\begin{proof}
Firstly,
 observe   that due to the cutoff function, at integrating with respect to the time we have to
take into account in which interval $s$ belongs to.
Let $\tau_j^u:=\inf_{s>0}\{h(\eta_j,\xi_j,s)\ge \kk\}$ and
$\tau_j^o:=\inf_{s>0}\{h(\eta_j,\xi_j,s)\ge 2\kappa\}$, $j=1,2$.
Distinguishing the different cases, and noting that $\operatorname{Lip}(\phi_\kappa)\le 2\kappa^{-1}$, we obtain
\begin{align*} 
\lqq{
 \left\vert\stose  \eta _1(s)\xi _1^2(s)-\stosz\eta _2(s)\xi ^2_2(s) \right\vert_{L^1}^\alpha \phantom{\Bigg\vert}
}
\\
\le&
\begin{cases}
\vert \eta_1(s)\xi_1^2(s)-\eta_2(s)\xi_2^2(s))\vert_{L^1}^\alpha
&\mbox{ if } 0\le s\le T\wedge \ul{\tau}^u,
\\[2ex]  \left\vert \eta _1(s)\xi _1^2(s)-\eta _2(s)\xi ^2_2(s) \right\vert_{L^1}^\alpha \operatorname{Lip}(\phi_\kappa)^\alpha\\
\times\Big(  \Vert\eta_1-\eta_2\Vert_{L^{\gamma+1}(0,T;L^{\gamma+1})}^\nu+  \Vert\xi_1-\xi_2\Vert_{L^{m_0}(0,T;L^m)}^\nu\Big)
 &
\\[2ex]
\quad \quad{} +2\vert\eta_1(s)\xi_1^2(s)-\eta_2(s)\xi_2^2(s))\vert_{L^1}^{\alpha}&
\mbox{ if }T\wedge\ul{\tau}^u < s \le T\wedge \ol{\tau}^o,
\\[2ex]
0 &\mbox{ if }T\ge s>T\wedge\ol{\tau}^o,
\end{cases}
\end{align*} 
where $\ul{\tau}^u:= \min(\tau_1^u,\tau_2^u)$ and
$\ol{\tau}^o:=\max(\tau_1^o,\tau_2^o)$. Therefore, we may integrate over $[0,T]$.

Secondly, observe that for any $n\in\NN$ and any $a,b\ge 0$, we have
\begin{align*} 
\vert a-b\vert \le&\left\vert \sum_{k=1}^n a^ \frac {k-1}n(a^\frac 1n-b^ \frac 1n) b^ \frac {n-k}n\right\vert
\nonumber
\\
\le&\vert a-b\vert^ \frac 1n \sum_{k=1}^n a^ \frac {k-1}n b^ \frac {n-k}n
\le C \vert a-b\vert^ \frac 1n  (n-1) \left( a^ \frac {n-1} n +  b^ \frac {n-1} n \right),
\end{align*} 
where the last inequality follows from an application of Young's inequality.
The difference can be split in the following way for $s\in [0, \ol{\tau}^o\wedge T]$
\begin{align*} 
\lqq{ \left\vert \xi ^2_1(s,x)\eta _1(s,x)-\xi _2^2(s,x)\eta _2(s,x)\right\vert }
&
\\\le&\vert\xi _1(s,x)-\xi _2(s,x)\vert\xi _1(s,x)+\xi _2(s,x)\vert\vert\eta _1(s,x)\vert  + \vert\xi _1^2(s,x)+\xi _2^2(s,x)\vert\vert\eta _1(s,x)-\eta _2(s,x)\vert.
\end{align*} 
Let $n\in\NN$ be that large such that
$$
\frac 1{n(\gamma+1)}+\frac {n-1}n\cdot\frac 1{p^\ast}+\frac 1m\le 1,\quad \frac 1{n(\gamma+1)}+\frac 1{m_0}\le \frac 1 \alpha,\quad\mbox{and}\quad  \frac 1{n(\gamma+1)}+\frac 1{p^\ast}\le \frac1\alpha.
$$
Therefore, we can write
\begin{align*} 
\lqq{\int_0^ {\ol{\tau}^o\wedge T} \left\vert \eta _1(s)\xi _1^2(s)-\eta _2(s)\xi ^2_2(s)\right\vert_{L^{1 }}^\qst \,ds}
&
\\
\le& \int_0^  {\ol{\tau}^o\wedge T}  \left( \left\vert \eta_1(s)-\eta_2(s)\right\vert_{L^{\gamma+1}}^\frac 1 n  \left\vert \eta_1(s)-\eta_2(s)\right\vert_{L^{p^\ast}}^\frac {n-1}n
\left\vert \xi ^2_1(s)\right\vert_{L^m }\right)^\qst \, ds
\\
\le& \left( \int_0^  {\ol{\tau}^o\wedge T}  \left\vert \eta_1(s)-\eta_2(s)\right\vert^{\gamma+1}_{L^{\gamma+1}}\, ds \right)^{\frac \alpha{(\gamma+1) n}}\\
&\sup_{0\le s\le T} \left\vert \eta_1(s)+\eta_2(s)\right\vert_{L^{p^\ast}}^{\alpha \frac {n-1}n}
\left( \int_0^  {\ol{\tau}^o\wedge T} \left\vert \xi ^2_1(s)\right\vert_{L^m }^{m_0} \, ds\right)^\frac \alpha{m_0}.
\end{align*} 
Taking expectation gives
\begin{align*} 
\lqq{\EE\left\{  \int_0^ {\ol{\tau}^o\wedge T} \left\vert \eta _1(s)\xi _1^2(s)-\eta _2(s)\xi ^2_2(s)\right\vert_{L^{1 }}^\qst \,ds\right\}^{q^\ast}}
&
\\
\le& C(\kappa)\, \left\{ \EE  \int_0^  {\ol{\tau}^o\wedge T}  \left\vert \eta_1(s)-\eta_2(s)\right\vert^{\gamma+1}_{L^{\gamma+1}}\, ds \right\}^{\frac {{q^\ast}\alpha}{(\gamma+1) n}}
\left\{ \EE \sup_{0\le s\le T} \left\vert \eta_1(s)+\eta_2(s)\right\vert_{L^{p^\ast}}^{p_0^\ast}\right\}^{\frac{{q^\ast} \alpha}{p_0^\ast}\frac {n-1}n}
.
\end{align*} 
\noindent
Similarly we get by the H\"older inequality
\begin{align*} 
 \left\vert   \xi ^2_1(s)\eta _2(s)-\xi _2^2\eta _2(s) \right\vert_{L^1 }\le \left\vert  \xi _1(s)-\xi _2(s)\right\vert_{L^m}
\left\vert\xi _1(s)+\xi _2(s)\right\vert_{L^m} \vert\eta _2(s)\vert_{L^{p^\ast}}
.
\end{align*} 
Integration over time gives
\begin{align*} 
\lqq{\int_0^ {\ol{\tau}^o\wedge T}  \left\vert   \xi ^2_1(s)\eta _2(s)-\xi _2^2\eta _2(s) \right\vert_{L^1 }^\qst \,ds}
&
\\
\le&\sup_{0\le s\le T} \vert\eta _2(s)\vert^\alpha_{L^{p^\ast}}
\, \left(\int_0^ {\ol{\tau}^o\wedge T} \left\vert  \xi _1(s)-\xi _2(s)\right\vert_{L^m}^{m_0}\,ds\right)^\frac \alpha{m_0}
\left(\int_0^ {\ol{\tau}^o\wedge T}  \left\vert\xi _1(s)+\xi _2(s)\right\vert_{L^m} ^{m_0}\,ds\right)^\frac \alpha{m_0} .
\end{align*} 
Taking expectation we get
\begin{align*} 
\lqq{\left\{\EE \int_0^ {\ol{\tau}^o\wedge T}  \left\vert   \xi ^2_1(s)\eta _2(s)-\xi _2^2\eta _2(s) \right\vert_{L^1 }^\qst \,ds\right\}^{q^\ast}}
&
\\
\le&C(\kappa) \left\{ \EE\sup_{0\le s\le T} \vert\eta _2(s)\vert_{L^{p^\ast}}^{p^\ast_0}\right\}^\frac {{q^\ast}\alpha}{p_0^\ast}
\, \left\{ \EE \int_0^ {\ol{\tau}^o\wedge T} \left\vert  \xi _1(s)-\xi _2(s)\right\vert_{L^m}^{m_0}\,ds\right\} ^\frac {{q^\ast}\alpha}{m_0}
.\end{align*} 
Finally, take into account that
\begin{align*} 
\lqq{ \EE\Bigg[\left( \int_0^{\ol{\tau}^o\wedge T}
  \left\vert \eta _1(s)\xi _1^2(s)-\eta _2(s)\xi ^2_2(s) \right\vert_{L^1}^\alpha\, ds\right)
  }&
  \\
&{}\times \big(  \Vert\eta_1-\eta_2\Vert_{L^{\gamma+1}(0,T;L^{\gamma+1})}^\nu+\Vert\xi_1-\xi_2\Vert_{L^{m_0}(0,T;L^m)}^\nu\big)\Bigg]^{q^\ast}
\\
\le& \EE \Bigg[\left(\sum_{j=1,2}\sup_{0\le s\le T}
\vert\eta_j(s)\vert_{L^{p^\ast}} \int_0^{\ol{\tau}^o\wedge T}\vert\xi_j(s)\vert_{L^m}^{m_0}\, ds\right)
  \\
&{}\times \big(  \Vert\eta_1-\eta_2\Vert_{L^{\gamma+1}(0,T;L^{\gamma+1})}^\nu+\Vert\xi_1-\xi_2\Vert_{L^{m_0}(0,T;L^m)}^\nu\big)\Bigg]^{q^\ast}
\\
\le& C(\kappa)\, \left\{ \EE \sum_{j=1,2}\sup_{0\le s\le T}
\vert\eta_j(s)\vert_{L^{p^\ast}}^{p_0^\ast}\right\}^\frac {q^\ast}{p_0^\ast}
  \\
&{}\times\Big( \big(  \EE \Vert\eta_1-\eta_2\Vert_{L^{\gamma+1}(0,T;L^{\gamma+1})}^{\gamma+1} \big)^\frac {\nu q^\ast}{\gamma+1}+
\big(\EE \Vert\xi_1-\xi_2\Vert_{L^{m_0}(0,T;L^m)}^{m_0}\big)^\frac {\nu q^\ast}{m_0}
\Bigg]
,
\end{align*} 
the assertion is shown.\end{proof}

The next proposition gives the continuity property of the operator $\Vcal_{\kappa,\Afrak}$.

\begin{proposition}\label{semigroup_continuous}
Let $(\eta_1,\xi_1)\in\subX\subset\CM_\MA^{2,m_0}(\BY,\mathbb{Z})$ and $(\eta_2,\xi_2)\in\subX\subset\CM_\MA^{2,m_0}(\BY,\mathbb{Z})$.
Let $v_\kappa^{(1)}$ and $v_\kappa^{(2)}$ be the solutions to
$$
 d \wk ^{(j)}(t)=[r_v\Delta \wk ^{(j)}(t)+ \phi_\kk(h(\eta_j,\xi_j,t)\,\eta_j(t)\xi_j^2(t)]\,dt+\sigma_2 v_\kappa^{(j)}(t)\,dW_2(t),\,\, t\in(0,T],
$$
with $  \wk ^{(j)}(0)=v_0, j=1,2$.
Then, under the assumptions of Theorem \ref{mainresult} there exists a constant $C=C(\Reins,\kappa)\,>0$ and real numbers $\delta_1,\delta_2>0$ such that
\renewcommand{\LLm}{{m_0}}
\begin{align*} 
\lqq{ \EE\left\Vert \wk ^{(1)}-\wk ^{(2)}\right\Vert^\LLm_{L^{m_0}(0,T;L^m)} }
&\\\nonumber
\le&
C(\Reins,\kappa)\,\left\{  \left[\EE \left\Vert \xi_1-\xi_2\right\Vert^{\LLm}_{L^{m_0}(0,T;L^m)}\right]^{\delta_1}
+ \left[\EE\left\Vert\eta_1-\eta_2\right\Vert^{\gamma+1}_{L^{\gamma+1}(0,T ;L^{\gamma+1})}  \right]^ {\delta_2}
\right\} \,.
\end{align*} 
\end{proposition}

\begin{proof}
First, we get by the analyticity of the semigroup,

\begin{align} \label{continue}
&\EE \left(\int_0^T \left\vert \wk ^{(1)}(t)-\wk ^{(2)}(t)\right\vert^{m_0}_{L^{m }}\, dt\right)
\\
\le& \EE\left(\int_0^T \left\vert\int_0^ t e^{r_v \Delta(t-s)}(\phi_\kk(h(\eta_1,\xi_1,s)) \eta^1(s) \xi_1^ {2}(s)-\phi_\kk(h(\eta_2,\xi_2,s)) \eta _2(s)\xi_2^ {2}(s))\, ds\right\vert_{L^{m}}^{m_0}\, dt\right)\nonumber
\\
\nonumber
&{}+\EE\left( \int_0^T \left\vert\int_0^ t (\sigma_2 \wk ^{(1)}(s)-\sigma_2 \wk ^{(2)}(s))\,dW_2(s)\right\vert_{L^m}^{m_0}\, dt\right)
\\
\nonumber =:& I+II.
\end{align} 
{Let us consider $I$.
Here, we use the  Sobolev embedding $L^ {1}(\CO)\hookrightarrow  H^{-\delta}_{m}(\CO)$, where $\delta\ge d(1-\frac 1m)$.
Next,  the Young inequality for convolution gives for $l,\mu\ge 1$, with $\frac 1{m_0}+1=\frac 1l+\frac 1\textkappa $ and  $\frac \delta 2<\frac 1l$
\begin{align*} 
 I\le T^{\beta_0}\EE \left(\int_0^T \left\vert\stose\eta ^1(s)  \xi_1^ {2}(s)-\stosz\eta_2(s)\xi_2^ {2}(s)\right\vert^\textkappa_{L^1}\, ds\right).
\end{align*} 
Note, the assumptions on $m_0$ and $m$ in the Hypothesis \ref{init} gives the existence of $l$ and $\mu$.
Now, by the technical Proposition \ref{techpropo}
we can infer that there exists a constant $C=C(\Reins,\kappa)\,>0$ and $\delta_1,\delta_2>0$ such that
\begin{align*} 
I \le C(\Reins,\kappa)\, \left(\left(\EE \Vert\xi_1-\xi_2\Vert_{L^{m_0}(0,T;L^m)}^{m_0}\right)^ {\delta_1}+ \left(\EE\Vert\eta_1-\eta_2\Vert_{L^{\gamma+1}(0,T;L^{\gamma+1})}^{\gamma+1}\right)^{\delta_2}\right).
\end{align*} 
It remains to tackle the second term $II$. This can be done by standard arguments using the Burkholder-Davis-Gundy inequality, see Section \ref{sec:BDG}.}
\end{proof}
\renewcommand{\LLm}{{2}}

Next, we shall tackle the continuity of $u_\kappa$ with respect to $\eta$ and $\xi$.

\begin{proposition}\label{continuity}
Let the couples $(\eta_1,\xi_1)\in\subX\subset\CM_\MA^{2,m_0}(\BY,\mathbb{Z})$ and $(\eta_2,\xi_2)\in\subX\subset\CM_\MA^{2,m_0}(\BY,\mathbb{Z})$ be given.
Let $\uk^{(1)}$ and $\uk^{(2)}$ be the solutions to
\begin{align*}
d \uk ^{(j)}(t)&=[r_u\Delta (\uk^{(j)})^{[\gamma]}(t)-\chi\phi_\kappa(h(\eta^{(j)},\xi^{(j)},t))\uk^{(j)}(t)(\xi^{(j)})^2(t)]\,dt+\sigma_1 {u_\kappa}^{(j)}(t)\,dW_1(t),\\
t&\in(0,T],
\end{align*}
with  $\uk ^{(j)}(0)=u_0, j=1,2$. Then, under the assumptions of Theorem \ref{mainresult} there exist  constants $C=C(\kappa,R_1,R_2)>0$ and   $c,\delta_1>0 $ such that
\begin{align*} 
\lqq{ \EE\left[\sup_{0\le s\le T}\left\vert \uk^{(1)} (s)-\uk^{(2)}(s)\right\vert^\LLm_{H^{-1}_2}\right]+ c \EE\left\Vert\uk ^{(1)}-\uk^{(2)}\right\Vert^{\gamma+1}_{L^{\gamma+1}(0,T ;L^{\gamma+1})}}
\\
\le&C(\Reins,\kappa)\left\{ \left(\EE \left\Vert \xi_1-\xi _2\right\Vert_{L^{m_0}(0,T;L^m)}^{\LLm}\right)^{\delta_1}+\EE \left\Vert \xi_1-\xi _2\right\Vert_{L^{m_0}(0,T;L^m)}^{\LLm}\right\}
.
\end{align*}

\end{proposition}

\begin{proof}
Applying the It\^o formula and integration by parts gives
	\begin{align*} 
	\lqq{\vert u^{(1)}(t)-u^{(2)}(t)\vert_{H^{-1}_2}^2}\\
	&+2r_u\int_0^{t}\int_\Ocal ((u^{(1)})^{[\gamma]}(s,x)-(u^{(2)})^{[\gamma]}(s,x))( u^{(1)}(s,x)-u^{(2)}(s,x))\,dx\, ds
	\\
	=&-
2\chi\int_0^{t}\int_\Ocal(\nabla^{-1}( u^{(1)}(s,x)-u^{(2)}(s,x)))\\
	&\qquad\times(\nabla^{-1} (\phi_\kk(h(\eta_1,\xi_1,s))\uk^{(1)}(s,x)\xi_1^2(s,x)-\phi_\kk(h(\eta_2,\xi_2,s)) \uk^{(2)}(s,x)\xi_2^2(s,x))\,dx \,ds
	\\
	&+2\int_0^{t}\left\langle u^{(1)}(s)-u^{(2)}(s),\sigma_1(u^{(1)}(s)-u^{(2)}(s))\,dW_1(s)\right\rangle\\ &\qquad+ \int_0^t\vert\sigma_1(u^{(1)}(s)-u^{(2)}(s))\vert_{L_{\text{HS}}(H_1,H_2^{-1})}^2\,ds
\\
=:& I+ II + III,
\end{align*} 
where $\nabla^{-1}:=-(-\Delta)^{1/2}$.
To deal with the second term and third term, i.e.\ II and III, we apply first the Burkholder Gundy inequality, resp. calculate the trace. Here, we apply Remark \ref{LSSigma}.
Next, we obtain
by the Young inequality and Lemma \ref{lem:porous-medium-inequality}  for $\varepsilon\in \left(0,2^{\frac{2-\gamma}{\gamma+1}} r_u^{\frac{1}{\gamma+1}}\right]$,
	\begin{align*} 
II+III	\le& \varepsilon^{\gamma+1} \EE\int_0^{t} \vert u^{(1)}(s)-u^{(2)}(s)\vert_{L^{\gamma+1}}^{\gamma+1}\, ds
\\
	& +C\EE \int_0^{t} \vert u^{(1)}(s)-u^{(2)}(s)\vert^2_{H_2^{-1}}\,ds
.%
		\end{align*}

Before tackling the first term $I$, let us introduce the following definition.
Let us set
for $k\in\NN$
$$
\tau_{k}(\eta,\xi):=\inf\{ t\ge 0: h(\eta,\xi,t)\ge k\},
\quad \Omega_{k}:=\left\{\omega\in\Omega: \tau_{k}(\eta_1,\xi_1)\le \tau_{k}(\eta_2,\xi_2)\right\},
$$
 and $\Omega_k^c:=\Omega\setminus {\Omega}_k$.

Now, we can decompose the third term into the following summands
\begin{align*}
&\EE \int_0^{t}\int_\Ocal(\nabla^{-1}( u^{(1)}(s,x)-u^{(2)}(s,x)))\\
	&\qquad\times(\nabla^{-1} (\phi_\kk(h(\eta_1,\xi_1,s))\uk^{(1)}(s,x)\xi_1^2(s,x)-\phi_\kk(h(\eta_2,\xi_2,s)) \uk^{(2)}(s,x)\xi_2^2(s,x))\,dx \,ds
\\
&\le \EE \mathbbm{1}_{\Omega_{2\kappa} } \int_0^{t\wedge \tau_{2\kappa}(\eta_2,\xi_2)}
\int_\Ocal(\nabla^{-1}( u^{(1)}(s,x)-u^{(2)}(s,x)))\\
	&\qquad\times(\nabla^{-1} (\phi_\kk(h(\eta_1,\xi_1,s))(\uk^{(1)}(s,x)-\uk^{(2)}(s,x))\xi_1^2(s,x))\,dx \,ds
\\
&{}+\EE \mathbbm{1}_{\Omega_{2\kappa} } \int_0^{t\wedge \tau_{2\kappa}(\eta_2,\xi_2)}
\int_\Ocal(\nabla^{-1}( u^{(1)}(s,x)-u^{(2)}(s,x)))\\
	&\qquad\times(\nabla^{-1} \left(  \phi_\kk(h(\eta_1,\xi_1,s)) -\phi_\kk(h(\eta_2,\xi_2,s))\right)    \uk^{(2)}(s,x)\xi_1^2(s,x))\,dx \,ds
\\
&{}+\EE \mathbbm{1}_{\Omega_{2\kappa} } \int_0^{t\wedge \tau_{2\kappa}(\eta_2,\xi_2)}
\int_\Ocal(\nabla^{-1}( u^{(1)}(s,x)-u^{(2)}(s,x)))\\
	&\qquad\times(\nabla^{-1} (\left(\xi_1^2(s,x)-\xi_2^2(s,x)\right)\phi_\kk(h(\eta_2,\xi_2,s)) \uk^{(2)}(s,x)\,dx \,ds
\\
&{}+\EE \mathbbm{1}_{\Omega_{2\kappa}^c } \int_0^{t\wedge \tau_{2\kappa}(\eta_1,\xi_1)}
\int_\Ocal(\nabla^{-1}( u^{(1)}(s,x)-u^{(2)}(s,x)))\\
	&\qquad\times(\nabla^{-1} (\phi_\kk(h(\eta_1,\xi_1,s))(\uk^{(1)}(s,x)-\uk^{(2)}(s,x))\xi_1^2(s,x))\,dx \,ds
\\
&{}+\EE \mathbbm{1}_{\Omega_{2\kappa}^c } \int_0^{t\wedge \tau_{2\kappa}(\eta_1,\xi_1)}
\int_\Ocal(\nabla^{-1}( u^{(1)}(s,x)-u^{(2)}(s,x)))\\
	&\qquad\times(\nabla^{-1} \left(  \phi_\kk(h(\eta_1,\xi_1,s)) -\phi_\kk(h(\eta_2,\xi_2,s))\right)    \uk^{(2)}(s,x)\xi_1^2(s,x))\,dx \,ds
\\
&{}+\EE \mathbbm{1}_{\Omega_{2\kappa}^c } \int_0^{t\wedge \tau_{2\kappa}(\eta_1,\xi_1)}
\int_\Ocal(\nabla^{-1}( u^{(1)}(s,x)-u^{(2)}(s,x)))\\
	&\qquad\times(\nabla^{-1} (\left(\xi_1^2(s,x)-\xi_2^2(s,x)\right)\phi_\kk(h(\eta_2,\xi_2,s)) \uk^{(2)}(s,x)\,dx \,ds
\\ & =: I_1(t)+I_2(t)+ I_3(t)+I_4(t)+I_5(t)+I_6(t).
\end{align*}
Let us start with $I_1$.
Let us put  for $r,r',\tilde{m},\tilde{m}',p,p'$, and
$s^\ast=\frac{2}{\gamma +1}$,
\begin{align}\label{herewieder}
&\quad r=2\left(1+\frac 1\gamma \right),\quad r'=\frac {2(1+\gamma )}{2+\gamma },
\nonumber 
\\
&\quad  {\tilde m}= {2(\gamma +1)}, \quad{\tilde m}'=\frac{2(\gamma +1)}{2\gamma +1},\quad p=\gamma ,\quad p'=\frac \gamma {\gamma -1}.
\end{align}
 Using duality  with $\tilde{m}\ge \gamma +1$, $2\le r<\gamma +1$, $\frac {1}{\tilde{m}}+\frac 1{\tilde{m}'}\le 1$ and $\frac 1r+\frac 1{r'}\le 1$,
we know for the integrand of $I_{1}$
\begin{align*} 
\lqq{\left\vert\left\langle \Delta^{-\frac 12+\frac {s^\ast}2} (\uk^{(1)}(s)-\uk^{(2)}(s)),\Delta ^{-\frac 12-\frac {s^\ast}2}\left(\phi_\kk(h(\eta_1,\xi_1,\cdot ))\uk^{(1)}\xi_1^2
-\phi_\kk(h(\eta_1,\xi_1,\cdot ))\uk^{(1)}\xi_1^2 \right)\right\rangle\right\vert}
&
\\
\le&
\left\vert \Delta^{-\frac 12+\frac {s^\ast}2} (\uk^{(1)}(s)-\uk^{(2)}(s))\right\vert_{L^r}
\\
&{}\times \left\vert\Delta ^{-\frac 12-\frac {s^\ast}2}(\phi_\kk(h(\eta_1,\xi_1,s ))\uk ^{(1)}(s)\xi_1^2(s))-\phi_\kk(h(\eta_1,\xi_1,\cdot ))\uk ^{(2)}(s)\xi_1^2(s)))\right\vert_{L^{r'}}.
\end{align*} 
By integration in time, and applying, firstly, the H\"older and, secondly,  Young inequality we know for any $\ep_1>0$ there exists a $C(\ep_1)>0$ such that
\begin{equation} \label{hereweitermachen}
\begin{split}
&\int_0^ {\tau_{2\kappa}^1(\eta_1,\xi_1 )\wedge t}\left\vert\left\langle \Delta^{-\frac 12+\frac {s^\ast}2} (\uk^{(1)}(s)-\uk^{(2)}(s)),\Delta ^{-\frac 12-\frac {s^\ast}2}(\phi_\kk(h(\eta_1,\xi_1,\cdot ))\left( \uk ^{(1)}(s)-\uk ^{(2)}(s)\right)\xi_1^2(s))\right\rangle\right\vert\, ds\\
\le&
C(\ep_1)\left(\int_0^ {\tau_{2\kappa}^1(\eta_1,\xi_1 )\wedge t}
 \left\vert \Delta^{-\frac 12+\frac {s^\ast}2} (\uk^{(1)}(s)-\uk^{(2)}(s))\right\vert_{L^r}^{\tilde m}\, ds\right)^{r}
\\ &\qquad {}+ 
\ep_1
\left(\int_0^ {\tau_{2\kappa}^1(\eta_1,\xi_1 )\wedge t}\phi^{{\tilde m}'}_\kk(h(\eta_1,\xi_1,s ))\left\vert(\uk ^{(1)}(s)-\uk ^{(2)}(s))\xi^2_1(s)\right\vert^{{\tilde m}'}_{H^{-1-{s^\ast}}_{r'}}\, ds\right)^\frac {r'}{{\tilde m}'}.
\end{split}
\end{equation} 
Let us analyze the first term on the right hand side.
Take some
 $\theta\in (0,2/\tilde m)$ and let 
$$
\frac 1{r''}=\frac \theta2+\frac {1-\theta}{\gamma+1},\quad \tilde s=-\theta+(1-\theta)(s^\ast-1).
$$ 
Then, we know by interpolation that 
$$
\vert u\vert_{H^{s^\ast-1}_r}\le \vert u \vert^\theta _{H^{-1}_2}\vert u\vert^{1-\theta}_{H^{\tilde s}_{r''}}
$$
This gives
\begin{align*}
&\EE\left\Vert \Delta^{-\frac 12+\frac {s^\ast}2} (\uk^{(1)}-\uk^{(2)})\right\Vert^r_{L^{\tilde m}(0,T;{L^r})}
\le C\EE \left\Vert \uk^{(1)}-\uk^{(2)}\right\Vert^r_{L^{\tilde m}(0,T;H^{s^\ast-1}_r)}
\nonumber
\\
&
\le C\EE\left( \int_0^ t\vert \uk^{(1)}(s)-\uk^{(2)}(s)\vert^{\theta \tilde m}_{H^{-1}_2} \vert \uk^{(1)}(s)-\uk^{(2)}(s)\vert^{(1-\theta) \tilde m}_{H^{\tilde s}_{r''})}\, ds\right)^\frac r{\tilde m}
\nonumber.
\end{align*}
The H\"older inequality gives
\begin{align*}
&\EE\left\Vert \Delta^{-\frac 12+\frac {s^\ast}2} (\uk^{(1)}-\uk^{(2)})\right\Vert^r_{L^{\tilde m}(0,T;{L^r})}
\nonumber
\\
&
\le C\EE \left(\int_0^ t \vert \uk^{(1)}(s)-\uk^{(2)}(s)\vert^{2}_{H^{-1}_2}\, ds\right)^{\frac {\theta{\tilde m}}2\frac r{\tilde m}} \,\left(\int_0^ t \vert \uk^{(1)}(s)-\uk^{(2)}(s)\vert^{(1-\theta) \tilde m\frac 2{2-\theta\tilde m}}_{H^{\tilde s}_{r''}}\, ds\right)^ {\frac r{\tilde m}\, \frac {2-\theta\tilde m}2}
\\
&
\le C\EE \left(\int_0^ t \vert \uk^{(1)}(s)-\uk^{(2)}(s)\vert^{2}_{H^{-1}_2}\, ds\right)^{\frac {r\theta}2} \,
 \Vert \uk^{(1)}-\uk^{(2)}(s)\Vert_{L^{(1-\theta) \tilde m\frac 2{2-\theta\tilde m}}(0,T;H^{\tilde s}_{r''} )}^{r(1-\theta)}
.
\end{align*}
By the Young inequality, we know for any $\ep_2>0$ there exists some $C(\ep_2)>0$ such that
\begin{align*}
&\EE\left\Vert \Delta^{-\frac 12+\frac {s^\ast}2} (\uk^{(1)}-\uk^{(2)})\right\Vert^r_{L^{\tilde m}(0,T;{L^r})}
\nonumber
\\
&
\le C(\ep_2)\EE \int_0^ t\vert \uk^{(1)}(s)-\uk^{(2)}(s)\vert^{2}_{H^{-1}_2}\, ds +
\ep _2 \EE \Vert \uk^{(1)}-\uk^{(2)}(s)\Vert_{L^{(1-\theta) \tilde m\frac 2{2-\theta\tilde m}}(0,T;H^{\tilde s}_{r''} )}^{r(1-\theta)\frac {2-r\theta}2}.
\end{align*}
Observe, for small $\theta$, we know $r''\ge r$, $\tilde s\le  (s^\ast-1)$, and $\tilde m'':=(1-\theta) \tilde m\frac 2{2-\theta\tilde m}\ge \tilde m$. 
Thus we know, since
$$
\frac 1{r}> \frac{-(s^\ast-1)}2+\frac{1-s^\ast}{\gamma +1},\quad \frac 1{\tilde m''}\ge  \frac{1-s^\ast}{\gamma +1}.
$$
that the same holds for $r'',\tilde m''$, and $\tilde s$.
By Proposition \ref{interpolation_11}, we know, if $\gamma>1$ then there exists some $\theta>0$ such that
\begin{align*} \lqq{
 \Vert \uk^{(1)}-\uk^{(2)}\Vert_{L^{(1-\theta) \tilde m\frac 2{2-\theta\tilde m}}(0,T;H^{\tilde s}_{r''} )}^{r(1-\theta)\frac {2-r\theta}2}
}&
\\
\le&  \sup_{0\le s\le T}\vert\uk^{(1)}(s)-\uk^{(2)}(s)\vert^2_{H^{-1}_2}+\int_0^ T\vert\uk^{(1)}(s)-\uk^{(2)}(s)\vert_{L^{\gamma +1}}^{\gamma +1}\, ds.
\end{align*} 
This means, that later on, we can cancel this term by the LHS.
Let us analyze the second term of the RHS in \eqref{hereweitermachen}, i.e.\ 
$$
\ep_1
\left(\int_0^ {\tau_{2\kappa}^1(\eta_1,\xi_1 )\wedge t}\left(\phi_\kk(h(\eta_1,\xi_1,s ))\right)^{{\tilde m}'}\left\vert(\uk ^{(1)}(s)-\uk ^{(2)}(s))\xi^2_1(s)\right\vert^{{\tilde m}'}_{H^{-1-{s^\ast}}_{r'}}\, ds\right)^\frac {r'}{{\tilde m}'}.
$$
First,  we use \cite{runst} p. 171, Theorem 1, where we put
 $s_1=-s^{\ast\ast}<0$ and $s_2=s^{\ast\ast}+\ep$, $s^{\ast\ast}=s^\ast/2$, $\ep>0$ very small. In this way we get
$$
\left\vert(\uk ^{(1)}(s)-\uk ^{(2)}(s))\xi_1^2(s)\right\vert_{H^{-1-{s^\ast}}_{r'}}\le C\, \left\vert\uk ^{(1)}(s)-\uk ^{(2)}(s)\right\vert_{H^{s_1}_{r'}}
\left\vert\xi_1^2(s)\right\vert_{H^{s_2}_{r'}}.
$$
or, applying the H\"older inequality in time,
\begin{align*} 
\lqq{ \left( \int_0^ {\tau_{2\kappa}^1(\eta_1,\xi_1 )\wedge t} \left\vert(\uk ^{(1)}(s)-\uk ^{(2)}(s))\xi_1^2(s)\right\vert^{{\tilde m}'}_{H^{-1-{s^\ast}}_{r'}}\, ds\right)^\frac{r'}{{\tilde m}'}}
&
\\
\le&  C\,\left( \int_0^ {\tau_{2\kappa}^1(\eta_1,\xi_1 )\wedge t}  \left\vert\uk ^{(1)}(s)-\uk ^{(2)}(s)\right\vert_{H^{s_1}_{r'}}^{p{\tilde m}'}\,ds \right)^\frac {r'}{{\tilde m}'p}
\left( \int_0^{\tau_{2\kappa}^1(\eta_1,\xi_1 )\wedge t}\left\vert\xi_1^2(s)\right\vert^{{\tilde m}'p'}_{H^{s_2}_{r'}}\,ds \right)^\frac {r'}{{\tilde m}'p'}.
\end{align*} 
Now, since for our choice of $r'>2$, $s_1$, and $m'p$, 
we know that 
$$
(i)\quad \frac 1{r'}\ge \frac{-s_1}
2+\frac{1+s_1}{\gamma +1}=\frac {s^\ast}{4}+\frac {1-s^\ast/2}{\gamma+1},\quad \frac 1{m'p}\ge  \frac{1+s_1}{\gamma +1}=\frac{1-s^\ast/2}{\gamma +1},
$$
we have 
\begin{align*} \lqq{
\EE \big\Vert\uk ^{(1)}-\uk ^{(2)}\big\Vert_{L^{p\tilde m'}(0,T;H^{s_1}_{r''})}^{r''}}
\\
\le& C
\EE \left( \sup_{0\le s\le t}\vert\uk^{(1)}(s)-\uk^{(2)}(s)\vert^2_{H^{-1}_2}+\int_0^ t\vert\uk^{(1)}(s)-\uk^{(2)}(s)\vert_{L^{\gamma +1}}^{\gamma +1}\, ds\right).
\end{align*} 
Next, we tackle $\left\vert\xi_1^2(s)\right\vert^{\tilde m'p'}_{H^{s_2}_{r'}}$. Here, we have by \cite[p. 364]{runst}
for
$$ s_2+2\left(\frac d{\tilde r}-s_2\right)=2 \frac  d{\tilde r}-s_2=\frac d{r'}
\quad \Rightarrow
 \quad
 \frac d{\tilde r}=\frac 12\left(\frac d{r'}+s_2\right)=
 \frac 12\left(\frac{c(2+\gamma) }{2+2\gamma }+s_2\right),
$$
the estimate
$$
\left\vert\xi_1^2(s)\right\vert^{\tilde m'p'}_{H^{s_2}_{r'}}\le C\left\vert\xi_1(s)\right\vert^{2\tilde m'p'}_{H^{s_2}_{\tilde r}}
.
$$
Let us put $s_2=\frac {1}{(1/2+\gamma )}$. Then, $\tilde r=\frac {2(1+\gamma)}{2+\gamma}$.
Furthermore, note that $r\ge 2$, $r'\le 2$ and $\tilde r\le 2$ for $s_2=\frac {1}{(1/2+\gamma )}$.
Next,
 the Sobolev embedding $H^{\delta}_2(\CO)\hookrightarrow  H^{s_2}_{\tilde r}(\CO)$ with $\delta\ge s_2+(\frac d2-\frac d{\tilde r})=\frac {1+d\gamma }{4(1+\gamma )}\le\frac d4$. In particular, $\delta=\frac {3+2\gamma}{2+6\gamma+4\gamma^2}$.
 Now, to estimate the term by $\BH$ we need that
$$
(ii)\quad \delta=(1-\theta)(\rho+1)=\rho+(1-\theta),\qquad \frac 1{l}=\frac {1-\theta}2,\qquad l\ge \min(2,\tilde m'p')
.$$
Hence, $\rho\ge\delta-\frac 2{2\tilde m'p'}=\frac {1+d\gamma }{4(1+\gamma )}-\frac {(1+\gamma )(1+2\gamma )}{2\gamma (2+\gamma )}<0$.
In this way we get
$$ \left\Vert\xi_1^2\right\Vert_{L^{p'\tilde m'}(0,T;H^{s_1}_{r'})}\le c \Vert \xi_1\Vert^2_{L^{2\tilde m'p'}(0,T;H^{\delta}_2)}\le c \Vert \xi_1\Vert^2_{\BH}.
$$
It follows by taking into account the cut-off function
\begin{align*} 
\lqq{ \left( \int_0^  {\tau_{2\kappa}^1(\eta_1,\xi_1 )\wedge t}\phi_\kk(h(\eta_1,\xi_1,\cdot ))\left\vert(\uk ^{(1)}(s)-\uk ^{(2)}(s))\xi_1^2(s)\right\vert^{\tilde m'}_{H^{-1-{s^\ast}}_{r'}}\, ds\right)^{r'}}
&
\\
\le&  C\, \Vert \phi_\kk(h(\eta_1,\xi_1,\cdot ))\mathbbm{1}_{[0, {\tau_{2\kappa}^1(\eta_1,\xi_1 )\wedge t})}\xi_1\Vert^{2r'}_{\BH}
\\
&{}
\times \EE \left( \sup_{0\le s\le t}\vert\uk^{(1)}(s)-\uk^{(2)}(s)\vert^2_{H^{-1}_2}+\int_0^ t\vert\uk^{(1)}(s)-\uk^{(2)}(s)\vert_{L^{\gamma +1}}^{\gamma +1}\, ds\right)^{ {r'}}
\\
\le& C(\kappa)\,
\left\{ \EE \left( \sup_{0\le s\le t}\vert\uk^{(1)}(s)-\uk^{(2)}(s)\vert^2_{H^{-1}_2}+\int_0^ t\vert\uk^{(1)}(s)-\uk^{(2)}(s)\vert_{L^{\gamma +1}}^{\gamma +1}\, ds\right)^{r'}\right\}.
\end{align*} 
In this way, we get for $I_1$
\begin{align*} 
\lqq{\EE\vert I_1(t)\vert\le C(\ep_1,\ep_2)
\EE \int_0^ t \vert\uk^{(1)}(s)-\uk^{(2)}(s)\vert^2_{H^{-1}_2}\, ds
}
&
\\
&{} +
\ep_1 \EE \left( \sup_{0\le s\le t}\vert\uk^{(1)}(s)-\uk^{(2)}(s)\vert^2_{H^{-1}_2}+\int_0^ t\vert\uk^{(1)}(s)-\uk^{(2)}(s)\vert_{L^{\gamma +1}}^{\gamma +1}\, ds\right)
\\&{}+
\ep_2 C(\kappa)\,
\left\{ \EE \left( \sup_{0\le s\le t}\vert\uk^{(1)}(s)-\uk^{(2)}(s)\vert^2_{H^{-1}_2}+\int_0^ t\vert\uk^{(1)}(s)-\uk^{(2)}(s)\vert_{L^{\gamma +1}}^{\gamma +1}\, ds\right)
^{{r'}} \right\}
.
\end{align*} 
Next, we have to tackle $I_2(t)$. Again
take $r,r',\tilde m,\tilde m',p$, and $p'$ defined as in \eqref{herewieder}.
Again, using duality  
we know for the integrand of $I_{2}$
\begin{align*} 
&\left\vert\left\la \Delta^{-\frac 12+\frac {s^\ast}2} (\uk^{(1)}(s)-\uk^{(2)}(s)),
\Delta ^{-\frac 12-\frac {s^\ast}2}\left(\phi_\kk(h(\eta_1,\xi_1,\cdot ))\uk^{(1)}\xi_1^2
-\phi_\kk(h(\eta_1,\xi_1,\cdot ))\uk^{(1)}\xi_1^2 \right)\right\ra\right\vert
\\
\le&
\left\vert \Delta^{-\frac 12+\frac {s^\ast}2} (\uk^{(1)}(s)-\uk^{(2)}(s))\right\vert_{L^r}\\
&\qquad
\times\left\vert\Delta^{-\frac 12-\frac {s^\ast}2}\left(  \phi_\kk(h(\eta_1,\xi_1,s)) -\phi_\kk(h(\eta_2,\xi_2,s))\right)    \uk^{(2)}(s)\xi_1^2(s))\right\vert_{L^{r'}}
\end{align*} 
By the Young inequality we know for any $\ep_1>0$ there exists a constant $C(\ep_1)>0$ such that 
\begin{align*}\label{mmm01}
&\EE \mathbbm{1}_{\Omega_{2\kappa} } \int_0^{t\wedge \tau_{2\kappa}(\eta_2,\xi_2)}
\int_\Ocal(\nabla^{-1}( u^{(1)}(s,x)-u^{(2)}(s,x)))
\\
	&\qquad\times(\nabla^{-1} \left(  \phi_\kk(h(\eta_1\xi_1,s)) -\phi_\kk(h(\eta_2,\xi_2,s))\right)    \uk^{(2)}(s,x)\xi_1^2(s,x))\,dx \,ds
\\
&\le \ep \EE\Vert u^{(1)}-u^{(2)}\Vert^{r}_{L^{m}(0,T;H^{s^\ast-1}_r)}+C(\ep)\EE \Vert\xi_1-\xi_2\Vert^{r'}_{\BH}  \Vert  \mathbbm{1}_{\Omega_{2\kappa} }  \uk^{(2)}\xi_1^2\Vert^{r'}
_{L^{m'}(0,{t\wedge \tau_{2\kappa}(\eta_2,\xi_2)};H^{-s^\ast-1}_{r'})},
\end{align*}
Next, we follow  the calculation as done before.
First,  we use again \cite[p. 171, Theorem 1]{runst}, where we put
 $s_1=-s^{\ast\ast}<0$ and $s_2=s^{\ast\ast}+\ep$, $s^{\ast\ast}=s^\ast/2$, $\ep>0$ very small. In this way we get 
$$
\left\vert\uk ^{(2)}(s)\xi_1^2(s)\right\vert_{H^{-1-{s^\ast}}_{r'}}\le C\, \left\vert\uk ^{(2)}(s)\right\vert_{H^{s_1}_{r'}}
\left\vert\xi_1^2(s)\right\vert_{H^{s_2}_{r'}}.
$$
Next,  applying the H\"older inequality in time,
\begin{align*} 
\lqq{ \left( \int_0^ {\tau_{2\kappa}^1(\eta_2,\xi_1 )\wedge t} \left\vert\uk ^{(2)}(s)\xi_1^2(s)\right\vert^{{\tilde m}'}_{H^{-1-{s^\ast}}_{r'}}\, ds\right)^\frac{r'}{{\tilde m}'}}
&
\\
\le&  C\,\left( \int_0^ {\tau_{2\kappa}^1(\eta_2,\xi_1 )\wedge t}  \left\vert\uk ^{(2)}(s)\right\vert_{H^{s_1}_{r'}}^{p{\tilde m}'}\,ds \right)^\frac {r'}{{\tilde m}'p}
\left( \int_0^{\tau_{2\kappa}^1(\eta_1,\xi_1 )\wedge t}\left\vert\xi_1^2(s)\right\vert^{{\tilde m}'p'}_{H^{s_2}_{r'}}\,ds \right)^\frac {r'}{{\tilde m}'p'}.
\end{align*} 
Again, by the choice of $r,r',\tilde m,\tilde m',p$, and $p'$, we know that even for a $r''>r'$ we have 
\begin{align*}
&\left( \int_0^ {\tau_{2\kappa}^1(\eta_1,\xi_1 )\wedge t}  \left\vert\uk ^{(2)}(s)\right\vert_{H^{s_1}_{r''}}^{p{\tilde m}'}\,ds \right)^\frac {r''}{{\tilde m}'p}\\
\le& 
C\left( \sup_{0\le s\le  {\tau_{2\kappa}^1(\eta_1,\xi_1 )\wedge t}}  \left\vert\uk ^{(2)}(s)\right\vert_{H^{-1}_2}^2
+
 \int_0^ {\tau_{2\kappa}^1(\eta_1,\xi_1 )\wedge t}  \left\vert\uk ^{(2)}(s)\right\vert_{L^{\gamma+1}}^{\gamma_1}\,ds \right).
\end{align*}
Again, similarly we know
$$
\left( \int_0^{\tau_{2\kappa}^1(\eta_1,\xi_1 )\wedge t}\left\vert\xi_1^2(s)\right\vert^{{\tilde m}'p'}_{H^{s_2}_{r'}}\,ds \right)^\frac {r'}{{\tilde m}'p'}
\le \Vert\xi_2\Vert_{\BH}^{r'} .
$$
Now, taking into account the cut off function we know 
\begin{align*} \lqq{
\EE  \mathbbm{1}_{\Omega_{2\kappa} }   \Vert\xi_1-\xi_2\Vert_{\BH^{{t\wedge \tau_{2\kappa}(\eta_2,\xi_2)}}}^{r'}  \Vert  \uk^{(2)}\xi_1^2\Vert^{r'}_{L^{m'}(0,{t\wedge \tau_{2\kappa}(\eta_2,\xi_2)};H^{-s^\ast-1}_{r'})}}
&
\\
\le&
\EE\Bigg\{    \mathbbm{1}_{\Omega_{2\kappa} } \Vert\xi_1-\xi_2\Vert_{\BH^{{t\wedge \tau_{2\kappa}(\eta_2,\xi_2)}}}^{r'}
\Vert h(\xi_2,\cdot)\xi_2\Vert_{\BH}^{r'}
\\ &{}\times  C\left( \sup_{0\le s\le  {\tau_{2\kappa}^1(\eta_1,\xi_1 )\wedge t}}  \left\vert\uk ^{(2)}(s)\right\vert_{H^{-1}_2}^2
+
 \int_0^ {\tau_{2\kappa}^1(\eta_1,\xi_1 )\wedge t}  \left\vert\uk ^{(2)}(s)\right\vert_{L^{\gamma+1}}^{\gamma_1}\,ds \right)^{\frac {r'}{r''}}\Bigg\}
\\
\le& C(\kappa) \left(\EE \Vert\xi_1-\xi_2\Vert_{\BH} ^2\right)^{1-\frac {r'}{r''}} 
\\
&{}\times \left\{ \EE\left( \sup_{0\le s\le  {\tau_{2\kappa}^1(\eta_1,\xi_1 )\wedge t}}  \left\vert\uk ^{(2)}(s)\right\vert_{H^{-1}_2}^2
+
 \int_0^ {\tau_{2\kappa}^1(\eta_1,\xi_1 )\wedge t}  \left\vert\uk ^{(2)}(s)\right\vert_{L^{\gamma+1}}^{\gamma_1}\,ds \right)\right\}^{\frac {r'}{r''}}
.
\end{align*} 
Since the operator is invariant, the second expectation value is smaller than $R_1$.
Next, let us treat $I_3$.
\begin{align*} 
\lqq{ \EE \mathbbm{1}_{\Omega_{2\kappa} } \int_0^{t\wedge \tau_{2\kappa}(\eta_2,\xi_2)}
\int_\Ocal(\nabla^{-1}( u^{(1)}(s,x)-u^{(2)}(s,x)))}
&
\\
	&\qquad\times(\nabla^{-1} (\left(\xi_1^2(s,x)-\xi_2^2(s,x)\right)\phi_\kk(h(\eta_2,\xi_2,s)) \uk^{(2)}(s,x)\,dx \,ds
\\
\le& \ep \EE\Vert u^{(1)}-u^{(2)}\Vert^{\gamma +1}_{L^{\gamma +1}(0,T;L^{\gamma +1})}+
C(\ep)  \Vert  \mathbbm{1}_{\Omega_{2\kappa} }  \uk^{(2)}(\xi_1^2-\xi_2^2)\Vert^{\frac {\gamma +1}\gamma }_{L^{\frac {\gamma +1}\gamma }(0,{t\wedge \tau_{2\kappa}(\eta_2,\xi_2)};H^{-2}_{\frac {\gamma +1}\gamma })}.
\end{align*} 
By similar calculation as before, we get
\begin{align*} 
\lqq{ \EE \mathbbm{1}_{\Omega_{2\kappa} } \int_0^{t\wedge \tau_{2\kappa}(\eta_2,\xi_2)}
\int_\Ocal(\nabla^{-1}( u^{(1)}(s,x)-u^{(2)}(s,x)))}
&
\\
	&\qquad\times(\nabla^{-1} (\left(\xi_1^2(s,x)-\xi_2^2(s,x)\right)\phi_\kk(h(\eta_2,\xi_2,s)) \uk^{(2)}(s,x)\,dx \,ds
\\
\le& \ep \EE\Vert u^{(1)}-u^{(2)}\Vert^{\gamma +1}_{L^{\gamma +1}(0,T;L^{\gamma +1})}\\
&\qquad+
C(\ep) \EE\left\{ \sup_{0\le s\le T}\vert \uk^{(2)}(s)\vert^{\frac {\gamma +1}\gamma }_{L^p}
 \Vert\xi_1^2-\xi_2^2\Vert^{\frac {\gamma +1}\gamma }_{L^{\frac {\gamma +1}\gamma }(0,{t\wedge \tau_{2\kappa}(\eta_2,\xi_2)};L^r)}\right\}
 \\
\le& \ep \EE\Vert u^{(1)}-u^{(2)}\Vert^{\gamma +1}_{L^{\gamma +1}(0,T;L^{\gamma +1})}\\
&\qquad+
C(\ep,\kappa)
 \EE\left\{ \sup_{0\le s\le T}\vert \uk^{(2)}(s)\vert^{\frac {\gamma +1}\gamma }_{L^p}\right\}
\left\{\EE \Vert\xi_1-\xi_2\Vert_{L^m(0,T;L^{2r})}\right\}
 .
\end{align*} 
$I_4$, $I_5$, and $I_6$ can be estimated by similar steps.

Now, we can collect all terms and get
	\begin{align*} 
&	\EE \vert u^{(1)}(t)-u^{(2)}(t)\vert_{H^{-1}_2}^2\\
&\qquad+2r_u\EE\int_0^{t}\int_\Ocal ((u^{(1)})^{[\gamma]}(s,x)-(u^{(2)})^{[\gamma]}(s,x))( u^{(1)}(s,x)-u^{(2)}(s,x))\,dx\, ds
	\\
\le& \varepsilon^{\gamma+1} \EE\int_0^{t} \vert u^{(1)}(s)-u^{(2)}(s)\vert_{L^{\gamma+1}}^{\gamma+1}\, ds
 +C\EE \int_0^{t} \vert u^{(1)}(s)-u^{(2)}(s)\vert^2_{H_2^{-1}}\,ds
\\
&{}+ \ep \EE\Vert u^{(1)}-u^{(2)}\Vert^{\gamma +1}_{L^{\gamma +1}(0,T;L^{\gamma +1})}\\
&{}\qquad+C(\ep)\left\{ \EE \Vert\xi_1-\xi_2\Vert^2_{\BH} \right\}^{\delta_1}\left\{ \EE \Vert  \mathbbm{1}_{\Omega_{2\kappa} }  \uk^{(2)}\xi_1^2\Vert^{\frac {\gamma +1}\gamma }_{L^{\frac {\gamma +1}\gamma }(0,{t\wedge \tau_{2\kappa}(\eta_2,\xi_2)};H^{-2}_{\frac {\gamma +1}\gamma })}\right\}^{\delta_2}
\\
&{}+\ep \EE\Vert u^{(1)}-u^{(2)}\Vert^{\gamma +1}_{L^{\gamma +1}(0,T;L^{\gamma +1})}+
C(\ep,\kappa)
 \EE\left\{ \sup_{0\le s\le T}\vert \uk^{(2)}\vert_{L^p}\right\}
\left\{\EE \Vert\xi_1-\xi_2\Vert_{L^m(0,T;L^{2r})}\right\}
 .
\end{align*} 

Term $I$ cancels with the corresponding term on the left hand side.
We may apply Gronwall's lemma in order to deal with the term $III$.
\end{proof}

\begin{proposition}\label{CC2}
For any initial condition $(u_0,v_0)$ satisfying Hypothesis \ref{init},
\begin{itemize}
\item[(i)] there exists $r=r(T,\gamma)>0$ such that for any $(\eta,\xi)\in \subX$, we have
\begin{align*}
\EE\int_0^T \vert \uk(s)\vert_{H_{\gamma+1}^r}^{\gamma+1}\, ds
\le C \Reins;
\end{align*}
\item[(ii)] we have that $t\mapsto u_\kappa (t)$ is $\P$-a.s. strongly continuous in $H_2^{-1}(\CO)$.
\end{itemize}
\end{proposition}

\begin{proof}
Let us prove (i) first.
We know from Proposition \ref{uniformlpbounds} with $p=1$, that for any $(\eta,\xi)\in \subX$ we have
\begin{align*} 
\EE \int_0^ T \left\vert\uk^{\left[\frac {\gamma-1}2\right]}(s) \nabla \uk(s)\right\vert^2_{L^2}\, ds \le C\Reins.
\end{align*} 
On the other side, we know by Proposition \ref{runst1} that we have for any $\theta<\frac 2{p+\gamma}$,
\begin{align*} 
\Vert\uk\Vert_{L^{p+\gamma}(0,T;H^{\theta}_{p+\gamma})}^{p+\gamma} \le C \left(\int_0^T \left\vert\uk^{\left[\frac {p+\gamma-2}2\right]}(s) \nabla \uk(s)\right\vert^2_{L^2}\, ds+\int_0^T \vert u\vert_{L^2}^{p+\gamma}\,ds\right),
\end{align*} 
where the LHS is bounded for $p=1$, noting that Proposition \ref{uniformlpbounds} with $p=\gamma$ gives
\[\EE\left[\sup_{t\in [0,T]}\vert u_\kappa(t)\vert^{\gamma+1}_{L^{\gamma+1}}\right]\le C.\]

The assertion (ii) follows from \cite[Theorem 4.2.5]{weiroeckner}.
\end{proof}

\subsection{The system \eqref{eq:cutoffu}--\eqref{eq:cutoffv} and uniform bounds on the stopping time}

Within this section let us assume that $\{(\bar{u}_\kappa,\bar{v}_\kappa)\;\colon\; \kk\in\NN\}$  the family of stochastic processes constructed in Step III of the proof of Theorem \ref{mainresult}. The constants $\Reins,\Rzwei$ and $\Rdrei$ are given as in \eqref{constantR1},  \eqref{constantR2}, and  \eqref{constantR3}, respectively.

\begin{proposition}\label{uniform1}
Fix $p\ge 1$ and suppose $u_0\in L^{p+1}(\CO )$. Then,
 there exists a constant $C_0(p,T)>0$ such that we have for all $\kk\in\NN$
\begin{align*} 
\lqq{\hspace{-4ex}\EE \left[\sup_{0\le s\le T} \vert\bar{u}_\kappa (s)\vert_{L^{p+1} }^{p+1}\right]
+ \gamma p(p+1)r_u\EE \int_0^ \TInt\int_\CO  \vert\bar{u}_\kappa(s,x)\vert^{p+\gamma-2} \vert\nabla \bar{u}_\kappa (s,x)\vert^2\, dx \, ds} &
\\\nonumber &{} +(p+1)\chi\EE \int_0^\TInt \int_\CO  \vert\bar{u}_\kappa(s,x)\vert^{p+1}\vert\bar{v}_\kappa(s,x)\vert^2\, dx \, ds \le C_0(p,T)\,\left( \EE\vert u_0\vert_{L^{p+1}}^{p+1}+1\right).
\end{align*} 
\end{proposition}

\begin{proof}
Follows from an application of Fatou's lemma and nonnegativity of $u_\kappa$ and $v_\kappa$ in the proof of Proposition \ref{uniformlpbounds}.
\end{proof}

\begin{proposition}\label{propvarational2}
Let $\rho<1-\frac d2$ and assume that the Hypotheses of Theorem \ref{mainresult} hold. There exists a generic constant ${C}(T)>0$, a number  $l\ge \frac 2{\vert\rho\vert}$ and $\delta_1,\delta_2\in(0,1)$ with $\delta_1+\delta_2=1$ such that
for all $\kk\in\NN$
we have
\begin{align} \label{estimatesol3}
\EE  \Vert\bar{v}_\kappa\Vert_{\BH}^{m_0}
\le&
C(T)\left(
 \EE\vert v_0\vert_{H^{-\delta_0}_{m}}^{m_0}+ \Reins^{\delta_1}\,
\Rzwei^{\delta_1}\right).
\end{align} 
\end{proposition}
Before proving Proposition \ref{propvarational2}, we consider the following lemma, which will be essential.
\begin{tlemma}\label{dasauch}
For all $1\le \alpha<2$, $-\rho<\frac {2(2-\alpha)}\alpha$, and
for any $r^\ast\ge 1$
there exist generic constants ${C}(T)>0$
and  $\delta_1,\delta_2\in(0,1)$ with $\delta_1+\delta_2=1$, such that for any pair of nonnegative processes $\eta,\xi:[0,T]\times \Omega\times \CO \rightarrow[0,\infty)$ we have
	\begin{align}  \label{Eq:Productestimate} \hspace{+2ex}
\EE \Vert \eta\xi^2\Vert^ {r^\ast}_{L^\alpha(0,T;L^1)}
\le&
 C(T)
 \left\{\mathbb{E}\left(\int_0^T\vert\eta^p(s) \xi^2(s)\vert_{L^1}\,ds \right)^\frac{ r^\ast p }{\alpha }\right\}^{\delta_1}
 \left(\mathbb{E} \Vert\xi \Vert_{\BH}^{r^\ast  }\right)^{\delta_2}.\end{align} 

\end{tlemma}

 \begin{proof}
\renewcommand{\pst}{p}
\renewcommand{\ns}{2}
\renewcommand{\LLm}{{r^\ast}}
Let us set $\beta=\frac 1 {\pst  }$, $1-\beta=\frac {\pst  -1 }{\pst  }=\frac{1}{p'}$.
 Observe also that
 	$$
 	\eta (s) \xi ^2(s) = \eta (s) \xi ^{2\beta } (s)\, \xi ^ {2(1-\beta)}(s),\quad s\in [0,T].
 	$$
Assume $\alpha<p$ and fix $\beta=\frac 1p$, $q\alpha/p=1$ and $q'=\frac {p}{p-\alpha}$
 	\begin{align*} 
 	\left\Vert \eta  \xi ^ {2}\right\Vert_{L^\alpha(0,T;L^1)}^\LLm
 	\le& \left( \int_0^ T \left( \int_\CO  \left\vert \eta ^ {p   } (s,x) \xi ^ {2\beta p }(s,x)\right\vert\, dx\right)^\frac \alpha {p} \left( \int_\CO  \left\vert  \xi ^ {2(1-\beta )p'}(s,x)\right\vert\, dx\right)^ \frac \alpha{p'} \, ds\right)^\frac \LLm \alpha
 	\\
 	\le
 	& \,\left( \int_0^ T  \int_\CO  \left\vert \eta ^ {\pst   }(s,x) \xi ^ { {\ns }}(s,x)\right\vert\, dx\, ds \right)^\frac \LLm{ q}
  \left(\int_0^ T  \left\vert  \xi (s)\right\vert_{L^ {(1-\beta)p'2}} ^  {\frac {\alpha q'}{p'}} \, ds\right)^\frac \LLm{q'}.
  	\\
 	\le
 	& \,\left\Vert \eta ^ {\pst   } \xi ^ { {\ns }}\right\Vert_{L^1(0,T;L^1)}^\frac \LLm{ q}
  \left\Vert\xi\right\Vert_{L^ {\frac {\alpha q'}{p'}} (0,T;L^ {(1-\beta)p'2}))}^{\LLm\frac \alpha {p'}}
    	\\
 	\le
 	& \,\left\Vert \eta ^ {\pst   } \xi ^ { {\ns }}\right\Vert_{L^1(0,T;L^1)}^\frac \LLm{ q}
  \left\Vert\xi\right\Vert_{L^ {\frac {\alpha (p-1)}{p-\alpha}} (0,T;L^ {2})}^{\LLm\frac \alpha {p'}}
.
	\end{align*} 
Here, we need $\alpha/p'< 1$ which gives $2\le p<\frac{\alpha}{\alpha-1}$.
Also, if
$$
\frac d2 -\rho<\frac {2({p-\alpha})} {\alpha (p-1)}+\frac d2,
$$
then  	\begin{align*} 
 	\left\Vert \eta  \xi ^ {2}\right\Vert_{L^\alpha(0,T;L^1)}^\LLm
 	\le&  \,\left\Vert \eta ^ {\pst   } \xi ^ { {\ns }}\right\Vert_{L^1(0,T;L^1)}^\frac \LLm{ q}
  \left\Vert\xi\right\Vert_{\BH}^{\LLm\frac \alpha {p'}}.
 \end{align*} 
 Taking the expectation and applying the H\"older inequality where we have to take into account that $\alpha/p'< 1$
 we get the assertion.
\end{proof}

\begin{proof}[Proof of Proposition \ref{propvarational2}:]
Let us consider the following equation for a locally integrable and progressively measurable $t\mapsto F(t)$,
\begin{align} \label{wwwoben}
 \qquad d w(t)=&r_v\Delta   w(t)\,dt+F(t)\, dt+\sigma_2 w(t)\, dW_2(t),\quad w(0)=w_0\in H^\rho_2(\CO).
\end{align} 
Following the proof of Proposition \ref{propvarational} verbatim, we find that
\begin{align} \label{zweites}
\left\vert\int_0^ t e^{r_v\Delta(t-s)}F(s)\,ds\right\vert_{H^\rho_2} \le C(T)\, \left(\int_0^ t  \left\vert F(s)\right\vert_{L^1}^\textkappa\, ds\right)^\frac 1 \textkappa.
\end{align} 
In particular, for $\alpha$ as in the Technical Lemma \ref{dasauch}, we have
$2-\frac 2\alpha>\frac d2+\rho$ and $\alpha<2$. This gives as condition for $\rho$, $\rho<1-\frac d2$. In addition, setting $p=2$ we
need $-\rho<\frac {2(2-\alpha)}\alpha$, which is not a restriction.
However, due the hypotheses we find some $\alpha<2$ such that the first inequality is satisfied.
 Note, that due to the condition in Hypothesis \ref{init}, there exists some $\mu=\alpha\ge 1$ such that $\rho$ satisfies the assumption above and those of the Technical Lemma \ref{dasauch}.
Setting
$$F=\phi_\kappa(h(\eta,\xi,\cdot))\eta\xi^2,$$
we obtain by applying Technical Lemma \ref{dasauch} that there exists some constant $C(T)>0$ and
$\delta_1,\delta_2\in(0,1)$, $\delta_1+\delta_2=1$ such that
\begin{align*} 
\EE \left(\int_0^ T \vert F(s)\vert_{L^1}^\mu\, ds\right)^{\frac {m_0}\mu}
\le& C(T)\Reins^{\delta_1} \Rzwei^{\delta_2}.
\end{align*} 
It remains to calculate the norm in $L^2(0,T;H^{\rho+1}_2(\CO))$.
By standard calculations (i.e. applying the smoothing property and the Young inequality for convolution), we get
for $\frac 32 \ge \frac 1\mu+\frac 1\kappa$ and $\delta \kappa/2<1$
\begin{align*} 
 \left\Vert\int_0^ \cdot e^{r_v\Delta(\cdot-s)}F(s)\,ds\right\Vert_{L^2(0,T;H_2^{\rho+1})}\le C\,  \Vert F\Vert_{L^\mu(0,T;H_2^{\rho+1-\delta})}
.
\end{align*} 
The embedding $L^1(\CO) \hookrightarrow H^{\rho+1-\delta}_2(\CO)$ for $\delta-(\rho+1)>\frac d2$,
gives
\begin{align*} 
 \Vert F\Vert_{L^\mu(0,T;H_2^{\rho-1})}\le  \Vert F\Vert_{L^\mu(0,T;L^1)},
 \end{align*} 
for  and similarly to before we know by hypothesis \ref{init}, that there exists some $\mu=\alpha$ such that $\delta-(\rho+1)>\frac d2$,
 $\frac 32 \ge \frac 1\mu+\frac 1\kappa$, $\delta \kappa/2<1$, and $\alpha(p-1)/(p-\alpha)\le m_0$. Therefore, we get
 by  the Technical Lemma \ref{dasauch}
\begin{align*} 
\EE \left\Vert\int_0^ \cdot e^{r_v\Delta(\cdot-s)}F(s)\,ds\right\Vert^{m_0}_{L^2(0,T;H_2^{\rho+1})}\le C(T)\Reins^{\delta_1} \Rzwei^{\delta_2}
.
\end{align*} 
\end{proof}

\section{Pathwise uniqueness of the solution}\label{pathwise}

\begin{proof}[Proof of Theorem \ref{thm_path_uniq}]

Let us remind, since $(u_1,v_1)$ and $(u_2,v_2)$ are solutions to the system
  \eqref{eq:uIto}--\eqref{eq:vIto}  with
$\PP(u_1(0)=u_2(0))=1$ and $\PP(v_1(0)=v_2(0))=1$, we can write for $i=1,2$,
\begin{align} 
d{u}_i(t)=& \Big( r_u \Delta u_i^{[\gamma]} (t)
-\chi  u_i(t) v^2_i(t)\big)\Big)\, dt+\sigma_1 u_i(t)dW_1(t)\,\quad t>0,\label{sysu12.11}
\\
d{v}_i(t) =& \big(r_v \Delta v_i(t)+ u_i(t) v^2_i(t)\Big)\, dt +\sigma_2 v_i(t) dW_2(t)\,\quad t>0.\label{sysv12.11}
\end{align}

In the first step we will introduce a family of stopping times $\{\tau_N\;\colon\;N\in\NN\}$,
and show that on the time interval $[0,\tau_N]$ the solutions $u_1$ and $u_2$, respective, $v_1$ and $v_2$, are indistinguishable. Here,
in the second step, we will show that $\PP\left( \tau_N<T\right)\to 0$  for $N\to\infty$. From this follows that $u_1$ and $u_2$ are indistinguishable on the time interval $[0,T]$.

By Remark \ref{rem:dimone}, $d=1$. Let us remind that $\delta_0\in(0,\frac 1\gamma)$, $\rho\in[0,\frac 12)$, $\frac 12<\frac{2}{m}+\frac{1}{m_0}+\rho\le \frac{3\gamma+2}{2\gamma+2}+\rho$.
\paragraph{Step I.}
Let us introduce the stopping times  $\{ \tau_N\;\colon\;N\in\NN\}$ as follows:
\begin{align*} 
\tau_{N,i}^1:=&\inf\{t \geq 0: \Vert\mathbbm{1}_{[0,t]}v_i\Vert_{\mathbb{H}_\rho} \geq N \} \wedge T,\quad i=1,2,
\\
\tau_{N,i}^2:=&\inf\{t \geq 0: \Vert\mathbbm{1}_{[0,t]}u_i\Vert_{L^{2\gamma}(0,T;H^{\delta_0}_2)}
 \geq N \} \wedge T,\quad i=1,2,
\end{align*} 
and $\tau_N:=\min_{i=1,2}(\tau_{N,i}^1,\tau_{N,i}^2)$.

The aim is to show
that $(u_1,v_1)$ and $(u_2,v_2)$ are indistinguishable on the time interval $[0,\tau_N]$.
\begin{remark}\label{muss_noch}
Note, due to the assumptions in Theorem \ref{thm_path_uniq} and by Corollary \ref{stroock}, and Proposition \ref{uniform1} we know that
there exists a constant $C>0 $ such that  for any solutions $(u,v)$  to  \eqref{eq:uIto}--\eqref{eq:vIto}
we have
$$\EE  \Vert v\Vert_{\mathbb{H}_\rho}^{m_0}\le C \quad\mbox{and} \quad \EE\Vert u\Vert_{L^{2\gamma}(0,T;H^{\delta_0}_2)}^{2\gamma}\le C.
$$
\end{remark}

Fix  $N\in\NN$.
To get uniqueness on $[0,\tau_N]$ we first stop the original solution processes at time $\tau_N$ and extend the processes $(u_1,v_1)$ and $(u_2,v_2)$
by other processes to the whole interval $[0,T]$.  For this purpose,
let $(y_1,z_1)$ be  solutions to
\begin{align} \label{eq11}
 \left\{\barray dy_1(t) =& r_u\Delta y_1^{[\gamma]}(t)+\sigma_1 y_1 (t)\,d\theta_{\tau_N}\circ W_1(t),\quad t\in[ \tau_N\wedge T ,T],
\\
 dz_1(t) =& r_v\Delta z_1(t)- z_1(t)+ \sigma_2 z_1 (t)\,d\theta_{\tau_N}\circ W_2(t),\quad t\in[ \tau_N\wedge T ,T],
\earray\right.  \end{align} 
with initial data $ y_1(0):=u_1(\tau_N)$,  $ z_1(0):=v_1(\tau_N)$, and
let $(y_2,z_2)$ be  solutions to
\begin{align} \label{eq22}
 \left\{\barray dy_2(t) =& r_u\Delta y_2^{[\gamma]}(t)+\sigma_1 y_2 (t)\,d\theta_{\tau_N}\circ W_1(t),\quad t\in[ \tau_N\wedge T ,T] ,
\\
 dz_2(t) =& r_v\Delta z_2(t)- z_2(t)+ \sigma_2 z_2 (t)\,d\theta_{\tau_N}\circ W_2(t),\quad t\in[ \tau_N\wedge T ,T],
\earray\right.
 \end{align} 
with initial data $ y_2(0):=u_2(\tau_N)$ and  $ z_2(0):=v_2(\tau_N)$. Here, $\theta$ denotes the shift operator, i.e.\ $\theta_\tau\circ W_i(t)=W_i(t+\tau)-W_i(\tau)$, $i=1,2$.
Since $(u_1,v_1)$ and $(u_2,v_2)$ are continuous in $H^{-1}_2(\CO)\times L^2(\CO)$, $(u_1(\tau_N),v_1(\tau_N))$ and $(u_2(\tau_N),v_1(\tau_N))$ are well-defined and
 belong $\PP$--a.s.\ to $H^{-1}_2(\CO)\times H^{-\delta_0}_2(\CO)$.
By \cite[Theorem 2.5.1]{BDPR2016}, we know that there exists a unique solutions $y_1$ and $y_2$ to  \eqref{eq11} belonging $\PP$-a.s.\ to $C([0,T];H^{-1}_2(\CO))$.  Since $(e^{t(r_v\Delta - \operatorname{Id})})_{t\ge 0}$ is an  analytic semigroup on  $H^\rho_2(\CO)$,
the existence of unique solutions  $z_1$ and $z_2$ to \eqref{eq11} and \eqref{eq22} belonging $\PP$-a.s.\ to $C([0,T];H^{-\delta_0}_2(\CO))$
can be shown by standard methods, cf. \cite{DaPrZa:2nd}.
Now, let us define two  processes $(\bar u_1,\bar v_1) $ and $(\bar u_2,\bar v_2)$ which are
equal to $(u_1,v_1) $ and $(u_2,v_2) $ on the time interval $[0,\tau_N)$ and
follow the processes  $(y_1,z_1)$ and $(y_2,z_2)$ afterwards.
In particular, let
$$
\bar u_1  (t) = \bcase u_1(t) & \mbox{ for } 0\le t< \tau_N,\\
y_1 (t) & \mbox{ for } \tau_N\le  t \le T,\ecase
\qquad
\bar v_1  (t) = \bcase v_1(t) & \mbox{ for } 0\le t< \tau_N,\\
z_1 (t) & \mbox{ for } \tau_N\le  t \le T,\ecase
$$
and
$$
\bar u _2  (t) = \bcase u_2 (t) & \mbox{ for } 0\le t< \tau_N,\\
y_2 (t) & \mbox{ for } \tau_N\le  t \le T,\ecase
\qquad
\bar v _2  (t) = \bcase v_2 (t) & \mbox{ for } 0\le t< \tau_N,\\
z_2 (t) & \mbox{ for } \tau_N\le  t \le T.\ecase
$$
Note, that  $( u_1, v_1) $ and $(u_2, v_2)$ solve on $[0,\tau_N)$
the
equation corresponding  to   \eqref{eq:uIto}--\eqref{eq:vIto}, that is, for $i=1,2$,
\begin{eqnarray} 
d{u}_i(t)&=& \Big( r_u \Delta u_i^\gamma (t)
-\chi  u_i(t) v^2_i(t)\Big)\, dt+\sigma_1 u_i(t)dW_1(t)\,\quad t>0,\label{sysu12.111}
\\
d{v}_i(t) &=& \big(r_v \Delta v_i(t)+ u_i(t) v^2_i(t)\Big)\, dt +\sigma_2 v_i(t) dW_2(t)\,\quad t>0.\label{sysv12.111}
\end{eqnarray}

\renewcommand{\tt}{{t}}
\newcommand{\bau}{\bar u}
\newcommand{\bav}{\bar v}

\paragraph{Step II.}
Let $\alpha=\rho+\frac 12$
(which implies $\alpha\ge \frac{1}{2}$ due to the assumption on $\rho$) and let  $\frac 12>\delta>\delta_0$.
Our goal is to show that $(u_1,v_1)$ and $(u_2,v_2)$ are indistinguishable on the interval $[0,\tau_N]$.
Applying the It\^o formula while setting
$$
\Phi(t):=\EE\left[\sup_{0\le s\le t}\vert\bau _1(s)-\bau _2(s)\vert^2_{H^{-1}_2}\right]
$$
over $[0,\tau_N]$ and taking into account that $(\bau _1,\bav_1)$ and $(\bau _2,\bav_2)$ are solutions to system  \eqref{eq:uIto}--\eqref{eq:vIto}
we obtain by standard calculations  for $0\le \tt\le \tau_N$
\begin{align}\label{here_tt_12}
&\EE \Big[\vert\bau _1({\tt })-\bau _2({\tt })\vert^2_{H^{-1}_2}\Big]
 +r_u \EE \Big[\int_0^{\tt } \vert\bau _1(s)-\bau _2(s)\vert^{\gamma+1}_{L^{\gamma+1}}ds
\Big]
\\
& \leq \chi\EE \Big[ \Big\vert
\int_0^{\tt } \int_{\CO} (-\nabla)^{-1}
\Big(\bau _1(s,x) \bav ^2_1(s,x)-\bau _2(s,x)\bav ^2_2(s,x)
\Big)
(-\nabla)^{-1}\big(\bau _1(s,x)-\bau _2(s,x)\big)dx\,ds
\Big\vert\Big]\notag
\\
&{}+ C\EE\int_0^ \tt \vert \bau _1({s })-\bau _2({s })\vert_{H^{-1}_2}^2\,ds.\notag
\end{align}
Furthermore, we obtain
\begin{align*}
&\EE \Big[ \Big\vert
\int_0^{\tt } \int_{\CO} (-\nabla)^{-1}
\Big(\bau_1(s,x) \bav ^2_1(s,x)-\bau _2(s,x) \bav ^2_2(s,x)
\Big)
(-\nabla)^{-1}\big(\bau _1(s,x)-\bau _2(s,x)\big)dx\,ds \Big\vert\Big]\notag\\
& \leq \EE \Big[
{\Big\vert \int_0^{\tt } \int_{\CO} (-\nabla)^{-1}
\big(\bau _1(s,x)-\bau _2(s,x)\big) \bav ^2_1(s,x)
(-\nabla)^{-1}\big(\bau _1(s,x)-\bau _2(s,x)\big)dx\,ds\Big\vert}\Big]
\notag\\
&\quad + \EE\Big[ {\Big\vert
\int_0^{\tt } \int_{\CO} (-\nabla)^{-1}\Big[ \bau _2(s,x) \bav _2(s,x)
 \big(\bav _1(s,x)-\bav _2(s,x)\big)\Big] (-\nabla)^{-1}\big(\bau _1(s,x)-\bau _2(s,x)\big)dx\,ds\Big\vert} \Big]\notag
\\
& =: I_1(t)+ I_2(t).\notag\end{align*}
Now we have by Proposition \ref{prop:runst190} with $\alpha\ge \frac{1}{2}=\frac{d}{2}$,
\begin{align*} 
\vert(\bar u _1-\bar u_2) \bav _1^2\vert_{H^{-1}_2}\le \vert\bau_1-\bau _2\vert_{H^{-1}_2}\vert \bav_1^2\vert_{H^{\alpha}_2}
\end{align*} 
Since $H^\alpha_2(\CO)$ is a Banach algebra (see \cite[Theorem 2-(18), p. 192]{runst}), we have
\begin{align*} 
\vert(\bau  _1-\bau _2) \bav_1^2\vert_{H^{-1}_2}\le \vert\bau _1-\bau _2\vert_{H^{-1}_2}\vert \bav_1\vert^2_{H^{\alpha}_2}.
\end{align*} 
Hence, we get by applying Young's inequality for $\ep>0$
\begin{align*} 
 I_1(t)
  \le&\EE\int_0^ \tt \vert\bau _1(s)-\bau _2(s)\vert_{H^{-1}_2}^2 \vert \bav_1(s)\vert^2_{H^{\alpha}_2}\, ds
\\
  \le&
\ep  \EE\int_0^ \tt \vert\bau _1(s)-\bau _2(s)\vert_{H^{-1}_2}^2 \vert \bav _1(s)\vert^4_{H^{\alpha}_2}\, ds
+C(\ep)  \EE\int_0^ \tt \vert\bau _1(s)-\bau _2(s)\vert_{H^{-1}_2}^2\, ds
\\
  \le&
\ep \EE\left[\sup_{ 0\le s\le \tt} \vert\bau _1(s)-\bau _2(s)\vert_{H^{-1}_2}^2\right]\times\EE\int_0^ \tt\vert \bav _1(s)\vert^4_{H^{\alpha}_2}\, ds
\\
&\qquad {}+C(\ep)  \EE\int_0^ \tt \vert\bau _1(s)-\bau _2(s)\vert_{H^{-1}_2}^2\, ds.
\end{align*} 
Since $\alpha= \rho+\frac 12$, we have $\Vert\bav_1\Vert_{L^4(0,T;H^\alpha_2)}\le \Vert\bav_1\Vert_{\mathbb{H}_\rho}$.
Next, note that $\vert\bau_2(\bav_1^2-\bav^2_2)\vert\le \vert\bau_2\bav_1(\bav_1-\bav_2)\vert+\vert\bau_2\bav_2(\bav_1-\bav_2)\vert$. Similarly as above we get for $\delta_1<\delta$ and $\delta_2=\delta<\frac 1 \gamma$ and $\delta_3>\frac 12=\frac{d}{2}$ (see Proposition \ref{prop:runst190},
\cite[Theorem 2, p. 200]{runst} and the identities  on  \cite[Proposition, p. 14]{runst}) for $i=1,2$
\begin{align} \label{uselater}\hspace{+6ex}
\vert\bau_2\bav_i(\bav_1-\bav_2)\vert_{H^{-1}_2}\le C_1\vert\bav_1-\bav_2\vert_{H^{-\delta_1}_2}\vert\bau_2\bav_i\vert_{H^{\delta}_2}\le
C_2\vert\bav_1-\bav_2\vert_{H^{-\delta_1}_2}\vert\bau_2\vert_{H^{\delta_2}_2}\vert\bav_i\vert_{H^{\delta_3}_2}.
\end{align} 
By  Young inequality,
 we know that for $\delta_3\le \frac {\gamma}{\gamma+1}+\rho$ and $-\delta_0>-\frac 1 \gamma$ and  for any $\ep>0$
 there exists a constant $C(\ep)>0$ such that
\begin{align*} 
I_2(t) \le& C(\ep)\EE\int_0^ t  \vert\bav _1(s)-\bav _2(s)\vert^2_{H^{-\delta_0}_2}\, ds +
\ep\EE\left(\int_0^ t \vert\bau _1(s)\vert_{H^{\delta_2}_2}^{2\gamma}\, ds \right)^\frac 1 {\gamma}\,
\\
&{}\times\EE \left(\int_0^ t \Big(\vert\bav _1(s)\vert^{2\frac {\gamma+1}\gamma}_{H^{\delta_3}_2}+\vert\bav _2(s)\vert^{2\frac {\gamma+1}\gamma}_{H^{\delta_3}_2}\big)\, \, ds\right)^\frac \gamma {\gamma+1}\EE\left[\sup_{0\le s\le t} \vert\bav _1(s)-\bav _2(s)\vert^2_{H^{-\delta_0}_2}\right]
\\
\le& C(\ep)\EE\int_0^ t  \vert\bav _1(s)-\bav _2(s)\vert^2_{H^{-\delta_0}_2}\, ds
\\&{}+\EE\Vert\bau _1\Vert_{L^{2\gamma}(0,T;H^{\delta_2}_2)} \left( \EE\Vert\bav _1\Vert_{\mathbb{H}_\rho} + \EE\Vert\bav _2\Vert_{\mathbb{H}_\rho} \right)\EE\left[\sup_{0\le s\le t} \vert\bav _1(s)-\bav _2(s)\vert^2_{H^{-\delta_0}_2}\right].
\end{align*} 
The third term in \eqref{here_tt_12} can be handled by the Burkholder-Davis-Gundy inequality.
Since $t\le \tau_N$ and $\delta_2<\frac 1\gamma$ we know that for any $\ep>0$ there exists a constant $C(\ep)>0$ such that
\begin{align*} 
\Phi(t) \le&\ep_1 N^2 \EE\left[\sup_{0\le s\le t}  \vert\bav _1(s)-\bav _2(s)\vert^2_{H^{-\delta_0}_2}\right]+\EE\int_0^ t
 \sup_{0\le r\le s}  \vert\bav _1(r)-\bav _2(r)\vert^2_{H^{-\delta_0}_2}\, ds
\\&+\ep N \EE\left[\sup_{ 0\le s\le \tt} \vert\bau _1(s)-\bau _2(s)\vert_{H^{-1}_2}^2\right]
+C(\ep)  \EE\int_0^ \tt \vert\bau _1(s)-\bau _2(s)\vert_{H^{-1}_2}^2\, ds
.\end{align*}

\paragraph{Step III.}
Let us set
\begin{align*} 
\Psi(t):=& \EE\left[\sup_{0\le s\le t} \Vert \bav _1(s)-\bav _2(s)\Vert_{H^{-\delta_0}_2}^{2}\right]+\EE\Vert \bav _1-\bav _2\Vert_{L^{2}(0,t;H^{-\delta_0+1}_2)}^{2}
.
\end{align*} 
Then we get by standard calculations for $\tilde \delta<1$
\begin{align*} 
&\EE\left[\sup_{0\le s\le t} \Vert \bav _1(s)-\bav _2(s)\Vert_{H^{-\delta_0}_2}^{2}\right]
+\EE\Vert \bav _1-\bav _2\Vert_{L^{2}(0,t;H_2^{-\delta_0+1})}^{2}\\
\le& C\EE\Vert\bau _1\bav _1^2-\bau _2\bav _2^2\Vert_{L^{2}(0,t;H_2^{-\delta_0-\tilde \delta})}^{2}
\\ &+ \EE\left[\sup_{0\le s\le t} \left\vert \int_0^ \tt \sigma_2 \left( \bav _1({s })-\bav _2({s })\right)dW_2(s)\right\vert^2_{H^{-\delta_0}_2}\right]
.\end{align*} 
Next, since $\tilde \delta$ can be chosen such that $\delta_0\ge 1-\tilde \delta$ we can use  \eqref{uselater} and obtain for $\delta_1<\delta$ and $\delta_2=\delta<\frac 1 \gamma$ and $\delta_3>\frac 12=\frac d2$
\begin{align*} 
 \vert\bau _1\bav  _1^2-\bau _1\bav _2^2\vert_{H^{-\delta_0-\tilde \delta}_2}
\le& C_
1\vert\bav _1-\bav _2\vert_{H^{-\delta_1}_2}\vert\bau _2\vert_{H^{\delta_2}_2}\left(\vert\bav _1\vert_{H^{\delta_3}_2}+\vert\bav _2\vert_{H^{\delta_3}_2}\right)
.
\end{align*} 
For $\alpha >\min(1+-\delta_0-\delta+\frac d2,1,\frac d2)$ and $\tilde \delta<1$ we have by
Proposition \ref{prop:runst190} and  \cite[p. 192, Theorem 2-(18)]{runst}
\begin{align*} 
\vert(\bau _1-\bau _2)\bav _1^2\vert_{H^{-\delta_0-\tilde \delta}_2}\le \vert\bau _1-\bau _2\vert_{H^{-1}_2}\vert\bav _1^2\vert_{H^{\alpha}_2}\le \vert\bau _1-\bau _2\vert_{H^{-1}_2}\vert\bav _1\vert^2_{H^{\alpha}_2}.
\end{align*} 
The third term can be handled by an application of the Burkholder-Davis-Gundy inequality.
The Young inequality implies that for any $\ep_1,\ep_2>0$ there exist constants $C(\ep_1),C(\ep_2)>0$ such that we have
\begin{align*} 
\lqq{\Psi(t)
\le  \EE\vert\bav _1(0)-\bav _2(0)\vert^2_{H^{-\delta_0}_2}}
&\\
&{}+
 \EE\left(\int_0^ t \vert\bau _1(s)\vert^2_{H^{\delta_2}_2}\,\left( \vert\bav _1(s)\vert^2_{H^{\delta_3}_2} +\vert\bav _2(s)\vert^2_{H^{\delta_3}_2}\right)\vert\bav _1(s)-\bav _2(s)\vert^2_{H^{-\delta_0}_2}\, ds\right)
\\
&{}+
 \EE\left(\int_0^ t  \vert\bau _1(s)-\bau _2(s)\vert^{2}_{H^{-1}_2}\, \vert\bav _1(s)\vert^{4}_{H^{\alpha}_2}\, ds\right)
\\
\le& \EE\vert v_1(0)-v _2(0)\vert^2_{H^{-\delta_0}_2}\\
& +
 \ep_1
 \EE\left(\int_0^ t \vert\bau _1(s)\vert^2_{H^{\delta_2}_2} ,\left( \vert\bav _1(s)\vert^2_{H^{\delta_3}_2} +\vert\bav _2(s)\vert^2_{H^{\delta_3}_2}\right) \vert\bav _1(s)-\bav _2(s)\vert^2_{H^{-\delta_0}_2}\, ds\right)
\\&{}+
C(\ep_1)\EE \left(\int_0^ t\vert\bav _1(s)-\bav _2(s)\vert^2_{L^2}\, ds\right)
\\
&{}+
\ep_2 \EE\left(\int_0^ t  \vert\bau _1(s)-\bau _2(s)\vert^2_{H^{-1}_2}\, \vert\bav _1(s)\vert^{4}_{H^{\alpha}_2}\, ds\right)
+  C(\ep_2)\EE\left(\int_0^ t  \vert\bau _1(s)-\bau _2(s)\vert^2_{H^{-1}_2}\,\, ds\right).
\end{align*} 
\begin{align*} 
I_2(t) \le&
\EE\left(\int_0^ t \vert\bau _1(s)\vert_{H^{\delta_2}_2}^{2\gamma}\, ds \right)^\frac 1 {\gamma}\,
\\
&{}\times \EE\left(\int_0^ t \Big(\vert\bav _1(s)\vert^{2\frac {\gamma+1}\gamma}_{H^{\delta_3}_2}+\vert\bav _2(s)\vert^{2\frac {\gamma+1}\gamma}_{H^{\delta_3}_2}\big)\, \, ds\right)^\frac \gamma {\gamma+1}\EE\left[\sup_{0\le s\le t} \vert\bav _1(s)-\bav _2(s)\vert^2_{H^{-\delta_0}_2}\right]
\\
\le&\EE\Vert\bau _1\Vert_{L^{2\gamma}(0,T;H^{\delta_2}_2)} \left( \EE\Vert\bav _1\Vert_{\mathbb{H}_\rho} +\EE \Vert\bav _2\Vert_{\mathbb{H}_\rho} \right)\EE\left[\sup_{0\le s\le t} \vert\bav _1(s)-\bav _2(s)\vert^2_{H^{-\delta_0}_2}\right].
\end{align*} 
Since $\delta_0\in(0,1)$ and $\alpha=\rho+\frac 12$,
\begin{align*} 
\Psi(t)
\le& \EE\vert\bav  _1(0)-\bav _2(0)\vert^2_{H^{-\delta_0}_2}+
\ep_1\EE \Vert\mathbbm{1}_{[0,t]}\bau _1\Vert_{L^{2\gamma}(0,T;H^{\delta_2}_2)}  \,  \left(\EE \Vert\bav _1\Vert_{\mathbb{H}_\rho} +\EE \Vert\bav _2\Vert_{\mathbb{H}_\rho} \right) \\
&\qquad\times\EE \left[\sup_{0\le s\le t} \vert\bav _1(s)-\bav _2(s)\vert^2_{H^{-\delta_0}_2}\right]
\\&{}+
C(\ep_1) \EE\left(\int_0^ t\vert\bav _1(s)-\bav _2(s)\vert^2_{H^{-\delta_0}_2}\, ds\right)
\\
&{}+
\ep_2 \EE\left[\sup_{0\le s\le t}\vert\bau _1(s)-\bau _2(s)\vert^2_{H^{-1}_2}\right]\EE \Vert\mathbbm{1}_{[0,t]} \bav _1\Vert_{L^4(0,T;H^{\alpha}_2)}^2
\\&+ C(\ep_2)\EE\int_0^ t  \vert\bau _1(s)-\bau _2(s)\vert^2_{H^{-1}_2}\,\, ds
\\
\le& \EE\vert v _1(0)-v_2(0)\vert^2_{H^{-\delta_0}_2}+
\ep_1\EE \Vert\mathbbm{1}_{[0,t]}\bau _1\Vert_{L^{2\gamma}(0,T;H^{\delta_2}_2)}  \,  \left(\EE \Vert\bav _1\Vert_{\mathbb{H}_\rho} + \EE\Vert\bav _2\Vert_{\mathbb{H}_\rho} \right) \\
&\qquad\times\EE \left[\sup_{0\le s\le t} \vert\bav _1(s)-\bav _2(s)\vert^2_{H^{-\delta_0}_2}\right]
\\&{}+
C(\ep_1) \EE\left(\int_0^ t\vert\bav _1(s)-\bav _2(s)\vert^2_{H^{-\delta_0}_2}\, ds\right)
\\
&{}+
\ep_2 \EE\left[\sup_{0\le s\le t}\vert\bau _1(s)-\bau _2(s)\vert^2_{H^{-1}_2}\right]\EE \Vert\mathbbm{1}_{[0,t]}\bav _1\Vert_{\mathbb{H}_{\rho}}^2
\\&+ C(\ep_2)\EE\int_0^ t\sup_{0\le r\le s}  \vert\bau _1(r)-\bau _2(r)\vert^2_{H^{-1}_2}\,\, ds
.
\end{align*} 
Taking into account the definition of $\tau_N$ we obtain
\begin{align} \label{oben11}
\lqq{\Psi(t)\le  \EE \vert\bav  _1(0)- \bav _2(0)\vert^2_{H^{-\delta_0}_2}}\nonumber
&
\\
&{}+\ep_1 N^2 \Psi(t)+\ep_2N\Phi(t)+ C(\ep_1) \int_0^ t \Phi(s)\, ds+C(\ep_2)\int_0^ t \Psi(s)\, ds.
\end{align} 

\paragraph{Step IV.}
Noting that the estimates can be extended to the whole interval $[0,T]$ by standard calculations.
Next, collecting altogether  we know, that for any $\ep_1,\ep_2,\ep_3,\ep_4>0$ there exist constants $C(\ep_1),C(\ep_2),C(\ep_3),C(\ep_3)>0$ such that
\begin{align*} 
\lqq{\Phi(t)+\Psi(t) \le  \EE \vert\bau_1(0)-\bau_2(0)\vert^2_{H^{-1}_2}+ \EE\vert\bav  _1(0)-\bav  _2(0)\vert^2_{H^{-\delta_0}_2}}
&
\\
&{}+
\ep_1N^2\Psi(t) + C(\ep_1)\int_0^t \Psi(s)\, ds
+  \ep_0 N \Phi(t)
+C(\ep_0)\int_0^ t  \Phi(s)\, ds
\\
&{}+\ep_3 N^2 \Psi(t)+\ep_4N\Phi(t)+ C(\ep_1) \int_0^ t \Phi(s)\, ds+C(\ep_2)\int_0^ t \Psi(s)\, ds.
\end{align*} 
Taking the  $\ep_1,\ep_2,\ep_3,\ep_4>0$ accordingly, we know, that for any $N\in\NN$  there exists a constant $C(N)>0$ such that
\begin{align*} 
&\Phi(t)+\Psi(t) \le   \EE\vert\bau_1(0)-\bau_2(0)\vert^2_{H^{-1}_2}+ \EE\vert\bav  _1(0)-\bav _2(0)\vert^2_{H^{-\delta_0}_2}\\
+& C(N)\int_0^ t \left( \Phi(s)+\Psi(s)\right)\, ds.
\end{align*} 
An application of the Gronwall lemma and taking into account that $v_1(0)=v_2(0)$, $u_1(0)=u_2(0)$ give that
$\Phi(t)\le 0$ and $\Psi(t)\le 0$.

\paragraph{Step V.}
We show that $\PP\left( \tau_N<T\right) \longrightarrow 0$ as $N\to\infty$.
Therefore,
\begin{align*} 
\lqq{\PP\left( \tau_N <  T \right) }
&\\\le&
\PP\left( \Vert\bar u_1\Vert_{L^{2\gamma}(0,T;H^{\delta_0}_2)}\ge N \right) + \PP\left(  \Vert\bar u_2\Vert_{L^{2\gamma}(0,T;H^{\delta_0}_2)}\ge N \right)
\\
&\qquad{}+  \PP\left( \Vert \bar v_1\Vert_{\mathbb{H}_\rho}\ge N \right)+ \PP\left( \Vert\bar v_2\Vert_{\mathbb{H}_\rho}\ge N \right)
.\end{align*} 
Due to Remark \ref{muss_noch}, we can apply the Chebyshev inequality and
get by the above that
\begin{align*} 
\PP\left( \tau_N <  T \right) \le\frac C{N^2},
\end{align*} 
It follows that
$$\PP\left( \tau_N \le   T \right)\longrightarrow 0,$$
as $N\to\infty$.
Hence,  both the processes $\bau_1$ and $\bau_2$, and likewise $\bav_1$ and $\bav_2$ are indistinguishable on $[0,T]$.
Since the processes $(\bau_1,\bav_1)$ and $(\bau_2,\bav_2)$ solves on $[0,T\wedge \tau_N]$ the system
 corresponding  to   \eqref{eq:uIto}--\eqref{eq:vIto},
the last arguments  completes the proof of Theorem \ref{thm_path_uniq}.
\end{proof}

\appendix

\section{Some useful inequalities}\label{app:A}
\begin{lemma}\label{lem:porous-medium-inequality}

  Let $\gamma>1$, $x,y\in\mathbb{R}$. Then it holds that
\begin{equation}\label{eq:porous-medium-inequality}
(x^{[\gamma]}-y^{[\gamma]})(x-y)\ge2^{1-\gamma}\vert x-y\vert^{\gamma+1},\end{equation}
where $z^{[\gamma]}:=\vert z\vert^{\gamma-1}z$ for $z\in\mathbb{R}$.
\end{lemma}

\begin{proof}
See e.g. \cite[Lemma 3.1]{Liu:2009ih}.
\end{proof}

Fix a bounded domain $\Ocal\subset\mathbb{R}^d$ with sufficiently smooth boundary.

\begin{proposition}\label{interpolation_11}
For any for $r\in(2,q+1)$, $m\in(q+1,\infty)$, $s\in(0,1)$, with  $\frac 1r\ge\frac 1{m}-\frac s2$, and $\frac 1m\ge\frac {1+s}{q+1}$, there exists a constant $C>0$ such that
\begin{align} \label{ineq001_1}
\Vert\xi\Vert_{L^m(0,T;H^s_r)}^r\le C\left( \Vert\xi\Vert^2_{L^\infty(0,T;H_2^{-1})}+\Vert\xi\Vert^{q+1}_{L^{q+1}(0,T;L^{q+1})}\right).
\end{align} 
\end{proposition}

\begin{proof}
In order that
$L^m(0,T;H^s_r(\CO))$ is an
interpolation space between $$
L^\infty(0,T;H_2^{-1}(\CO))\quad\text{and}\quad L^{q+1}(0,T;L^{q+1}(\CO)),
$$
we need that there exists $\theta\in(0,1)$ such that for the parameters $m,r,s$ the following inequalities are satisfied, see e.g. \cite{bergh},
\begin{align*}
s&\le-\theta,
\\
\frac 1m &\ge\frac{1-\theta}{q+1},
\\
\frac 1r &\ge\frac \theta 2 + \frac{1-\theta}{q+1}.
\end{align*}
Now, if  $\frac 1r\ge\frac 1m-\frac s2$ and $\frac 1m\ge\frac {1+s}{q+1}$ are satisfied for  $r\in(2,q+1)$, $m\in(q+1,\infty)$, $s\in(0,1)$,
then the set of inequalities are satisfied and
we obtain \eqref{ineq001_1}.
\end{proof}

\begin{proposition}\label{prop:runst190}
Let $s_1,s_2\in\mathbb{R}$ and $p>1$. Let $\Ocal\subset\mathbb{R}^d$.
Assume that $s_1\le s_2$ and that $s_1+s_2>d\left(0\vee \left(\frac{2}{p}-1\right)\right)$ and $s_2<\frac{d}{p}$. Then there exists a constant $C>0$ such that
\[\vert uv\vert_{H^r_p}\le C\vert u\vert_{H_p^{s_1}}\vert v\vert_{H_p^{s_2}},\]
for any $r\le s_1+s_2-\frac{d}{p}$ and for any $u\in H_p^{s_1}(\Ocal)$ and any $v\in H_p^{s_2}(\Ocal)$.
\end{proposition}
\begin{proof}
See \cite[p. 190, Theorem 1 (iii)]{runst}.
\end{proof}

For the definition of the space $F^{s}_{p,q}$, we refer to \cite{sickel1,triebel,runst}. It translates to classical function spaces
as in e.g. \cite[Remark 2.1.1]{sickel1}, in particular, $F_{p,2}^{0}=L^p$, $1<p<\infty$ (Lebesgue spaces), $F_{p,2}^{m}=W^m_p$, $m\in\N$ (Sobolev spaces), $1<p<\infty$ and $F_{p,2}^{s}=H^s_p$, $s\in\mathbb{R}$, $1<p<\infty$ (fractional Sobolev spaces).

Next, we shall record a variant of the Stroock-Varopoulos inequality together with its proof. See e.g. \cite[Lemma 3.6]{DGG2020} for another version of this result.

\begin{proposition}\label{runst1}
For any bounded domain $\Ocal\subset\mathbb{R}^d$ with sufficiently smooth boundary, for any $T>0$, $\gamma>1$, $\theta\in(0,\frac{1}{\gamma})$, there exists a constant $C=C(\gamma,T,\Ocal)>0$ such that
\begin{align*} 
\Vert\eta\Vert_{L^{2\gamma}(0,T;H^{\theta}_{2\gamma})}^{2\gamma} \le C\left(\int_0^T \vert\eta(s)^{[\gamma-1]}\nabla \eta(s)\vert_{L^2}^2\,ds+\int_0^T\vert\eta(s)\vert^{2\gamma}_{L^2}\,ds\right),
\end{align*} 
where $z^{[\gamma-1]}:=\vert z\vert^{\gamma-2}z$ for $z\in\mathbb{R}$.
\end{proposition}

\begin{proof}
By \cite[p.~365]{runst}, we have for any  $p\in(1,\infty)$, $s\in(0,1)$, $\mu\in(0,1)$ and $\ep\in(0,s\mu)$,
	\begin{equation}\label{from_runst}
\vert \vert w\vert^\mu \vert_{H^{s\mu-\ep}_{\frac{p}{\mu}}}  =	\vert \vert w\vert^\mu\vert_{F^{s\mu-\ep}_{\frac p\mu,2}} \le C	\vert \vert w\vert^\mu \vert_{F_{\frac p\mu,\frac 2\mu}^{s\mu}}\le C	 \vert w\vert_{ F_{p, 2 }^s} ^\mu = C\vert w\vert_{H^{s}_p}^\mu, \quad w\in H^s_p(\CO).
	\end{equation}
		From \eqref{from_runst} we know that  for any  $\gamma>1$, $p\in (1,\infty)$, $\theta\in(0,\frac 1\gamma)$, $\ep>0$, there exists a constant $C>0$ such that
		$$
	 \vert w\vert_{H^\theta_{p\gamma}}=\vert\vert w^{\gamma} \vert^\frac 1\gamma \vert_{H^\theta_{p\gamma}}\leq C \vert\vert w\vert^\gamma\vert^{\frac 1\gamma}_{H^{\theta \gamma+\ep}_{p}}.$$
In particular, for any  $0<\theta<\frac 1 \gamma$ and $p=2$, there exists a constant $C>0$ such that
		$$	
	 \vert w\vert_{H^\theta_{2\gamma}}^\gamma \le C\vert\vert w\vert^\gamma\vert_{H^1_2}.
		$$
Since we know by the chain rule,
$$\int_0^ T \vert\nabla \vert\eta(s)\vert ^{\gamma}\vert_{L^2}^2\, ds=\gamma^2\int_0^ T \vert\eta^{[\gamma-1]}(s)\nabla \eta(s) \vert_{L^2}^2\, ds
$$
we know that for any $\theta<\frac 1\gamma$,
$$\int_0^ T \vert\eta(s)\vert_{H^\theta_{2\gamma}}^{2\gamma}\, ds\le C \left(\int_0^ T \vert\eta ^{\gamma-1}(s)\nabla \eta(s) \vert_{L^2}^2\, ds+\int_0^T \vert\eta(s)\vert^{2\gamma}_{L^2}\,ds\right).$$
\end{proof}

Clearly, we have by the fractional Rellich-Kondrachov theorem that for any $\ep\in (0,\theta)$,
$$\vert w\vert_{H^{\theta-\ep}_{2}}^\gamma\le C\vert w\vert_{H^\theta_{2\gamma}}^\gamma.$$
	As $\theta\in (0,\gamma^{-1})$ is arbitrary, this yields the following result.

\begin{corollary}\label{stroock}
For any bounded domain $\Ocal\subset\mathbb{R}^d$ with sufficiently smooth boundary, for any $T>0$, $\gamma>1$, $\theta\in(0,\frac{1}{\gamma})$, there exists a constant $C=C(\gamma,T,\Ocal)>0$ such that
\begin{align*} 
\Vert\eta\Vert_{L^{2\gamma}(0,T;H^{\theta}_{2})}^{2\gamma} \le C\left(\int_0^T \vert\eta(s)^{[\gamma-1]}\nabla \eta(s)\vert_{L^2}^2\, ds+\int_0^T\vert\eta(s)\vert^{2\gamma}_{L^2}\,ds\right),
\end{align*} 
where $z^{[\gamma-1]}:=\vert z\vert^{\gamma-2}z$ for $z\in\mathbb{R}$.
\end{corollary}

We have the following embedding.
\begin{proposition}\label{interp_rho}
Let $l_1,l_2\in(2,\infty)$ with $\frac{d}{2}-\rho\le \frac 2{l_1}+\frac d{l_2}$.
Then there exists $C>0$ such that
$$
\Vert\xi\Vert_{L^{l_1}(0,T;L^{l_2})}\le C\Vert\xi\Vert_{\BH},\quad \xi\in\BH,
$$
where $\BH$ is defined in \eqref{eq:BZeq}.
\end{proposition}
\begin{proof}
First, let us note that for $\delta>0$ such that
\begin{equation}\label{eq:deltaembedding}\frac {1}{l_2}\ge \frac{1}{2}-\frac{\delta}{d},
\end{equation}
we have the embedding $H^{\delta}_2(\CO)\hookrightarrow L^{l_2}(\CO)$,
and therefore
$$
\Vert\xi\Vert_{L^{l_1}(0,T;L^{l_2})}\le C\Vert\xi\Vert_{L^{l_1}(0,T;H^{\delta}_2)},\quad \xi\in\X.
$$
Due to interpolation, compare with \cite[Theorem 5.1.2, p.\ 107 and Theorem 6.4.5, p.\ 152]{bergh}, we have
$$
 \Vert\xi\Vert_{L^{l_1}(0,T;H^{\delta}_2)}\le  C\Vert\xi\Vert_{L^\infty(0,T;H^{\rho}_2)}^\theta \Vert\xi\Vert_{L^2(0,T;H^{\rho+1}_2)}^{1-\theta},\quad \xi\in\BH,
$$
for $\theta\in (0,1)$ with,
\begin{equation}\label{eq:delta_theta}\frac  1{l_1}\le \frac{1}{2}(1-\theta),\quad \delta\le\theta\rho+(1-\theta)(\rho+1).
\end{equation}
Taking into account that we have
\[\frac{d}{2}-\rho\le \frac 2{l_1}+\frac d{l_2},\]
gives that $\delta$ and $\theta$ satisfying \eqref{eq:deltaembedding} and \eqref{eq:delta_theta} exist and thus the Young inequality for products yields the assertion.
\end{proof}

\section{Function spaces and the Aubin-Lions-Simon compactness theorem}\label{dbouley-space}

Let $B$ be a separable Banach space, $0\le c<d<\infty$.
Let $C^{(\beta)}_b(c,d;B)$ denote a set of all continuous and bounded functions $u:[c,d]\to B$ such that
$$
\Vert u\Vert_{C_b^{\beta}(c,d;B)} :=\sup_{c\le t\le d} \vert u(t)\vert_B +\sup_{c\le s,t\le d\atop t\not= s} \frac{ \vert u(t)-u(s)\vert_{B}}{\vert t-s\vert^\beta},
$$
is finite. The space $C_b^{(\beta)}(c,d;E)$  endowed with the norm $\Vert \cdot\Vert_{C_b^{\beta}(c,d;B)}$ is a Banach space.
Let
 $$ L^p(c,d; ;B)=\left\{ u:[c,d)\to B: \mbox{ $u$ measurable and  }\int_{[c,d)}  \vert u(t)\vert_{B} ^ p\, dt <\infty\right\}.
$$
In addition, for $1< p < \infty$  let $W^1_p(\CO)$ be the standard Sobolev space defined by (compare \cite[p.\ 263]{Brezis})
\begin{align*} 
\lqq{W^{1}_p(\CO)
:=\left\{ u\in L^p(\CO)\mid \exists g_1,\cdots,g_d\in L^p(\CO)\mbox{ such that }\phantom{\bigg\vert}\right.}
\\
&{}\left. \int_\CO u(x)\frac{\partial \phi(x)}{\partial x_i} \, dx =-\int_\CO g_i(x)\phi(x)\, dx \quad \forall \phi\in C^\infty_0(\Ocal), \forall i=1,\ldots ,d \right\}
\end{align*} 
equipped with norm
$$
\vert u\vert_{W^1_p}:=\vert u\vert_{L^p}+\sum_{j=1}^d\left\vert\frac{\partial u}{\partial x_j}\right\vert_{L^p},\quad u\in W^1_p(\CO).
$$
Given an integer $m\ge2$ and a real number $1\le p<\infty$, we define by induction the space
\begin{align*} 
W^{m}_p(\CO) :=\left\{ u\in W^{m-1}_p(\CO)\mid D u\in W^{m-1}_p(\CO) \right\}
\end{align*} 
equipped with norm
$$
\vert u\vert_{W^m_p}:=\vert u\vert_{L^p}+\sum_{\alpha =1}^m\left\vert D_\alpha u\right\vert_{L^p},\quad u\in W^m_p(\CO).
$$
Let $H_2^m(\CO):=W^m_2(\CO)$, and for $\rho\in(0,1)$ let $H_2^\rho (\CO)$ be the real interpolation  space given by $H^\rho _2(\CO):=(L^2(\CO),H^1_2(\CO))_{\rho ,2}$.
In addition, let $H^{-1}_2(\CO)$ be the dual space of $H^1_2(\CO)$ and for $\rho \in(0,1)$ let $H^{-\rho }_2(\CO)$ be the  real interpolation  space given by $H^{-\rho }_2(\CO):=(L^2(\CO),H^{-1}_2(\CO))_{1-\rho ,2}$.
Note, by Theorem 3.7.1 \cite{bergh}, $H^{-\rho }_2(\CO)$ is dual to $H^\rho _2(\CO)$, $\rho \in(0,1)$. Furthermore, we have  $(H^{-\rho }_2(\CO),H^{\rho }_2(\CO))_{\frac 12,2}=L^2(\CO)$ and $(H^\alpha_2(\CO),H^\beta_2(\CO))_{\rho,2}=H^\theta_2(\CO)$ for $\theta=\alpha(1-\rho)+\beta\rho$, $\rho\in(0,1)$ and $\vert\alpha\vert,\vert\beta\vert\le 1$.

Since we need it to tackle the compactness, let us introduce the following space.
Given $p\in (1,\infty)$, $\alpha\in(0,1)$, let $\WW ^ {\alpha}_p (I;B)$ be the Sobolev space
of all $u\in L^p(0,\infty;B)$ such that
$$
\int_I  \int_{I\cap [t,t+1]}   \;\frac{\vert u(t)-u(s)\vert_B ^ p}{ \vert t-s\vert ^ {1+\alpha p}}\,ds\,dt<\infty;
$$
equipped with the norm
$$
\left\Vert u\right\Vert_{ \WW^ {\alpha}_p(I ;B)}:=\left( \int_I  \int_{I\cap [t,t+1]}  \;\frac{\vert u(t)-u(s)\vert_B ^ p}{ \vert t-s\vert ^ {1+\alpha p}}\,ds\,dt\right) ^ \frac 1p.
$$

\begin{theorem}\label{th-gutman}
Let $B_0\subset B\subset B_1$ be Banach spaces, $B_0$ and $B_1$ reflexive, with compact embedding of $B_0$ to $B$. Let $p\in(1,\infty)$ and $\alpha\in(0,1)$ be given. Let $X$ be the space
$$
X=L^p(0,T;B_0)\cap \WW ^{\alpha}_p (0,T;B_1).
$$
Then the embedding of $X$ to $L^p(0,T;B)$
is compact.
\end{theorem}
\begin{proof}
See \cite[p. 86, Corollary 5]{Simon1986} or \cite[Theorem 2.1]{franco}.
\end{proof}

\section{On the Burkholder-Davis-Gundy inequality}\label{sec:BDG}

We collect the exact form of the Burkholder-Davis-Gundy inequality needed here. Given a cylindrical Wiener process $W$ on $H^\delta_2(\CO)$,
$\delta>1$, over $\MA=(\Omega,\Fcal,\mathbb{F},\PP)$, and a progressively measurable process $\xi\in \mathcal{M}^2_\MA(H^\rho_2(\CO))$,
$\rho\in[0,\frac 12]$,
let us define $\{Y(t):t\in[0,T]\}$ by
 $$
 Y(t):=\int_0^ t \xi(s)\, dW(s), \quad t\in[0,T].
 $$
Here, for each $t\in[0,T]$, $\xi(t)$ is interpreted as a multiplication operator acting on the
elements of $H^\delta_2(\Ocal)$, namely, $\xi:H_2^\delta(\Ocal)\ni\psi \mapsto \xi\psi\in \CS'(\Ocal)$.
Since for any $\nu>\frac 12$ and for any~$\varphi\in H^\nu_2(\CO)$
the product~$\xi(t)\varphi$ belongs to~$H^\rho_2(\CO)$ by Proposition \ref{prop:runst190}, we can view~$\xi(t)$ as a linear map from~$H^\nu_2(\CO)$ into~$H^\rho_2(\CO)$.
It is shown in Proposition \ref{prop:runst190} that
\begin{equation}\label{equ:runst}
\vert\xi(t)\varphi\vert_{H^\rho_2}
\le
C
\vert\xi(t)\vert_{H^\rho_2}
\vert\varphi\vert_{H^\nu_2}
\end{equation}
where the constant $C>0$  is independent of~$\varphi$.
Consequently, for any $p\ge 1$, $\delta > 1 $, and any $\rho\in[0,\frac 12]$
\begin{equation*}\label{BDG}
\EE \left[\sup_{t \in [0, T]} \vert Y(t)\vert^p_{H^\rho_2}\right] \leq C \, \EE \,\left[ \int_0^T \vert\xi(t)\vert^2_{L_\text{HS}(H^{\delta}_2, H^\rho_2 )}\, dt\right]^\frac p2,
\end{equation*}
where $\vert\cdot\vert_{L_\text{HS}(H^{\delta}_2,H^{\rho}_2)}$ denotes the Hilbert--Schmidt norm from
$H^{\delta}_2(\CO)$ to $H^{\rho}_2(\CO)$.
First, let us note that for $\delta>1$ there exists a $\nu>\frac 12 $ such that the embedding
$H^{\delta}_2(\CO)\hookrightarrow H^\nu_2(\CO)$ is a Hilbert--Schmidt operator.
Using the fact that  $\{\psi_k^{(\delta)}:k\in\mathbb{Z} \}$ is an orthonormal basis of $H^{\delta}_2(\CO)$,
and $\psi_k^{(\delta)}=\lambda_k \psi_k$ we obtain, by using~\eqref{equ:runst},
\begin{align}\label{e lhs}
\vert\xi(s)\vert^2_{L_\text{HS}(H^{\delta}_2, H^\rho_2)}
&
=\sum_{k \in \mathbb{Z} } \left\vert\xi(s)\psi_k^{(\delta)}\right\vert^2_{ H^\rho_2}
\le C\sum_{k \in \mathbb{Z} } \left\vert\xi(s)\right\vert_{H^\rho_2}^2\left\vert\psi_k^{(\delta)}\right\vert^2_{H^{\nu}_2}\\
&=
C \vert\xi(s)\vert^2_{ H^\rho_2}  \sum_{k \in \mathbb{Z} }\vert\psi_k^{(\delta)}\vert^2_{ H^\nu_2}
.\nonumber
\end{align}
If the embedding $H^{\delta}_2(\CO) \hookrightarrow H^{\rho}_2(\CO)$ is supposed to be  a Hilbert--Schmidt,
the right hand side of \eqref{e lhs} is finite
and we obtain
\begin{equation}\label{HSnorm}
\EE \left[\sup_{t \in [0, T]} \vert Y(t)\vert^p_{H^\rho_2}\right] \leq C \, \EE \,\left[ \int_0^T
\vert\xi(t)\vert^2_{H^\rho_2}\, dt \right]^\frac p2.
\end{equation}
In case $\rho\ge \frac12$, we use the fact that $H^\rho_2(\CO)$ is an algebra and obtain
\begin{align}\label{eq lhs fin}
\vert\xi(s)\vert^2_{L_\text{HS}(H^{\delta}_2, H^\rho_2)}
&
=
\sum_{k \in \mathbb{Z} } \left\vert\xi(s)\psi_k^{(\delta)}\right\vert^2_{ H^\rho_2}
=\vert\xi(s)\vert^2_{ H^\rho_2} \sum_{k \in \mathbb{Z} } \left\vert\psi_k^{(\delta)}\right\vert^2_{ H^\rho_2}.
\end{align}
If the embedding  $H^\delta_2(\CO)\hookrightarrow H^{\rho}_2(\CO)$ is a Hilbert--Schmidt operator,
then the right hand side of \eqref{eq lhs fin} is finite
and we obtain for $\delta>\rho+1$
\begin{equation}\label{HSnorm.1}
\EE \left[\sup_{t \in [0, T]} \vert Y(t)\vert^p_{H^\rho_2}\right] \leq C \, \EE \,\left[ \int_0^T
\vert\xi(t)\vert^2_{H^\rho_2}\, dt \right]^\frac p2.
\end{equation}

\end{document}